%% file: Atiyah_Stacks.tex
\documentclass{amsart}

\usepackage{amsmath}
\usepackage{amsthm}
\usepackage{amssymb}
\usepackage{comment}
\usepackage{hyperref}
\usepackage{mathtools}
\usepackage{bbm}
\usepackage{xcolor}
\usepackage{stmaryrd}
\usepackage{enumitem}

\usepackage{tikz-cd}
\usepackage[arrow,curve,matrix,arc,2cell]{xy}
\UseAllTwocells
\usepackage[disable]{todonotes}
\usepackage{slashed}

\makeatletter
\def\paragraph{\@startsection{paragraph}{4}%
	\z@\z@{-\fontdimen2\font}%
	{\normalfont\bfseries}}
\makeatother

\def\Z{\mathbb{Z}}

\DeclareMathOperator{\At}{At}
\DeclareMathOperator{\Hom}{Hom}
\DeclareMathOperator{\id}{id}

\DeclareMathOperator{\Ext}{Ext}
\DeclareMathOperator{\Ker}{Ker}
\DeclareMathOperator{\Mod}{Mod}

\DeclareMathOperator{\Spec}{Spec}

\DeclareMathOperator{\at}{at}

\newcommand{\Sheafhom}{\mathcal{H}om}

\DeclareMathOperator{\coneop}{Cone}

\newcommand{\eqtopos}[1]{{#1}_{\parallel}}

\DeclareGraphicsExtensions{ext-list}

\theoremstyle{plain}
\newtheorem{theorem}{Theorem}[section]
\newtheorem{proposition}[theorem]{Proposition}

\newtheorem{lemma}[theorem]{Lemma}
\newtheorem{assumption}[theorem]{Assumption}
\newtheorem{corollary}[theorem]{Corollary}

\theoremstyle{definition}
\newtheorem{definition}[theorem]{Definition}
\newtheorem{claim}[theorem]{Claim}
\newtheorem{remark}[theorem]{Remark}
\newtheorem{construction}[theorem]{Construction}

\newtheorem{example}[theorem]{Example}

\newtheorem{situation}[theorem]{Situation}
\newtheorem{variant}[theorem]{Variant}

\makeatletter
\@addtoreset{equation}{section}
\makeatother

\begin{document}	
	\title{The Atiyah class on algebraic stacks} 
	\author{Nikolas Kuhn}
	\thanks{The author is supported by Research Council of
		Norway grant number 302277 - ”Orthogonal gauge duality and non-commutative
		geometry”.}
	\begin{abstract}
		We generalize Illusie's definition of the Atiyah class to complexes with quasi-coherent cohomology on arbitrary algebraic stacks. We show that this gives a global obstruction theory for moduli stacks of complexes in algebraic geometry without derived methods. We give a similar generalization of the reduced Atiyah class, and show various useful properties for working with Atiyah classes, such as compatibilities between the reduced and ordinary Atiyah class, and compatibility with tensor products and determinants.
	\end{abstract}
	\maketitle
	
	\input{main}
	
	\bibliography{references}
	\bibliographystyle{amsalpha}
\end{document}

%% file: main.tex
It is classical fact that the deformation-obstruction theory of a coherent sheaf $E$ on a projective variety $X$ is governed by the groups $\Ext^i_X(E,E)$ for $i=1,2$ \cite[Section 7]{HartDef}. The Atiyah class, as defined in \cite{Ill} in the algebraic setting, globalizes and generalizes this correspondence -- for example it measures how much a family of sheaves varies over a base. 
When working over a moduli space these self-$\Ext$ groups often give an important extra structure in the form of an \emph{obstruction theory} (often with additional properties), which is the foundational ingredient in enumerative sheaf theories. Examples are the famous Donaldson--Thomas theory \cite{Tho}, PT-theory \cite{PT}, \cite{HuTho}, moduli spaces of stable sheaves on surfaces \cite{Moch}, and more recently CY4 theory \cite{OhTh}. In each of these cases, the technical tool used to obtain the obstruction theory is the Atiyah class. 
 
In many cases, particularly for wall-crossing arguments as in \cite{Moch}, \cite{Joyce_En}, \cite{KuTa}, it is necessary to consider moduli stacks that include properly semi-stable objects which may have positive-dimensional stabilizers. To work with these, one would like an obstruction theory \emph{on the stack}. This is constructed in \cite{Moch} for moduli stacks of objects with a two-term resolution, and in \cite{Joyce_En} by using derived moduli stacks of perfect complexes. 

The main result of this paper is to generalize Illusie's construction of the Atiyah class to algebraic stacks, using Olsson's definition of the cotangent complex \cite{Olso1}. We also treat some variants, such as Gillam's reduced Atiyah class \cite{Gill}, and a version of the Atiyah class for exact sequences. We then show various compatibility properties for the Atiyah class and its variants. 

As a second main result, we show that the Atiyah class indeed gives an obstruction theory for moduli stacks of perfect complexes (Theorem \ref{prop:atisob}). The main new part here is that it captures the infinitesimal automorphisms. For simplicity, we only treat the absolute case of a proper scheme over a field, although the relative case follows along the same lines.  

\paragraph{Relation to existing work.}
Throughout, we build on the constructions and results of \cite{HuTho}, \cite{Ill} and \cite{Gill}, which we generalize to algebraic stacks. Our construction also recovers the $G$-equivariant Atiyah class constructed by Ricolfi \cite{Rico}, although note that the cotangent complex of a scheme with $G$-action with its $G$-equivariant structure, does not agree with the cotangent complex of the quotient stack, so our construction is strictly stronger even in that case. Throughout, we use Olsson's definition of the quasi-coherent derived category and the cotangent complex for algebraic stacks developed in \cite{Olso1} and \cite{LasOl}.
\paragraph{Notations and Conventions.}
For a Grothendieck topos $T$ and a ring $R$ in $T$, we let $\Mod(R)$ denote the category of $R$-modules and $D(R)$ its derived category. We identify $\Mod(R)$ with the subcategory of $D(R)$ generated by complexes concentrated in degree zero.

We will use the following convention regarding shift functors: For any complex $E$ of $R$-modules, we have a natural isomorphism 
\[E[1]=\Z[1]\otimes_{\Z} E.\]
If $E,F$ are complexes, then by $E\otimes F[1]$ we mean $(E\otimes F)[1]$ rather than $E\otimes (F[1])$. 
All tensor products, pullbacks and duals will be in the derived sense unless remarked otherwise. 

\section{Statements of results}\label{}
In this section, we introduce the Atiyah class on an algebraic stack and several variants, and we state their basic properties and mutual relations. The proofs will be given in the following sections.

\subsection{The Atiyah class}\label{subsec:theatiyahclass}
Let $f:\mathcal{X}\to \mathcal{Y}$ be a morphism of algebraic stacks and let $E\in D^-_{qcoh}(\mathcal{X})$. The Atiyah class of $E$ over $\mathcal{Y}$, defined in \ref{sec:constr-at}, is a natural map
\[\at_E:E\to L_{\mathcal{X}/\mathcal{Y}}\otimes E[1].\]
We also use the notation $\at_{E,\mathcal{X}/\mathcal{Y}}$ to emphasize the dependence on $f$. 
If $E$ is dualizable in the derived category (equivalently, a perfect complex, see \cite[Lemma 4.3]{HaRy}), then the data of $\at_E$ is equivalent to that of a map
\[\at'_E:E\otimes E^{\vee}[-1]\to L_{\mathcal{X}/\mathcal{Y}},\]
which we also call the Atiyah class.

The Atiyah class has the following fundamental properties, which we address in \S \ref{sec:properties}: Let $F,E$ be objects of $D^-_{qcoh}(\mathcal{X})$.

\paragraph{Functoriality.} Given a map $F\to E$ in $D^-_{qcoh}(\mathcal{X})$, the induced diagram 
\begin{equation*}
	\begin{tikzcd}
		F\ar[r,"\at_F"]\ar[d]& L_{\mathcal{X}/\mathcal{Y}}\otimes F[1]\ar[d] \\
		E\ar[r,"\at_E"] & L_{\mathcal{X}/\mathcal{Y}}\otimes E [1]
	\end{tikzcd}
\end{equation*}
commutes.

\paragraph{Pullback.} Given another morphism $f':\mathcal{X}'\to \mathcal{Y}'$ together with maps $A:\mathcal{X}'\to \mathcal{X}$ and $B:\mathcal{Y}'\to \mathcal{Y}$, and a $2$-isomorphism $B\circ f'\Rightarrow f\circ A$, the induced diagram 

\begin{equation*}
	\begin{tikzcd}
		LA^*E\ar[r,"LA^*\at_E"]\ar[d]& LA^*L_{\mathcal{X}/\mathcal{Y}}\otimes LA^*E[1]\ar[d] \\
		LA^*E\ar[r,"\at_{LA^*E}"] & L_{\mathcal{X}'/\mathcal{Y}'}\otimes LA^*E[1].
	\end{tikzcd}
\end{equation*}
commutes. If $E$ is perfect, then equivalently, the diagram
\begin{equation*}
	\begin{tikzcd}
		LA^*E\otimes LA^*E^{\vee}[-1]\ar[r,"LA^*\at'_E"]\ar[dr,"\at'_{LA^*E}"']& LA^*L_{\mathcal{X}/\mathcal{Y}}\ar[d] \\
		&L_{\mathcal{X}'/\mathcal{Y}'}
	\end{tikzcd}
\end{equation*} 
commutes.
\paragraph{Tensor products.} Identify $E\otimes L_{\mathcal{X}/\mathcal{Y}}[1]\otimes F\simeq L_{\mathcal{X}/\mathcal{Y}}[1]\otimes E\otimes F$ using the standard symmetry isomorphism of the derived tensor product. Then, up to this identification, we have an equality
\[\at_{E\otimes F} = \at_E\otimes F +E\otimes {\at_F}.\]
As a special case of this, if $E$ and $F$ are perfect and $\at_F$ is trivial (e.g. if $F$ is pulled back from $\mathcal{Y}$), then the following diagram commutes:
\begin{equation*}
	\begin{tikzcd}
		E\otimes E^{\vee}[-1]\ar[r,"\at'_E"]\ar[d]& L_{\mathcal{X}/\mathcal{Y}}\ar[d,equals] \\
		E\otimes F\otimes E^{\vee}\otimes F^{\vee}[-1]\ar[r,"\at'_{E\otimes F}"] & L_{\mathcal{X}/\mathcal{Y}}.
	\end{tikzcd}
\end{equation*} 
Here the left vertical map is induced by the diagonal map $\mathcal{O}_X\to F\otimes F^{\vee}$and the symmetry isomorphisms of the tensor product.
\paragraph{Determinants.}
Suppose that $E$ is perfect and consider the natural trace map $\operatorname{tr}:\Hom(E,L_{\mathcal{X}/\mathcal{Y}}[1]\otimes E)\to \Hom(\mathcal{O}_{\mathcal{X}}, L_{\mathcal{X}/\mathcal{Y}}[1])$. Then we have $\at_{\det(E)} = \operatorname{tr}(\at_E)\otimes \det(E)$ as morphisms $\det E\to L_{\mathcal{X}/\mathcal{Y}}\otimes \det E[1]$, at least when $E$ has a global finite length resolution by locally free sheaves. In particular, if the latter condition holds, the following diagram commutes
\begin{equation*}
	\begin{tikzcd}
		\mathcal{O}_X[-1]\ar[dr,"\at'_{\det E}"]\ar[d]& \\
		E\otimes E^{\vee}[-1]\ar[r,"\at'_E"] & L_{\mathcal{X}/\mathcal{Y}},
	\end{tikzcd}
\end{equation*}
where the left horizontal map is induced by the natural diagonal map $\mathcal{O}_X\to E\otimes E^{\vee}$. 
\paragraph{Pushforward.} 
Consider a cartesian diagram
\begin{equation*}
	\begin{tikzcd}
		\mathcal{X}'\ar[r]\ar[d,"p"]& \mathcal{Y}'\ar[d] \\
		\mathcal{X}\ar[r] &\mathcal{Y}
	\end{tikzcd}
\end{equation*}
and suppose that $\mathcal{X}'\to \mathcal{X}$ is concentrated \cite[Definition 2.4]{HaRy}. Let $E\in D^-_{qcoh}(\mathcal{X}')$. Then we have a commutative diagram
\begin{equation}\label{eq:pushforwarddiag}
	\begin{tikzcd}[column sep=huge]
		Rp_*E\ar[r,"\at_{Rp_*E}"]\ar[d]& L_{\mathcal{X}/\mathcal{Y}}[1]\otimes Rp_*E\ar[d] \\
		Rp_*E\ar[r,"Rp_*(\at_E)"] & Rp_*(L_{\mathcal{X}' /\mathcal{Y}'}[1]\otimes E).
	\end{tikzcd}
\end{equation} 
Suppose that, moreover, the diagram is Cartesian and that the morphisms $\mathcal{X}\to \mathcal{Y}$ and $\mathcal{Y}'\to \mathcal{X}$ are Tor-independent, and that $E$ is a perfect complex. Then the right vertical morphism in \eqref{eq:pushforwarddiag} is an isomorphism and this gives a natural identification $Rp_*(\at_E)=\at_{Rp_*E}$ as morphisms $Rp_*E\to L_{\mathcal{X}/\mathcal{Y}}[1]\otimes Rp_*E$. If moreover $Rp_*E$ is perfect, this can be restated as commutativity of the following diagram 
\begin{equation*}
	\begin{tikzcd}
		Rp_*E\otimes (Rp_*E)^{\vee}[-1]\ar[d]\ar[dr]&  \\
		\left(Rp_*(E\otimes E^{\vee})\right)^{\vee}[-1]\ar[r] &L_{\mathcal{X}/\mathcal{Y}}.
	\end{tikzcd}
\end{equation*}
\subsection{The reduced Atiyah class}
Let $f:\mathcal{X}\to \mathcal{Y}$ be a map of algebraic stacks and let $E\in D^-_{qcoh}(\mathcal{Y})$.  Let  
\[ F\to f^*E \to G\xrightarrow{+1}\]
be an exact triangle in $D^-_{qcoh}(\mathcal{X})$ such that $R\Hom^{-1}(F,G)=0$. For example, this applies if $E$ is a sheaf and $G$ is a quotient of the ordinary pullback of $E$ as a quasi-coherent sheaf.

Then the \emph{reduced Atiyah class} associated to this data is a natural map
\[\overline{\at}_{E,\mathcal{X}/\mathcal{Y},G}:F\to L_{\mathcal{X}/\mathcal{Y}}\otimes G.\]

If $G$ is dualizable, this corresponds to a map
\[\overline{\at}_{E,\mathcal{X}/\mathcal{Y},G}':F\otimes G^{\vee}\to L_{\mathcal{X}/\mathcal{Y}}.\]
We also write $\overline{\at}_{E}$ and $\overline{\at}'_{E}$ if the rest of the data is understood. 
\begin{proposition}\label{prop:redatcomphard}
	Assume that $E, F$ and $G$ are dualizable. 
	We have the following compatibility between the reduced Atiyah class and the ordinary Atiyah class of $f^*E$: The diagram
	\begin{equation*}
		\begin{tikzcd}
			F\otimes G^{\vee}\ar[r]\ar[d]& f^*E\otimes f^*E^{\vee}\ar[d] \\
			L_{X/Y}\ar[r] & f^*L_{Y}[1]
		\end{tikzcd}
	\end{equation*}
	anti-commutes.  
\end{proposition}

\subsection{Atiyah class of an exact sequence}\label{subsec:atexseq}
Let $\mathcal{X}\to \mathcal{Y}$ be a map of algebraic stacks and let 
\[\underline{E}:= \left[0\to F\to E\to G\to 0\right]\]
be an exact sequence of bounded above complexes of $\mathcal{O}_{\mathcal{X}}$ modules with quasi-coherent cohomology sheaves. Assume that the images of $F,E,G$ in the derived category of $\mathcal{X}$ are perfect complexes with bounded amplitude. 
Then there is a canonical way to complete the natural map $F\otimes G^{\vee}\to E\otimes E^{\vee}$ in $D(\mathcal{X})$ to a triangle 
\begin{equation}\label{eq:quotienttriang}
F\otimes G^{\vee}\to E\otimes E^{\vee}\to E\otimes E^{\vee}/F\otimes G^{\vee}\xrightarrow{+1}.
\end{equation}
Moreover, there exists a natural mophism
\[\at_{\underline{E}}:\frac{E\otimes E^{\vee}}{F\otimes G^{\vee}}[-1]\to L_{\mathcal{X}/\mathcal{Y}},\]
which we call the \emph{Atiyah class of the exact sequence $\underline{E}$}.

\begin{proposition}\label{prop:atexseqcommdiag}
We have a natural commutative diagram 
\begin{equation*}
	\begin{tikzcd}
		F\otimes F^{\vee}[-1]\ar[r]\ar[dr,"\at_F"']& \frac{E\otimes E^{\vee}}{F\otimes G^{\vee}}[-1]\ar[d,"\at_{\underline{E}}"] \\
		 & L_{\mathcal{X}/\mathcal{Y}}
	\end{tikzcd}
\end{equation*}
where the horizontal map is the morphism 
\[F\otimes F^{\vee}\simeq \frac{F\otimes E^{\vee}}{F\otimes G^{\vee}}\hookrightarrow \frac{E\otimes E^{\vee}}{F\otimes G^{\vee}}.\]
\end{proposition}

\begin{proposition}\label{prop:atexseqcomp}
	Let $\mathcal{X}\xrightarrow{f}\mathcal{Y} \to \mathcal{Z}$ be maps of algebraic stacks with $f$ flat and let $E$ be a bounded above complex of $\mathcal{O}_{\mathcal{Y}}$-modules with quasi-coherent cohomology. Let $E_{\mathcal{X}}:=f^*E$ and suppose we are given an exact sequence $\underline{E}_{\mathcal{X}}$ of the form 
	\[0\to F\to E_X\to G\to 0.\]
	Then we have a natural morphism of distinguished triangles
	\begin{equation*}
		\begin{tikzcd}
			E_X\otimes E_X^{\vee}[-1] \ar[r]\ar[d]&\frac{E_X\otimes E_X^{\vee}}{F\otimes G^{\vee}}[-1]\ar[r]\ar[d]& F\otimes G^{\vee}\ar[d]\ar[r, "{+1}"]&\, \\
			f^*L_{Y/Z}\ar[r]&L_{X/Z}\ar[r] & L_{X/Y}\ar[r,"{+1}"]&\,.
		\end{tikzcd}
	\end{equation*}
Here, the morphisms in the upper row are the shifts of the ones arising in the triangle \eqref{eq:quotienttriang}, except for the last one which is \emph{minus} the one arising there. 
\end{proposition}

\subsection{Deformation theoretic properties}
We present two important examples of how the Atiyah class can be used to construct obstruction theories. For simplicity, we work over a base field $k$ and let $X$ be a smooth and proper scheme of dimension $d$ over $k$.
\paragraph{Obstruction theory on moduli spaces of sheaves.}
Let $\mathcal{M}$ be a stack over $\Spec k$ and let $E\in D^-_{qcoh}(X\times \mathcal{M})$ be perfect. Consider the Atiyah class map of $E$ relative to $X$: 
\[\at_{E}:E\to L_{X\times \mathcal{M}/X}\otimes E[1]\simeq \pi_{\mathcal{M}}^*L_{\mathcal{M}}\otimes E[1].\]
Since $E$ is dualizable, and by the projection formula, this data is equivalent to a map
\[\mathcal{O}_{\mathcal{M}}\to R\pi_*(\pi_{\mathcal{M}}^*L_{\mathcal{M}}\otimes (E\otimes E^{\vee}))[1] \simeq L_{\mathcal{M}}\otimes R\pi_*(E\otimes E^{\vee})[1].\]

Using dualizability again, we obtain a morphism
\begin{equation}\label{eq:atclobth}
	\At_E:R\pi_{\mathcal{M}*}(E\otimes E^{\vee})^{\vee}[-1]\to L_{\mathcal{M}}.
\end{equation}

\begin{theorem}\label{prop:atisob}
	Suppose that $\mathcal{M}$ is an open substack of the moduli stack of coherent sheaves on $X$. Then $\At_E$ is an obstruction theory. More generally, this holds when $\mathcal{M}$ is an open substack of a moduli space of universally gluable perfect complexes on $X$. 
\end{theorem}	
This is proven in \S \S \ref{subsec:obstrcomp}--\ref{subsec:autcomp} for moduli of sheaves. The statement for complexes is addressed in Remark \ref{rem:complexes}.

\paragraph{Obstruction theory on Quot-schemes.}
Let $\mathcal{Y}$ be an algebraic stack and let $E$ be a $\mathcal{Y}$-flat coherent sheaf on $X\times \mathcal{Y}$. Let $f:\mathcal{Q}\to \mathcal{Y}$ be an open substack of the relative Quot-scheme of $E$ over $\mathcal{Y}$ and let 
\[0\to F\to E_{\mathcal{Q}}\to G\to 0\]
be the universal exact sequence on $X\times \mathcal{Q}$. We consider the associated reduced Atiyah class $\overline{\at}_{E}:=\overline{\at}_{E,V\times \mathcal{X}/V\times \mathcal{Y},G}$ as a map $\overline{\at}_E:F\to \pi_\mathcal{X}^* L_{\mathcal{X}/\mathcal{Y}}\otimes G$ in the derived category. As before, this data is equivalent to a morphism 
\[\overline{\At}_E: Rf_{*}(G\otimes F^{\vee})^{\vee}\to L_{\mathcal{Q}/\mathcal{Y}}\]

\begin{proposition}\label{prop:redatisob}
	The map $\overline{\At}_E$ is a relative obstruction theory for $f:\mathcal{Q}\to \mathcal{Y}$.
\end{proposition}

\section{Preliminaries}
We will use the following notation: If $R$ is a ring in a topos $T$, then $C(R)$ denotes the category of complexes of $R$-modules, while $C_{-1,0}(R)$ denotes the full subcategory of complexes concentrated in degrees $-1,0$ (which is canonically identified with the category of maps of $R$-modules). 
\subsection{Derived category and cotangent complex of an algebraic stack}
Let $\mathcal{X}$ be an algebraic stack and let $\mathcal{O}_{\mathcal{X}}$ denote its structure sheaf in the lisse-\'etale topos on $\mathcal{X}$. Given a smooth cover $X\to \mathcal{X}$, where $X$ is an algebraic space, one can form the strictly simplicial algebraic space $X_{\bullet} =  X_{\bullet, \operatorname{et}}$, which we consider as a strictly simplicial topos with respect to the \'etale topology on every component. We write $X_{\bullet, \operatorname{lis-et}}$ for the strictly simplicial topos obtained by taking the corresponding lisse-\'etale topos in place of each $X_n$. It is shown in \cite[4.6]{Olso1}, we have flat morphisms of topoi
\[X_{\bullet,\operatorname{et}}\xleftarrow{\epsilon} X_{\bullet, \operatorname{lis-et}}\xrightarrow{\pi} \mathcal{X}_{\operatorname{lis-et}}.\]
Moreover, the functor $\epsilon_*$ is exact and preserves flatness. We define $\eta_X^*$ to be the composition $\epsilon_*\circ \pi^*:\operatorname{Mod}(\mathcal{O}_{\mathcal{X}})\to \operatorname{Mod}(\mathcal{O}_{X_{\bullet}})$. It defines a functor on the categories of chain complexes and due to exactness also on the derived categories, both of which we also denote by $\eta_X^*$. Let $\eta_{X_*}:=R\pi_*\circ\epsilon^*:D(\mathcal{O}_{X_{\bullet}})\to D(\mathcal{O}_{\mathcal{X}})$.
By \cite[Example 2.2.5]{LasOl}, the functors $\eta_X^*$ and $\eta_{X_*}$ restrict to mutually inverse equivalences on the derived categories with quasi-coherent cohomology sheaves
\begin{align}
	\eta_X^*:D_{qcoh}(\mathcal{O}_{\mathcal{X}})  & \to D_{qcoh}(\mathcal{O}_{X_{\bullet}}),\\
	\eta_{X*}:D_{qcoh}(\mathcal{O}_{X_{\bullet}})  & \to D_{qcoh}(\mathcal{O}_{\mathcal{X}}). 
\end{align}

Given a smooth surjective map of algebraic spaces $g:W\to X$ over $\mathcal{X}$, let $W_{\bullet}$ be the induced hypercover for the map $W\to \mathcal{X}$ and let $g_{\bullet}:W_{\bullet}\to X_{\bullet}$ be the map induced by $g$. Then there is a canonical natural isomorphism between the functors $g^*\eta_X^*$ and $\eta_W^*$ on the levels of sheaves, which induces isomorphisms between the induced functors on complexes and derived categories respectively.
We recall the notion of the cotangent complex of algebraic stacks as given in \cite[\S 8]{Olso1}:
Given a morphism of algebraic stacks $f:\mathcal{X}\to \mathcal{Y}$, choose a $2$-commutative diagram 
\begin{equation}\label{diag:2comm}
	\begin{tikzcd}
		X\ar[r]\ar[d]& \mathcal{X}\ar[d] \\
		Y\ar[r] & \mathcal{Y},
	\end{tikzcd}
\end{equation}
where $X,Y$ are algebraic spaces and where the maps $Y\to \mathcal{Y}$ and $X\to\mathcal{X}_Y:= Y\times_{\mathcal{Y}} \mathcal{X}$ are smooth and surjective. Let $X_{\bullet}$ and $Y_{\bullet}$ be the strictly simplicial algebraic spaces associated to $X\to \mathcal{X}$ and $Y\to \mathcal{Y}$ respectively. One defines a complex $L_{\mathcal{X}/\mathcal{Y},X/Y}$ on $X_{\bullet}$ whose restriction to $X_n$ is given by the complex
\[L_{X_n/Y_n}\to \Omega_{X/\mathcal{X}_{Y_n}},\]
where $\Omega_{X/\mathcal{X}_{Y_n}}$ is placed in degree one, and the map is induced from the natural map of differentials $h^0(L_{X_n/Y_n})\simeq \Omega_{X_n/Y_n}\to \Omega_{X_n/\mathcal{X}_{Y_n}}$.
It is shown in \cite{Olso1}, that this defines an element of $D_{qcoh}^{\leq 1}(X_{\bullet})$, and that the element $\eta_* L_{\mathcal{X}/\mathcal{Y},X/Y}\in D^{\leq 1}_{qcoh}(\mathcal{X})$ is independent of the choice of diagram \eqref{diag:2comm} up to canonical isomorphisms. This is used to define $L_{\mathcal{X}/\mathcal{Y}}$, so that one has a canonical isomorphism  $\eta^*L_{\mathcal{X}/\mathcal{Y}} \simeq L_{\mathcal{X}/\mathcal{Y},X/Y}$ for any choice of diagram \eqref{diag:2comm}. 
\begin{remark}\label{rem:cotcomp}
	Define $\Omega_{X_{\bullet}/\mathcal{X}_{Y_{\bullet}}}$ to be the $\mathcal{O}_{X_{\bullet}}$-module which on $X_n$ is given by $\Omega_{X_n/\mathcal{X}_{Y_n}}$ with the obvious pullback maps. Then we may restate
	\[L_{\mathcal{X}/\mathcal{Y},X/Y}= \operatorname{Cone}\left(L_{X_{\bullet}/Y_{\bullet}}\xrightarrow{-}  \Omega_{X_{\bullet}/\mathcal{X}_{Y_{\bullet}}}\right)[-1].\]
	Here $L_{X_{\bullet}/Y_{\bullet}}$ is the usual cotangent complex for the map of topoi $X_{\bullet}\to Y_{\bullet}$, and the map indicated by ``$-$'' is minus the natural map whose restriction to the $n$-th simplicial degree is given by the composition $L_{X_n/Y_n}\xrightarrow{\tau_{\geq 0}} \Omega_{X_n/Y_n}\to \Omega_{X_n/\mathcal{X}_{Y_n}} $.
\end{remark}

\subsection{Simplicial methods}
We recall some`notation and basic facts about simplicial rings and simplicial sheaves of modules in general topoi. For a general reference, see Illusie's book \cite{Ill}. Throughout this subsection, let $T$ denote a topos. 
\paragraph{Simplicial modules.}
For a simplicial ring $A$ in $T$, we denote by $A\operatorname{-Mod}$ the category of $A$-modules. If $A$ is an ordinary ring, we let $A\operatorname{-Mods}$ denote the category of \emph{simplicial} $A$-modules, i.e. the category of modules over $A$ regarded as a constant simplicial ring. In either case, we denote by $D^{\Delta}(A)$ the derived category obtained by localizing $A\operatorname{-Mod}$ (resp. $A\operatorname{-Mods}$) at the class of quasi-isomorphisms.

\paragraph{Dold--Kan correspondence.}
Let $A$ be an ordinary ring in $T$. The \emph{normalized chain functor} induces an equivalence of abelian categories $N:A\operatorname{-Mods}\to C^{\leq 0}(A)$, see \cite[I 1.3]{Ill}. It sends homotopic maps to homotopic maps and there are natural identifications $\pi_i(M)\simeq h^{-i}(NM)$ for a simplicial $A$-module $M$. In particular, $N$ preserves quasi-isomorphisms and induces an equivalence $N:D^{\Delta}(A)\to D^{\leq 0}(A)$.

\paragraph{Cones and distinguished triangles.}
As in \cite[I 3.2.1]{Ill}, let $\sigma$ denote the simplicial $\Z$-module satisfying $N\sigma=\Z[1]$, and let $\gamma$ be the simplicial $\Z$-module such that $N\gamma$ is the complex $\Z\to \Z$ concentrated in degrees $[-1,0]$. Let $A$ be a simplicial ring in the topos $T$, and let $E$ be an $A$-module. We write $\sigma E:=\sigma\otimes_\Z E$ and $\gamma E:=\gamma\otimes_\Z E$.  One has canonical isomorphisms $\pi_i(E) = \pi_{i+1}(\sigma E)$ for any $i\geq 0$. Note also that we have a natural exact sequence of $A$-modules
\[0\to E\xrightarrow{\iota} \gamma E\xrightarrow{q} \sigma E\to 0. \] 

 For a map $\alpha:E\to F$ of $A$-modules, we define 
\[\operatorname{Cone}^{\Delta}(\alpha):= \operatorname{Coker}(E\xrightarrow{(\iota,\alpha)} \gamma E\oplus F).\] 
We have the sequence of natural maps 
\begin{equation}\label{eq:disttri}
	E\xrightarrow{\alpha} F\to \operatorname{Cone}^{\Delta}(\alpha)\to \sigma E,
\end{equation}
and the induced maps on homotopy groups fit into a long exact sequence. 
One declares a sequence $E\to F\to G\to \sigma E$ in $D^{\Delta}(A)$ to be a \emph{distinguished triangle} if it is isomorphic in $D^{\Delta}(A)$ to a sequence of the form \eqref{eq:disttri}, see \cite[I 3.2.2]{Ill}. If $A$ is an ordinary ring in $T$, and $E$ a simplicial $A$-module, then in $D^{\leq 0}(A)$, we have natural isomorphisms $N\sigma E\simeq (NE)[1]$, and the Dold--Kan correspondence preserves the notions of distinguished triangle.

\paragraph{Derived tensor product.}
Let $A$ be a simplicial ring in $T$. The derived tensor product defines a functor $D^{\Delta}(A)\times D^{\Delta}(A)\to D^{\Delta}(A), (E,F)\mapsto E\otimes^{\ell} F$, which can be computed as follows: For any quasi-isomorphism $L\to E$ where $L$ is a flat (i.e. degreewise flat) $A$-module, there is a natural quasi-isomorphism $L\otimes_A F\to E\otimes^{\ell} F$. The analogous statement holds with a flat replacement of $F$. For fixed $E$, the functor $E\otimes ^{\ell}-:D^{\Delta}(A)\to D^{\Delta}(A)$ is naturally triangulated, and similarly for $-\otimes^{\ell} E$.  If $A$ is an ordinary ring and $E$, $F$ are simplicial $A$-modules, then we have canonical natural isomorphisms $N(E\otimes^{\ell} F) \simeq NE\otimes^{L} NF$ in $D^{\leq 0}(A)$, which are compatible with the symmetry isomorphism of the tensor product.	

Now let $P\to B$ be a morphism of $A$-algebras. Then derived tensor  product $B\otimes_{P}^{\ell} -$ induces a triangulated functor $D^{\Delta}(P)\to D^{\Delta}(B)$, which is left-adjoint to the functor $D^{\Delta}(B)\to D^{\Delta}(P)$ obtained by restriction of scalars. We have
\begin{lemma}\label{lem:dertens}
	Suppose that $P\to B$ is a quasi-isomorphism. Then the derived tensor product and restriction of scalars are mutually inverse equivalences of categories. More precisely, the natural adjunction maps $M\to B\otimes^{\ell} M$ for $M$ in $D^{\Delta}(P)$ and $B\otimes^{\ell}N_P\to N_P$ for $N$ in $D^{\Delta}(B)$ are isomorphisms. 
\end{lemma}  
\begin{proof}
This is {\cite[I Corollaire 3.3.4.6]{Ill}}.
\end{proof}

\paragraph{Simplicial resolutions.}
Let $A\to B$ be a map of ordinary rings in a topos $T$. We denote by 
\[P_A(B)\]
the \emph{standard simplicial resolution of $B$ over $A$} \cite[I 1.5]{Ill}. It is a simplicial $A$-algebra and flat over $A$ in each degree. There is a natural quasi-isomorphism $P_A(B)\to B$, where we regard $B$ as a constant simplicial $A$-algebra in $T$. 

We will use the following result. 
\begin{lemma} \label{lem:flat}
	Let 
	\begin{equation*}
		\begin{tikzcd}
			W_2\ar[r,"a"]\ar[d,"h"]& W_1\ar[d,"g"] \\
			Y_2\ar[r,"b"] & Y_1
		\end{tikzcd}
	\end{equation*}
	be a commutative diagram of locally ringed topoi with enough points. Assume that $a$ and $b$ are flat. Then the natural map $a^{-1}P_{g^{-1}\mathcal{O}_{Y_1}}(\mathcal{O}_{W_1})\to P_{h^{-1}\mathcal{O}_{Y_2}}(\mathcal{O}_{W_2})$ of simplicial sheaves of rings on $W_2$ is flat in each degree.
\end{lemma} 
\begin{proof}
	This can be checked on stalks of $W_1$. Since taking the standard simplicial resolution commutes with pullback of topoi and with filtered direct limits, we are reduced to the following setting: We have a diagram in the category of local rings
	\begin{equation*}
		\begin{tikzcd}
			B_2\ar[r,leftarrow]\ar[d,leftarrow]& B_1\ar[d,leftarrow] \\
			A_2\ar[r,leftarrow] & A_1
		\end{tikzcd}
	\end{equation*}
	with $B_2$ flat over $B_1$ and $A_2$ flat over $A_1$, and we need to show that the natural map $P_{A_1}(B_1)\to P_{A_2}(B_2)$ is degreewise flat. Denote this map by $F:P\to R$ with $n$-th part $F_n:P_n\to R_n$ for $n\geq 0$. We also define $F_{-1}: B_1\to B_2$. We show by induction on $n\geq -1$ that $F_n$ is flat and injective. The base case follows from the fact, that a flat morphism of local rings is faithfully flat and therefore injective. We observe that $F_{n+1}$ is the natural map $P_{n+1}=A_1[P_n]\to A_2[R_n]=R_{n+1}$. The induction step then follows from Lemma \ref{lem:inductionlem} (note that $A_1\to A_2$ is injective since it is a flat map of local rings).
\end{proof}

\begin{lemma}\label{lem:inductionlem}
	Consider a commutative diagram of rings 
	\begin{equation*}
		\begin{tikzcd}
			B_1\ar[r]\ar[d]& B_2\ar[d] \\
			A_1\ar[r] & A_2
		\end{tikzcd}
	\end{equation*}
	where the map $A_1\to A_2$ is flat and injective and $B_1\to B_2$ is injective. Then $A_1[B_1]\to A_2[B_2]$ is flat and injective.
\end{lemma}
\begin{proof}
	The injectivity is clear. The map $A_1[B_1]\to A_2[B_1]=A_2\otimes_{A_1} A_1[B_1]$ is a base change of a flat map, hence flat, and $A_2[B_1]\to A_2[B_2]$ is a free algebra by injectivity, hence also flat. Since a composition of flat morphisms is flat, the result follows. 
\end{proof}

\paragraph{Module of principal parts.}
Let $A\to B$ be a map of rings in $T$. The \emph{(first) module of principal parts} for the ring map $A\to B$ is given by $P_{B/A}^1:= B\otimes B/I_{\Delta}^2$, where $I_{\Delta}$ is the kernel of the multiplication map $B\otimes_AB\to B$. It is naturally a $(B,B)$-bimodule. Recall that $\Omega^1_{B/A}=I_{\Delta}/I_{\Delta}^2$, so that we have an exact sequence of $(B,B)$-bimodules called the \emph{exact sequence of principal parts}
\[0\to \Omega^1_{B/A}\to P^1_{B/A}\to B \to 0.\]
Here, for each of the outer terms, the two $B$-module structures agree. We will denote this sequence by $\underline{P}^1_{B/A}$. The map $b\mapsto b\otimes 1$ gives a splitting of this sequence for the left $B$-module structures and $b\to 1\otimes b$ gives a splitting for the right $B$-module structures.
Now let $E$ be a $B$-module.  Then we set $P^1_{B/A}(E):=P^1_{B/A}\otimes_B E$, where we take the tensor product with respect to the right $B$-module structure on $P^1_{B/A}$. Equivalently, $P_{B/A}^1(E)=B\otimes_A E/(I_{\Delta}^2 B\otimes_A E)$. Then the \emph{sequence of principal parts for $E$} is $\underline{P}^1_{B/A}(E):=\underline{P}^1_{B/A}\otimes_B E$, or explicitly:
\begin{equation*}
	0\to \Omega_{B/A}\otimes_B E\to P_{B/A}^1(E)\to E\to 0.
\end{equation*}
We usually regard this as a sequence of $B$-modules with respect to the \emph{left} $B$-module structure. Note that it is in general not split. 

\paragraph{Illusie's Atiyah class.}	 
With these ingredients, we recall Illusie's definition of Atiyah class: Let $f:X\to Y$ be a morphism of ringed topoi and let $E\in D^{\leq 0}(X)$. Let $P:= P_{f^{-1}\mathcal{O}_Y}(\mathcal{O}_X)$ be the standard simplicial resolution. By the Dold--Kan correspondence, we may regard $E$ as an object of $D^{\Delta}(\mathcal{O}_X)$ and thus of $D^{\Delta}(P)$ by restriction of scalars. We have the exact sequence of principal parts associated to $E_P$, which is an exact sequence of $P$-modules:
\[\underline{P}^1_{P/f^{-1}\mathcal{O}_Y}(E):\quad 0\to \Omega^1_{P/f^{-1}\mathcal{O}_Y}\otimes_P E_P\to P^1_{P/f^{-1}\mathcal{O}_Y}(E_P)\to E_P\to 0.\]
Note that the leftmost term here computes the derived tensor product since $\Omega_{P/f^{-1}\mathcal{O}_Y}$ is flat over $P$. Moreover, it is canonically quasi-isomorphic to the restriction of scalars of $L_{X/Y}\otimes_{\mathcal{O}_X} E$. From the sequence $\underline{P}_{P/f^{-1}\mathcal{O}_Y}(E_P)$, we obtain a morphism $E_P\to \sigma(L_{X/Y}\otimes_{\mathcal{O}_X} E)_P$. Extending scalars to $\mathcal{O}_X$, this defines a canonical morphism $E\to \sigma  L_{X/Y}\otimes_{\mathcal{O}_X} E$ in $D^{\Delta}(\mathcal{O}_X)$. The Atiyah class is the corresponding morphism
\[E\to L_{X/Y}[1]\otimes E\] 
in $D^{\leq 0}(\mathcal{O}_X)$ obtained via the Dold--Kan correspondence.

\subsection{The parallel arrow category.}
Let 
\[ W  \overset{s}{\underset{t}{\rightrightarrows}} X  \]
be a diagram of topoi. We obtain an induced topos $\eqtopos{W}$ whose objects are tuples $(A_X,A_W,s^{\sharp},t^{\sharp})$, where $A_X$ and $A_W$ are objects of $X$ and $W$ respectively, and $s^\sharp:s^{-1}A_X\to A_W$ and $t^{\sharp}:t^{-1}A_X\to A_W$ are morphisms in $W$. Giving a ring $R=(R_X,R_W,s^{\sharp},t^{\sharp})$ in $\eqtopos{W}$ is equivalent to giving rings on $X$ and $W$, and giving $s,t$ the structure of morphism of ringed topoi. Given such an $R$, we use the following notation: For an $R_X$-module $M_X$, we write $s_R^*:=s^{-1}M_X\otimes_{s^{-1}R_X} R_W$ and  $t_R^*:=t^{-1}M_X\otimes_{t^{-1}R_X} R_W$. Then an $R$-module is given by a tuple $M=(M_X,M_W, s^*, t^*)$, where $M_X$ and $M_W$ are $R_X$ and $R_W$-modules respectively, and where $s^*:s_R^*M_X\to M_W$ and $t^*:t_R^*M_X\to M_W$ are morphisms of $R_W$-modules. 

We define a functor $\coneop_R:C(R) \to C(R_W)$ by
\[\coneop_R:(M_X,M_W,s^*,t^*)\mapsto \operatorname{Cone}(s_R^*M_X\oplus t_R^*M_X\xrightarrow{-s^*\oplus t^*} M_W).\]

If $R$ is a \emph{simplicial} ring in $\eqtopos{W}$, the analogous discussion holds for $R$-modules, and we get a functor
\begin{align*}
	\coneop_R^{\Delta}: R-\operatorname{Mod}&\to R_W-\operatorname{Mod} \\
	(M_X,M_W,s^*,t^*)&\mapsto \operatorname{Cone}^{\Delta}(s_R^*M_X\oplus t_R^*M_X\xrightarrow{-s^*\oplus t^*} M_W).
\end{align*}

We have induced functors on the derived categories:

\begin{lemma}
	\begin{enumerate}[label=\roman*)]
		\item 	Let $R$ be a ring on $\eqtopos{W}$ with $s^{\sharp}:s^{-1}R_X\to R_W$ and $t^{\sharp}:t^{-1}R_X\to R_W$ flat. Then $\coneop_R$ descends to a triangulated functor of derived categories $D(R)\to D(R_W)$ (also denoted $\coneop_R$). 
		\item Let $R$ be a simplicial ring on $\eqtopos{W}$  with $s^{\sharp}:s^{-1}R_X\to R_W$ and $t^{\sharp}:t^{-1}R_X\to R_W$ flat. Then $\coneop_R^{\Delta}$ descends to a triangulated functor $D^{\Delta}(R)\to D^{\Delta}(R_W)$ (also denoted $\coneop_R^{\Delta}$).
		\item Let $R$ be an ordinary ring $\eqtopos{W}$, viewed as a constant simplicial ring. Then the two constructions in i) and ii) are compatible with the Dold--Kan correspondence, in the sense that the two functors obtained by traversing the outer edges of the diagram
		\begin{equation*}
			\begin{tikzcd}
				D^{\Delta}(R)\ar[r,"\coneop_R^{\Delta}"]\ar[d]&D^{\Delta}(R_W) \ar[d] \\
				D^{\leq 0}(R)\ar[r,"\coneop_R"] & D^{\leq 0}(R_W)
			\end{tikzcd}
		\end{equation*}
		are related by a canonical natural isomorphism.	
	\end{enumerate}
\end{lemma}
\begin{proof}
	The functor $\coneop_R:C(R)\to C(R_W)$ factors as $C(R)\to C(C_{-1,0}(R_W))\to C(R_W)$, where the first map is induced from the map $R\operatorname{-Mod}\to C_{-1,0}(R_W)$ sending $(M_X,M_W,s^*,t^*)$ to $s_R^*M_X\oplus t_R^*M_X\xrightarrow{-s^*\oplus t^*}M_W$ and the second map is given by taking the mapping cone. It is clear that the first map in the composition induces a triangulated functor $D(R)\to D(C_{-1,0}(R_W))$. 
	Therefore, it is enough to show that for any ring $S$, the map $C(C_{0,-1}(S))\to C(S)$ sending $\alpha:A_{-1}\to A_0$ to $C(\alpha)$ is canonically triangulated and preserves quasi-isomorphisms.
	A morphism $u:A\to B$ in $C(C_{0,-1}(S))$ amounts to a commutative diagram
	\begin{equation*}
		\begin{tikzcd}
			A_{-1}\ar[r,"u_{-1}"]\ar[d,"\alpha"]& B_{-1}\ar[d,"\beta"] \\
			A_{0}\ar[r,"u_0"] & B_0.
		\end{tikzcd}
	\end{equation*} 
	The triangle $A\to B\to C(u)\to A[1]$ corresponds to the following map of triangles in $D(S)$:
	\begin{equation*}
		\begin{tikzcd}
			A_{-1}\ar[r,"u_{-1}"]\ar[d,"\alpha"]& B_{-1}\ar[d,"\beta"]\ar[r]& C(u_{-1})\ar[r]\ar[d]& A_{-1}[1]\ar[d,"{\alpha[1]}"] \\
			A_{0}\ar[r,"u_0"]\ar[d] & B_0 \ar[r]\ar[d] & C(u_0)\ar[r]& A_{0}[1]\\
			C(\alpha)\ar[r,"\coneop(u)"]& C(\beta)
		\end{tikzcd}
	\end{equation*} 
	There are two obvious ways to complete the lower row  and make this into a $3\times 4 $ commutative diagram: Either by taking cones vertically, or by forming the exact triangle associated to the morphism $\coneop(u)$. Now i) follows from the statement that there is a canonical isomorphism between the resulting diagrams. One checks this by a direct calculation. The triangulated structure is given by the isomorphism $C(\alpha[1])\to C(\alpha)[1]$ that is minus the identity on the components coming from $A_{-1}$ and the identity on the components coming from $A_0$. 
	
	Part (ii) follows from an analogous argument with simplicial modules. For $S$ any simplicial ring, and $[A_{-1}\xrightarrow{\alpha} A_0]\in C_{-1,0}(S)$, we describe the triangulated structure on the mapping cone functor: Note that 
	\begin{align*}
		C^{\Delta}(\sigma \alpha) &= \operatorname{Coker}(\sigma A_{-1}\to \gamma \sigma A_{-1} \oplus \sigma A_0),\\
		\sigma C^{\Delta}(\alpha) &= \operatorname{Coker}(\sigma A_{-1}\to \sigma \gamma A_{-1}\oplus \sigma A_0).
	\end{align*}
	The isomorphism is then induced by the symmetry isomorphism of the tensor product $\gamma\sigma A_{-1}\simeq \sigma \gamma A_{-1}$.
	
	The compatibility (iii) follows from the constructions by using the basic compatibilities of the Dold--Kan correspondence, in particular that it is compatible with the symmetry isomorphisms of tensor products on the level of derived categories. 
\end{proof}

\begin{lemma}\label{lem:coneforqiso}
	Let $R$ be a ring in $\eqtopos{W}$. For any complex of $R$-modules  $E$,  we have a natural morphism $\gamma:\coneop_R(M)\to E_{W}[1]$. If the pullback maps $s^*:s_R^*E_X\to E_W$ and $t^*:t_R^*E_X\to E_W$ are quasi-isomorphisms, then so is $\gamma$. The same picture holds in the category of simplicial $\mathcal{O}_{\mathcal{X}}$ modules with the obvious modifications, and the two situations are compatible via the Dold--Kan correspondence.
\end{lemma}
\begin{proof}
	The map is defined as the composition
	\[\coneop(E)\to s_R^*E_{X}[1]\oplus t_R^*E_{X}[1]\xrightarrow{s^*\oplus 0} E_{W}[1]\]
	(The alternative choice of second map $(0\oplus t^*)$ gives the same map up to chain-homotopy). One checks directly, that this is a quasi-isomorphism when $s^*$ and $t^*$ are quasi-isomorphisms. The proofs of the remaining statements are left to the reader. 
\end{proof}	 	

We have the following result regarding tensor products:
\begin{lemma}\label{lem:tensorcone}
	\begin{enumerate}[label = \roman*)]
		\item Let $R$ be a simplicial ring on $\eqtopos{W}$ and let $L,E$ be $R$-modules. Then there is a natural map $\coneop_R^{\Delta}(L\otimes_R E)\to \coneop_R^{\Delta}(L)\otimes E_W$.
		If either of $L$ and $E$ are flat and if $s^*:s_R^*E_X\to E_W$ and $t^*:t_R^*E_X\to E_W$ are quasi-isomorphisms, then  $\coneop_R^{\Delta}(L\otimes_R E)\to \coneop_R^{\Delta}(L)\otimes E_W$ is a quasi-isomorphism. In particular, for any $E\in D^{\Delta}(R)$ we have a canonical $2$-morphism
		\begin{equation*}
			\begin{tikzcd}[row sep = large]
				D^{\Delta}(R)\ar[r,"\coneop^{\Delta}_R"]\ar[d,"-\otimes E"]&D^{\Delta}(R_W) \ar[d,"- \otimes E_W"] \\
				D^{\Delta}(R)\ar[r,"\coneop^{\Delta}_R"]\ar[ur,Rightarrow] & D^{\Delta}(R_W),
			\end{tikzcd}
		\end{equation*}
		which is an isomorphism if the pullback maps $s^*$, $t^*$ of $E$ are isomorphisms in $D^{\Delta}(R_W)$. 
		
		\label{lem:tensorcone1}
		\item The analogous statement holds if $R$ is an ordinary ring and $L,E$ are bounded above complexes of $R$-modules. \label{lem:tensorcone2}
		\item The natural isomorphisms in the derived category in i) and ii) are compatible via the Dold--Kan correspondence. \label{lem:tensorcone3}
	\end{enumerate}	 
\end{lemma}
\begin{proof}
	We only address \ref{lem:tensorcone1}. Part \ref{lem:tensorcone2} is analogous, and \ref{lem:tensorcone3} can be seen by tracing through the argument and using the compatibilites of the Dold--Kan correspondence.
	Note that for any $R$-module $F$, we have 
	\[\coneop^{\Delta}_R(F)=\operatorname{Cone}^{\Delta}(s^*_RF_X\to \operatorname{Cone}^{\Delta}(t_R^*F_X\xrightarrow{t^*} F_Z)).\]
	where the outer cone is induced by the composition of $-s^*$ with the inclusion of $F_Z$ into the cone. 
	For any commutative triangle
	\begin{equation*}
		\begin{tikzcd}
			K\ar[d,"u"]\ar[dr,"vu"]&  \\
			M\ar[r,"v"] & N
		\end{tikzcd}
	\end{equation*} 
	in $R_W-\operatorname{Mod}$, we obtain induced maps 
	\[\operatorname{Cone}^{\Delta}(u)\to \operatorname{Cone}^{\Delta}(vu)\xrightarrow{\alpha} \operatorname{Cone}^{\Delta}(v),\]
	which form the first three terms of an exact triangle. In particular, if $u$ is a quasi-isomorphism, then $\operatorname{Cone}^{\Delta}(u)$ is acyclic, so $\alpha$ is also a quasi-isomorphism.
	Applying this to the triangle
	\begin{equation*}
		\begin{tikzcd}
			t^*_RL_X \otimes_{R_W} t^*_RE_X\ar[d, "1\otimes t^*"]\ar[dr,"t^*\otimes t^*"]&  \\
			t^*_RL_X \otimes_{R_W} E_W\ar[r," t^*\otimes 1"] &  L_W \otimes_{R_W} E_W,
		\end{tikzcd}
	\end{equation*} 
	we get a natural map, 
	\[\beta:\operatorname{Cone}^{\Delta}(t^*\otimes t^*) \to \operatorname{Cone}^{\Delta}(t^*\otimes 1) =\operatorname{Cone}^{\Delta}(t^*_RL_X\to L_W)\otimes E_W.\]
	
	Similarly, from the triangle
	
	\begin{equation*}
		\begin{tikzcd}
			s^*_RL_X \otimes_{R_W} s^*_RE_X\ar[d, "- 1\otimes s^*"]\ar[dr]&  \\
			s^*_RL_X \otimes_{R_W} E_W\ar[r] &  \operatorname{Cone}^{\Delta}(t^*_RL_X\to L_W) \otimes_{R_W} E_W,
		\end{tikzcd}
	\end{equation*} 
	where the horizontal map is induced from $-s^*_R$, we get a morphism 
	\begin{gather*}
		\gamma: \operatorname{Cone}^{\Delta}\left(s^*_RL_X \otimes_{R_W} s^*_RE_X \to \operatorname{Cone}^{\Delta}(t^*_RL_X\to L_W) \otimes_{R_W} E_W\right) \\\to \coneop_R^{\Delta}(L)\otimes E_W.
	\end{gather*}
	Putting together $\gamma$ and $\beta$, we get a natural morphism
	\[\coneop_R^{\Delta}(L\otimes_R E)\to \coneop_R^{\Delta}(L)\otimes_{R_W} E_W.\]
	Under the assumption that $L$ (resp. $E$) is flat and that the pullback maps of $E$ are quasi-isomorphisms, we find that the vertical maps of the triangles defining $\beta$ and $\gamma $ are quasi-isomorphims, hence so are $\beta$ and $\gamma$. It follows that the desired map is also a quasi-isomorphism. 
\end{proof}

\begin{variant}
	Let $W_{\wedge}$ denote the topos associated to the diagram 
	\[X\xrightarrow{s} W \xleftarrow{t} X.\]
	Then everything goes through with $W_{\wedge}$ in place of $\eqtopos{W}$. Note also that we have an exact pullback functor $\operatorname{Sh}(\eqtopos{W})\to \operatorname{Sh}(W_{\wedge})$. If $\eqtopos{R}$ is a ring on $\eqtopos{W}$ and $R_{\wedge}$ its restriction to $W_{\wedge}$, then $\coneop_{\eqtopos{R}}$ factors as 
	\[D(\eqtopos{R})\to D(R_{\wedge})\xrightarrow{\coneop_{R_{\wedge}}}D(R_W).\]
	The analogous picture holds for the simplicial versions. 
\end{variant}

\paragraph{Application to algebraic stacks.}
We apply the preceding discussion to the derived categories of algebraic stacks and the cotangent complex. This lays the groundwork for our definition of the Atiyah class on a stack using the modules of principal parts. 

\begin{situation}\label{sit:wxtopoi}
	Consider the diagram \eqref{diag:2comm} with the associated strictly simplicial algebraic spaces $X_{\bullet}$ and $Y_{\bullet}$. Let $W:=X\times_{\mathcal{X_Y}} X$, with associated projections $s,t:W\to X$ on the first and second factor respectively, and let $h:W\to Y$ be the induced map. Let $W_{\bullet}$ be the strictly simplicial algebraic space associated to the covering $W\xrightarrow{s} X\to \mathcal{X}$. One has canonical isomorphisms $W_n\simeq X_n\times_{\mathcal{X}_{Y_n}}X_n$ and maps $h_n:W_n\to Y_n$.
	We denote by $\eqtopos{W}$ the topos associated to the diagram 
	\[W_{\bullet} \overset{s_{\bullet}}{\underset{t_{\bullet}}{\rightrightarrows}} X_{\bullet}.\] 
	It has a natural structure of ringed topos, and we write $\mathcal{O}_{\eqtopos{W}}$ for the structure sheaf. 
	
	We denote by $(\eqtopos{Y},\mathcal{O}_{\eqtopos{Y}})$ the ringed topos associated to the diagram
	\[Y_{\bullet}\rightrightarrows Y_{\bullet}\]
	with both arrows the identity. There is a natural map of topoi $\eqtopos{h}:\eqtopos{W}\to \eqtopos{Y}$.
	
	For $E$ a complex of $\mathcal{O}_{\mathcal{X}}$-modules, we write $E_{X_{\bullet}}:=\eta_X^*E$ and $E_{W_{\bullet}}:=\eta_W^*E$. We have natural morphisms  $s_{\bullet}^*E_{X_{\bullet}}\to E_{W_{\bullet}}$ and 
	$t_{\bullet}^*E_{X_{\bullet}}\to E_{W_{\bullet}}$. Thus we get naturally a complex of $\mathcal{O}_{\eqtopos{W}}$-modules, which we denote $E_{\eqtopos{W}}$.
\end{situation}

In Situation \ref{sit:wxtopoi}, we have the following properties:
\begin{lemma}\label{lem:conepresqcoh}
	For any a complex of $\mathcal{O}_{\mathcal{X}}$-modules  $E$,  we have a natural morphism $ \coneop(E_{\eqtopos{W}})\to E_{W_{\bullet}}[1]$. If $E$ has quasi-coherent (or more generally, Cartesian) cohomology sheaves, this map is a quasi-isomorphism. The same picture holds in the category of simplicial $\mathcal{O}_{\mathcal{X}}$ modules with the obvious modifications, and the two situations are compatible via the Dold--Kan correspondence.
\end{lemma}
\begin{proof}
	This is a restatement of Lemma \ref{lem:coneforqiso}, using that the pullback maps $s^*$ and $t^*$ are quasi-isomorphisms whenever $E$ has quasi-coherent cohomology sheaves.
\end{proof}

\begin{lemma}\label{lem:coneprescot}
	In $D_{qcoh}(W_{\bullet})$, the object $\eta_{W}^*L_{\mathcal{X}/\mathcal{Y}}$ is naturally isomorphic to 
	\[\coneop(L_{\eqtopos{W}/\eqtopos{Y}})[-1]=\operatorname{Cone} \left(s_{\bullet}^*L_{X_{\bullet}/Y_{\bullet}}\oplus t_{\bullet}^*L_{X_{\bullet}/Y_{\bullet}}\xrightarrow{-s^*_{\bullet}+t_{\bullet}^*} L_{W_{\bullet}/Y_{\bullet}}\right)[-1].\]
	This isomorphism is functorial in the diagram \eqref{diag:2comm}.
\end{lemma}
\begin{proof}
	The stated equality is immediate from the definition of $\eqtopos{W}$ and the definition of cotangent complex for a morphism of topoi. By the construction of the cotangent complex and Remark \ref{rem:cotcomp}, we have a canonical isomorphism 
	\begin{equation}\label{eq:cotcompcone}
		\eta_X^*L_{\mathcal{X}/\mathcal{Y}}\simeq L_{\mathcal{X}/\mathcal{Y},X/Y} = \operatorname{Cone}\left(L_{X_{\bullet}/Y_{\bullet}}\xrightarrow{-}  \Omega_{X_{\bullet}/\mathcal{X}_{Y_{\bullet}}}\right)[-1].
	\end{equation}
	
	We pull this isomorphism back via $s_{\bullet}$. Due to the diagram 
	\begin{equation*}
		\begin{tikzcd}
			W_{\bullet}\ar[r,"t_{\bullet}"]\ar[d,"s_{\bullet}"]& X_{\bullet}\ar[d]& \\
			X_{\bullet}\ar[r] & \mathcal{X}_{Y_{\bullet}}\ar[r]&Y_{\bullet}
		\end{tikzcd}
	\end{equation*}
	we have natural quasi-isomorphisms of complexes
	\[s_{\bullet}^*\Omega_{X_{\bullet}/\mathcal{X}_{Y_{\bullet}}} \simeq \Omega_{W_{\bullet}/X_{\bullet},t_{\bullet}}\leftarrow \operatorname{Cone}(t_{\bullet}^*L_{X_{\bullet}/Y_{\bullet}} \xrightarrow{t^*}L_{W_{\bullet}/Y_{\bullet}}).\]
	We claim that the following diagram commutes: 
	\begin{equation*}
		\begin{tikzcd}
			s_{\bullet}^*L_{X_{\bullet}/Y_{\bullet}}\ar[r,"\gamma"]\ar[d]& \operatorname{Cone}(t_{\bullet}^*L_{X_{\bullet}/Y_{\bullet}} \xrightarrow{t^*}L_{W_{\bullet}/Y_{\bullet}})\ar[d] \\
			s_{\bullet}^*\Omega_{X_{\bullet}/\mathcal{X}_{Y_{\bullet}}}\ar[r,"\sim"] & \Omega_{W_{\bullet}/X_{\bullet},t_{\bullet}},
		\end{tikzcd}
	\end{equation*}
	where the upper horizontal map is given by the pullback $s^*:s_{\bullet}^*L_{X_{\bullet}/Y_{\bullet}}\to L_{W_{\bullet}/Y_{\bullet}}$ followed by the inclusion of $L_{W_{\bullet}/Y_{\bullet}}$ into the cone.
	Indeed, this follows from the observation  that the map $s^*_{\bullet}L_{X_{\bullet}/Y_{\bullet}}\to \Omega_{W_{\bullet}/X_{\bullet},t_{\bullet}}$ factors through $L_{W_{\bullet}/Y_{\bullet}}$. It follows from this that the right hand side of \eqref{eq:cotcompcone} is naturally quasi-isomorphic to 
	\[ \operatorname{Cone}\left(s_{\bullet}^*L_{X_{\bullet}/Y_{\bullet}}\xrightarrow{-\gamma}  \operatorname{Cone}\left(t_{\bullet}^*L_{X_{\bullet}/Y_{\bullet}}\xrightarrow{t^*} L_{W_{\bullet}/Y_{\bullet}}\right)\right)[-1].\]
	An easy calculation shows that this iterated cone is identical to
	\[\operatorname{Cone} \left(s_{\bullet}^*L_{X_{\bullet}/Y_{\bullet}}\oplus t_{\bullet}^*L_{X_{\bullet}/Y_{\bullet}}\xrightarrow{-s^*_{\bullet}+t_{\bullet}^*} L_{W_{\bullet}/Y_{\bullet}}\right)[-1]\]
	as desired.
	The compatibility with pullback follows from the compatibility of \eqref{eq:cotcompcone}, and the usual pullback compatibilities of the cotangent complex for algebraic spaces.
\end{proof}

\begin{remark}\label{rem:hugecommdiag}
	In Situation \ref{sit:wxtopoi}, suppose that $R$ is a simplicial ring on $\eqtopos{W}$ with a given map $R\to \mathcal{O}_{\eqtopos{W}}$. Then we have the following diagram of functors, which commutes up to canonical natural isomorphisms 
	\begin{equation}\label{diag:hugecommdiag}
		\begin{tikzcd}
			D^{\Delta}(R)\ar[r,"\mathcal{O}_{\eqtopos{W}}\otimes^{\ell}_R -"]\ar[d,"\coneop_R^{\Delta}"]& D^{\Delta}(\mathcal{O}_{\eqtopos{W}}) \ar[r,"N"]\ar[d,"\coneop^{\Delta}"] & D^{\leq 0}(\mathcal{O}_{\eqtopos{W}})\ar[d,"\coneop "]& \\
			D^{\Delta}(R_W)\ar[r,"\mathcal{O}_{W_{\bullet}}\otimes^{\ell}_{R_W} -"] & D^{\Delta}(\mathcal{O}_{W_{\bullet}})\ar[r,"N"]& D^{\leq 0 }(\mathcal{O}_{W_{\bullet}}) \ar[r, "\eta_{W*}"]& D^{\leq 0}(\mathcal{X}).
		\end{tikzcd}
	\end{equation} 
\end{remark}	

For the following lemma, let $W_{\wedge}$ be the ringed topos associated to the diagram $X_{\bullet}\xleftarrow{s_{\bullet}} W_{\bullet}\xrightarrow{t_{\bullet}} X_{\bullet}$. Recall that we have a natural restriction functor $D(\eqtopos{W})\to D(W_{\wedge})$. The following will be used in the proof of Lemma \ref{lem:functoriality}. 
\begin{lemma}\label{lem:technicalfunctorialitylemma}
	Let $E_{\eqtopos{W}}=(E_{X_{\bullet}},E_{W_{\bullet}}, s^*,t^*)$ be a complex of $\mathcal{O}_{\eqtopos{W}}$-modules such that $E_{X_{\bullet}}$ has quasi-coherent cohomology sheaves and such that the pullback maps $s^*$ and $t^*$ are quasi-isomorphisms. Let $F:=\eta_{W*}E_{W_{\bullet}}\in D_{qcoh}(\mathcal{X})$, and let $F_{\eqtopos{W}}=(\eta_X^*F, \eta_W^*F)$ denote the induced object of $D(\eqtopos{W})$. Let $E_{W_{\wedge}}$ and $F_{W_\wedge}$ denote the images of $E_{\eqtopos{W}}$ and $F_{\eqtopos{W}}$ in $D(W_{\wedge})$ respectively. Then there is a natural isomorphism $F_{W_{\wedge}}\to E_{W_{\wedge}}$ in $D(W_{\wedge})$.   
\end{lemma}
\begin{proof}
	Let $W_{\wedge,lis-et}$ be the ringed topos associated to the diagram 
	\[X_{\bullet, lis-et}\xleftarrow{s_{\bullet}}W_{\bullet, lis-et}\xrightarrow{t_{\bullet}} X_{\bullet, lis-et}.\]
	From the diagram of ringed topoi
	
	\begin{equation*}
		\begin{tikzcd}
			X_{\bullet}&X_{\bullet, lis-et}\ar[l,"{\epsilon}"] \ar[drr,"\pi_X"]& &\\
			W_{\bullet}\ar[u,"s_{\bullet}"]\ar[d,"t_{\bullet}"']&W_{\bullet, lis-et}\ar[l,"\epsilon"]\ar[d,"t_{\bullet}"']\ar[u,"s_{\bullet}"] \ar[rr,"\pi_W"]&  &\mathcal{X}\\
			X_{\bullet}&X_{\bullet, lis-et}\ar[l,"{\epsilon}"] \ar[urr,"\pi_X"']& &
		\end{tikzcd}
	\end{equation*}
	
	we see that there is an induced morphism $\epsilon_{\wedge}: {W}_{\wedge,lis-et} \to W_{\wedge}$. Let $F_{\wedge,lis-et}=(\pi_X^*F,\pi_W^*F, \pi_X^*F)$, with pullback maps induced by the identifications $\pi_W^*=s^*\pi_X^*$ and $\pi_W^*=t^*\pi_X^*$ respectively.  Since the functors $\epsilon_*$ are exact, we have $F_{\wedge}={\epsilon}_{\wedge*}F_{\wedge, lis-et}$. In particular, since $\epsilon_*\epsilon^*$ is naturally isomorphic to the identity, it is enough to show that there is a natural quasi-isomorphism $\epsilon_{\wedge}^*E_{W_{\wedge}}\to F_{\wedge,lis-et}$.
	
	We are reduced to the following situation: Let $G=(G_{X},G_W,G_{X},s^*,t^*)$ be a complex of $\mathcal{O}_{W_{\wedge, lis-et}}$-modules such that $G_X$ has quasicoherent cohomology sheaves, and such that $s^*$,$t^*$ are quasi-isomorphisms, and suppose that $F=R(\pi_W)_*G_W$. Then we need to show that $G$ is naturally isomorphic to $F_{\wedge,lis-et}$.  Without loss of generality, we may assume that $G$ is $K$-injective. Then $G_W$ and $G_X$ are themselves $K$-injective. We can therefore assume that $F=\pi_{W*}G_W$.  Let also $F':=\pi_{X*}G_X$. The morphism $s^*:s^*G_X\to G_W$ corresponds to a morphism $G_X\to s_*G_W$, and by applying $\pi_{W*}$ we obtain a map $\sigma:F'\to F$. Similarly, we obtain $\tau:F'\to F$ from $t^*:t^*G_X\to G_W$. The maps $\sigma $ and $\tau$ are quasi-isomorphisms, as one can check after applying $\pi_W^*$, since $F'$ and $F$ have quasi-coherent comohology. 
	
	Then we have the following commutative diagram of complexes of $\mathcal{O}_{W_{\bullet,lis-et}}$-modules:
	\begin{equation*}
		\begin{tikzcd}
			s^*\pi_X^*F\ar[d,"s*"']	&s^*\pi_X^*F'\ar[r]\ar[d,"\pi_W^*\sigma "']\ar[l,"s^*\pi_X^*\sigma"']&s^*G_X\ar[d] \\
			\pi_W^*F	&\pi_W^*F\ar[r]\ar[l] & G_W \\
			t^*\pi_X^* F\ar[u,"t^*"]	&t^*\pi_X^* F'\ar[u,"\pi_W^*\tau"]\ar[r]\ar[l,"t^*\pi_X^*\tau"]&t^*G_X\ar[u]
		\end{tikzcd}
	\end{equation*}	
	The horizontal maps going to the right are the quasi-isomorphisms coming from the push-pull adjunctions associated to $\pi_X$ and $\pi_W$ respectively.
	Setting $\widetilde{F}_{\wedge, lis-et}= (\pi_X^*F', \pi_W^*F, \pi_X^*F', s^*\pi_X^*\sigma, t^*\pi_X^*\tau)$, it follows that we have quasi-isomorphisms 
	\[F_{\wedge, lis-et} \leftarrow\widetilde{F}_{\wedge,lis-et}\to G\]
	as desired.
\end{proof}
\subsection{Tensor triangulated categories and additivity of traces.}
\paragraph{Traces in a closed symmetric monoidal category.}
Let $\mathcal{C}$ be a symmetric closed monoidal category with product $-\otimes -$, unit $\mathcal{O}$, and internal Hom-functor $\Sheafhom(-,-)$. Let $\tau$ denote the symmetry morphism of the tensor product. We let $E,F,G$ denote arbitrary elements of $\mathcal{C}$ and write $E^{\vee}:=\Sheafhom(E,\mathcal{O})$. Recall that we have an adjunction between $-\otimes E$ and $\Sheafhom(E,-)$. We have various natural maps in $\mathcal{C}$:	

\begin{enumerate}[label=(\arabic*)]
	\item {\it Evaluation.} We have a natural map $ev: \Sheafhom(E,F)\otimes E\to F$, corresponding to the identity on $\Sheafhom(E,F)$ under adjunction.
	\item As a special case, we get the evaluation map $ev:E^{\vee}\otimes E\to \mathcal{O}$.
	\item {\it Composition.} We have a natural map $comp: \Sheafhom(F,G)\otimes \Sheafhom(E,F)\to \Sheafhom(E,G)$. This map is adjoint to the map $\Sheafhom(F,G)\otimes \Sheafhom(E,F)\otimes E\to G$ which is obtained from composing two evaluation maps. 
	\item As a special case, we get the composition map $comp:F\otimes E^{\vee}\to\Sheafhom(E,F)$, where we use the canonical isomorphism $F\simeq \Sheafhom(\mathcal{O},F)$. 
	\item {\it Diagonal.} We have a diagonal map $s:\mathcal{O}\to \Sheafhom(E,E)$ which corresponds to $\operatorname{id}_{E}$ under adjunction.    
\end{enumerate}

We recall the notion of dualizable object:
\begin{definition}
	The object $E$ is called \emph{dualizable} if the composition map $comp:E\otimes  E^{\vee}\to \Sheafhom(E,E)$ is an isomorphism. See the discussion preceding Definition 2.2 in \cite{PoSh}.
\end{definition}
For dualizable $E$ and arbitrary $F$, the canonical map $comp: F\otimes E^{\vee}\to \Sheafhom(E,F)$ is an isomorphism and we will often use this isomorphism to identify the source and target.

Let $E$ be a dualizable object of $\mathcal{C}$. Then we have further natural maps

\begin{enumerate}[label=(\arabic*)]
	\item {\it Trace I.} We define the \emph{trace map} $\operatorname{tr}:\Sheafhom({E},{E})\to \mathcal{O}$ as the composition $\Sheafhom(E,E)\simeq{E}\otimes {E}^{\vee}\xrightarrow{\tau} E^{\vee}\otimes {E}\xrightarrow{ev} \mathcal{O}$. 
	\item  {\it Trace II.} We have a map $\operatorname{tr}:\Sheafhom({E}, {F} \otimes {E})\to {F}$ given by the composition $\Sheafhom({E}, {F}\otimes {E})\simeq {F}\otimes {E}\otimes {E}^{\vee}\xrightarrow{{F}\otimes tr}{F}$. By abuse of language, we call this also the \emph{trace map}.
	\item {\it Diagonal II.} We have the composition $\eta:=comp^{-1}\circ s:\mathcal{O}\to {E}\otimes {E}^{\vee}$, which we will also call \emph{diagonal}.  
	
\end{enumerate}

A morphism $f:{E}\to {F}\otimes {E}$ corresponds by adjunction to a map $\mathcal{O}\to \Sheafhom({E},{F}\otimes {E})$.  When $E$ is dualizable, we define $\operatorname{tr}_{f}$ as the composition 
$\mathcal{O}\to \Sheafhom({E},{F}\otimes {E})\xrightarrow{\operatorname{tr}} {F}$

As a consequence of the definition and the naturality of adjunction, we have
\begin{lemma}\label{lem:traltdef}
	For $f:{E}\to {F}\otimes {E}$, the map $tr_{f}$ is equal to the composition
	\[\mathcal{O}\xrightarrow{\eta} {E}\otimes {E}^{\vee}\xrightarrow{f\otimes \operatorname{id}_{{E}^{\vee}}} {F}\otimes {E}\otimes {E}^{\vee} \xrightarrow{\operatorname{id}_{F} \otimes \tau} {F}\otimes {E}^{\vee} \otimes {E}\xrightarrow{\operatorname{id}_{{F} }\otimes ev} {F}, \]
	where $\tau$ denotes the symmetry isomorphism of the tensor product.
\end{lemma}
\begin{proof}
	The composition of the final two maps is by definition equal to $\operatorname{id}_F\otimes \operatorname{tr}$.
	Now consider the following diagram
	\begin{equation*}
		\begin{tikzcd}[column sep = large]
			&  {E}\otimes {E}^{\vee}\ar[r,"f\otimes \operatorname{id}_{E^{\vee}}"]\ar[d,"comp"]&{F}\otimes {E}\otimes {E}^{\vee} \ar[d,"comp"]\ar[dr,"{\operatorname{id}_F\otimes \operatorname{tr}}"] &\\
			\mathcal{O}\ar[r,"s"]\ar[ur,"\eta"]& \Sheafhom({E},{E})\ar[r,"{\Sheafhom({E}, f)}"] & \Sheafhom({E},{F}\otimes {E})\ar[r,"{\operatorname{tr}}"]&F.
		\end{tikzcd}
	\end{equation*}
	The left and right triangles commute by definition, while the square commutes by naturality of the composition map.
	By naturality of adjunction, the first two arrows in the lower row compose to the map adjoint to $f$. Thus, by definition, the lower row composes to $\operatorname{tr}_f$ and the same must be true for the composition along the top. 
\end{proof}

We will need the following compatibility of traces with tensor products

\begin{lemma}\label{lem:trtensor}
	Let $f:E\to F\otimes E$ be a morphism in $\mathcal{C}$ with $E$ dualizable, and let $V$ be another dualizable object of $\mathcal{C}$. Then 
	\[\operatorname{tr}(f\otimes \operatorname{id}_V) = \operatorname{tr}(f)\circ \operatorname{tr}(\operatorname{id}_V).\]
\end{lemma}
\begin{proof}
	This follows by applying Corollary 5.9 of \cite{PoSh}.
\end{proof}

\paragraph{Criterion for Additivity of traces.}
We now assume that, say, $\mathcal{C}=D(R)$ is the derived category of a category of sheaves of modules over a ring $R$ in a topos $T$. In particular, it is closed monoidal with a compatible triangulated structure. 
Let ${F}, {E}, {G}$ be dualizable objects of $D(R)$ that are part of an exact triangle
\begin{equation}
	F\to E\to G\to F[1]. \label{eq:extridualizable}
\end{equation} 
Suppose that we have a commutative diagram for some element $L\in D(R)$:
\begin{equation*}
	\begin{tikzcd}
		{F}\ar[r]\ar[d,"g"]& {E} \ar[d,"f"]\ar[r]& {G}\ar[d,"h"] \\
		{L}\otimes {F}\ar[r] &{L}\otimes {E} \ar[r]&L\otimes {G},
	\end{tikzcd}
\end{equation*}
where the lower row is obtained from the upper by tensoring with $L$.  We want a criterion that allows us to conclude that $tr(f)=tr(g)+tr(h)$. We follow the strategy in \cite[\S 8]{May}. 

We will make the following
\begin{assumption}\label{assum:bounded}
	The objects $F,E,G$ are elements of the bounded above derived category $D^-(R)$.
\end{assumption}

\begin{situation}\label{sit:exseqdualized}
	Suppose Assumption \ref{assum:bounded} holds. (This is likely unnecessary if one works with $K$-flat complexes in what follows.)
	Choose a representation of \eqref{eq:extridualizable} by a short exact sequence $0\to{F}\to {E}\to {G}\to 0$ of bounded above complexes of flat $R$-modules (or suppose one is given).
	Further, choose an injective resolution $J$ of $R$, so that the derived dual of a complex is explicitly realized by $(-)^{\vee}:=\Sheafhom_R(-,J)$. Then we have again an exact sequence $0\to G^{\vee}\to {E}^\vee\to {F}^{\vee}\to 0$. 
	We define complexes $W$ and $\overline{W}$ via 
	
	\begin{align*}
		W&:= {E}\otimes {E}^{\vee}/{F}\otimes {G}^{\vee}\\
		\overline{W}&:= {E}^{\vee}\otimes {E}/{G}^\vee \otimes {F}.
	\end{align*}
	The tensor product here is the ordinary tensor product of complexes, which represents the derived tensor product due to the choice of ${F}$ and ${E}$. We have a natural isomorphism $W\to \overline{W}$ coming from the symmetry isomorphism of the tensor product. 
	Moreover, we have natural inclusions of complexes
	\begin{align*}
		{F}^{\phantom{\vee}}\otimes {F}^{\vee}\hookrightarrow &\, W\hookleftarrow \, {G}^{\phantom{\vee}} \otimes {G}^{\vee},
	\end{align*}
	and 
	\begin{align*}
		{F}^{\vee}\otimes {F}^{\phantom{\vee}}\hookrightarrow&\, \overline{W} \hookleftarrow \, {G}^{\vee}\otimes {G}.
	\end{align*}
	The inclusions into $W$ are given by 
	\[{F}\otimes {F}^{\vee}\simeq {F}\otimes {E}^{\vee}/{F}\otimes {G}^{\vee}\subset W\] and \[{G}\otimes {G}^{\vee}\simeq {E}\otimes {G}^{\vee}/{F}\otimes {G}^{\vee}\subset W\]
	repectively, and analogously for the inclusions into $\overline{W}$.
\end{situation}

\begin{lemma}\label{lem:traddsetup}
	Let the notation be as in Situation \ref{sit:exseqdualized}. 
	\begin{enumerate}[label=(\roman*)]
		
		\item 	Let $\overline{\eta}$ be the composition $\mathcal{O}\xrightarrow{\eta} {E}\otimes {E}^{\vee}\to W$ in $D(R)$. Then the following diagram commutes
		\begin{equation*}
			\begin{tikzcd}
				&\mathcal{O}\ar[d,"\overline{\eta}"]\ar[dr,"{(\eta,\eta)}"] \ar[dl,"{\eta}"']& \\
				{E}\otimes {E}^{\vee}\ar[r]&W& {F}\otimes {F}^{\vee}\oplus {G}\otimes {G}^{\vee}\ar[l] .
			\end{tikzcd}
		\end{equation*}
		\item There exists a natural map of complexes $\overline{\operatorname{ev}}:\overline{W}\to J$ whose image in $D(R)$ makes the following diagram commute
		\begin{equation*}
			\begin{tikzcd}
				{E}^{\vee}\otimes {E}\ar[r]\ar[dr,"{\operatorname{ev}}",swap]&\overline{W}\ar[d,"{\overline{\operatorname{ev}}}"] &  ({F}^{\vee}\otimes {F})\oplus(  {G}^{\vee}\otimes {G})\ar[l]\ar[dl,"{\operatorname{ev}\oplus \operatorname{ev}}"]\\
				&\mathcal{O}.&
			\end{tikzcd}
		\end{equation*}
	\end{enumerate}
\end{lemma}

\begin{proof}
	The map in (ii) is the natural quotient map of $ev:E^{\vee}\otimes E\to J$, which vanishes on the subcomplex $G^{\vee}\otimes F$. The commutativity of the right triangle follows by unwinding the definitions. 
	\begin{enumerate}[label=(\roman*)]
		\item  We let $V:=({F}\otimes {E}^{\vee}\oplus {E}\otimes {G}^{\vee})/({F}\otimes {G}^{\vee})$, where the quotient is with respect to the antidiagonal inclusion. Then we have a natural exact sequence of complexes
		\[0\to V \xrightarrow{j_1} {E}\otimes {E}^{\vee}\to {G}\otimes {F}^{\vee}\to 0.\]
		
		Moreover, we have natural maps $j_2:V\to {F}\otimes {F}^{\vee}$ and $j_3:V\to {G}\otimes {G}^{\vee}$ obtained by composing projection onto a factor with a respective quotient map. We claim that the following diagram of complexes commutes
		\begin{equation*}
			\begin{tikzcd}
				& V\ar[dr,"{(j_2,j_3)}"]\ar[dl,"j_1"']&   \\
				{E}\otimes {E}^{\vee}\ar[r] & W& 	({F}\otimes {F}^{\vee})\oplus ({G}\otimes {G}^{\vee})	\ar[l].	
			\end{tikzcd}
		\end{equation*}
		By precomposing with the surjection ${F}\otimes {E}^{\vee}\oplus {E}\otimes {G}^{\vee}\to V$, this reduces to the commutativity of the following two diagrams
		
		\begin{equation*}
			\begin{tikzcd}[column sep=small]
				& {F}\otimes {E}^{\vee}\ar[dr]\ar[dl]& && {E}\otimes {G}^{\vee}\ar[dr]\ar[dl]&   \\
				{E}\otimes {E}^{\vee}\ar[r] & W& 	{F}\otimes {F}^{\vee}\ar[l]&{E}\otimes {E}^{\vee}\ar[r] & W& {G}\otimes {G}^{\vee}	\ar[l]
			\end{tikzcd}
		\end{equation*}
		which one can see directly. Now the problem is reduced to finding a map $\eta_V:\mathcal{O}\to V$ in the derived category such that 
		\begin{equation*}
			\begin{tikzcd}
				&\mathcal{O}\ar[d,"\eta_V"]\ar[dr,"{(\eta,\eta)}"] \ar[dl,"{\eta}"']& \\
				{E}\otimes {E}^{\vee}\ar[r,leftarrow]&V& {F}\otimes {F}^{\vee}\oplus {G}\otimes {G}^{\vee}\ar[l, rightarrow] 
			\end{tikzcd}
		\end{equation*}
		commutes. After dualizing in the derived category, this follows from (ii).
	\end{enumerate}
\end{proof}

Now we have the following result
\begin{proposition}[{\cite[\S 8]{May}}]\label{prop:maytrace}
	Suppose there exists a dotted arrow making its two adjacent squares in the following diagram commute. Then $tr(f)=tr(g)+tr(h)$.
	\begin{equation*}
		\begin{tikzcd}
			&\mathcal{O}\ar[dr,"\eta"]\ar[d,"\overline{\eta}"] \ar[dl,"{(\eta,\eta)}"']& \\
			{E}\otimes {E}^{\vee}\ar[r]\ar[d]&W\ar[d,dashed]& {F}\otimes {F}^{\vee}\oplus {G}\otimes {G}^{\vee}\ar[d]\ar[l] \\
			L\otimes {E}\otimes {E}^{\vee}\ar[d]\ar[r]&L\otimes W \ar[d]&  (L\otimes {F}\otimes {F}^{\vee})\oplus (L\otimes {G}\otimes {G}^{\vee})\ar[d]\ar[l]\\
			L\otimes {E}^{\vee}\otimes {E}\ar[r]\ar[dr,"{\operatorname{ev}}"']&L\otimes \overline{W}\ar[d,"{\overline{\operatorname{ev}}}"] &  (L\otimes {F}^{\vee}\otimes {F})\oplus( L\otimes {G}^{\vee}\otimes {G})\ar[l]\ar[dl,"{\operatorname{ev}\oplus\operatorname{ev}}"]\\
			&L.&
		\end{tikzcd}
	\end{equation*}
\end{proposition}
\begin{proof}
	Everything else in the diagram commutes due to Lemma \ref{lem:traddsetup} and the definitions. By Lemma \ref{lem:traltdef}, the composition along the left side equals $\operatorname{tr}(f)$, and the composition along the right side equals $\operatorname{tr}(g)+\operatorname{tr}(h)$. 
\end{proof}

\section{Constructions for topoi}

\subsection{Constrution of the reduced Atiyah class}\label{subsec:redatiy}
We review the construction of the reduced Atiyah class for a map of ringed topoi, following the ideas in \cite{Gill}.
\paragraph{A diagram of exact sequences.}
Let $\mathcal{A}$ be an abelian category and consider the following diagram in $\mathcal{A}$ in which the solid arrows commute.

\begin{equation}\label{diag:cornersquare}
	\begin{tikzcd}
		E \ar[r,"e"']\ar[d,"q"]&E''\ar[d,"q''"]\ar[l,dotted,"s"', bend right] \\
		G\ar[r,"f"'] & G''.
	\end{tikzcd}
\end{equation}
Assume that $s$ is a section of $e$, i.e. $e\circ s =\id_{E''}$.
Then, the composition $q\circ s$ induces a morphism $\delta:\Ker(q'')\to \Ker(f)$.

Now suppose this diagram extends to a commutative diagram in which all solid rows and columns are exact and where the dotted arrows give a spliting of the middle exact sequence
\begin{equation*}
	\begin{tikzcd}
		&0\ar[d]&0\ar[d]&0\ar[d]& \\
		0\ar[r]& F'\ar[r,"a"]\ar[d,"j'"]&F\ar[r,"d"]\ar[d,"j"]& F''\ar[d,"j''"]\ar[r] &0 \\
		0\ar[r]& E'\ar[r,"b"']\ar[d]& E \ar[r,"e"']\ar[d,"q"]\ar[l,dotted,"r"', bend right]&E''\ar[r]\ar[d,"q''"]\ar[l,dotted,"s"', bend right]&0\\
		0\ar[r]& G' \ar[r,"c"']\ar[d]&G\ar[r,"f"']\ar[d] & G''\ar[r]\ar[d]& 0.\\
		&0&0&0&  
	\end{tikzcd}
\end{equation*}
Then we have the morphism $\delta: F''\to G'$ as described above. From the composition $r\circ j: F\to E'$, we get an induced map $\operatorname{Coker}(a)\to \operatorname{Coker}(j')$ and thus another map $F''\to G'$. By a diagram chase one checks this to be equal to $-\delta$.  

Now assume that further $\mathcal{A}$ is the category of modules over a simplicial ring $A$ in a topos. Then, the above morphisms fit into the following diagrams of triangles in $D^{\Delta}(A)$:

\begin{equation}\label{diag:conediags}
	\begin{tikzcd}[column sep = small]
		F'\ar[r]\ar[d,equals]& F'\ar[r]\ar[d]&F''\ar[r]\ar[d,"-\delta"]& \sigma F'\ar[d,equals]& & F'' \ar[r]\ar[d,"\delta"]& E''\ar[r]\ar[d]&G''\ar[r]\ar[d,equals]& \sigma F''\ar[d,"\sigma \delta"] \\
		F'\ar[r]&E'\ar[r]& G'\ar[r]& \sigma F', &&  G'\ar[r]& G\ar[r]& G''\ar[r]&\sigma G'.
	\end{tikzcd}
\end{equation} 

This shows that composing $\delta$ with the connecting map $G'\to \sigma F'$ yields \emph{minus} the connecting map associated to the exact sequence of the $F$'s, while composing $\sigma \delta$ with the connecting map $G''\to \sigma F''$ yields the connecting map associated to the exact sequence of the $G$'s. 
\paragraph{The reduced Atiyah class.}
Now let $f:X\to Y$ be a morphism of ringed topoi. Let $E\in C^{\leq 0}(\mathcal{O}_Y)$ be a complex whose components are Tor-independent to $f$, so that in particular the component-wise pullback $E_X:=f^*E$ equals the derived pullback.  Let 
\begin{equation}\label{eq:redatexseq}
	0\to F\to E_X\to G\to 0
\end{equation} 
be an exact sequence in $C^{\leq 0}(\mathcal{O}_X)$, which we also view as a sequence of simplicial $\mathcal{O}_X$-modules via the Dold--Kan correspondence. Let $R:=P_{f^{-1}\mathcal{O}_Y}(\mathcal{O}_X)$ be the standard simplicial resolution, and let $E_R:=f^{-1}E\otimes_{f^{-1}\mathcal{O}_Y} R$.  Denote by $\cdot\mid_R$ restriction of scalars from $\mathcal{O}_X$ to $R$. Since $R$ is flat, the natural morphism $E_R\to E_X|_R$ is a quasi-isomorphism, and it is termwise surjective, since $R\to \mathcal{O}_X$ is. Let $G_R:=G|_R$ and let $F_R$ be the kernel of the induced map $E_R\to G_R$. We have an induced map of exact sequences of $R$-modules, in which the vertical arrows are quasi-isomorphisms
\begin{equation*}
	\begin{tikzcd}
		0\ar[r]&F_R \ar[r]\ar[d] & E_R\ar[r]\ar[d]& G_R\ar[r]\ar[d]&0\\
		0\ar[r] & F|_R\ar[r]&E_X|_R\ar[r]&G|_R \ar[r]& 0.
	\end{tikzcd}
\end{equation*} 
We now take the exact sequence of principal parts with respect to the upper row. This gives the following commutative diagram of solid arrows with exact rows and columns

\begin{equation}\label{diag:exseqofexseq}
	\begin{tikzcd}
		&0\ar[d]&0\ar[d]&0\ar[d]& \\
		0\ar[r]& \Omega_{R/f^{-1}\mathcal{O}_Y}\otimes_R F_R\ar[r]\ar[d]&P^1_{R/f^{-1}\mathcal{O}_Y}(F_R)\ar[r]\ar[d]& F_R\ar[d]\ar[r] &0 \\
		0\ar[r]& \Omega_{R/f^{-1}\mathcal{O}_Y}\otimes_R E_R\ar[r]\ar[d]& P^1_{R/f^{-1}\mathcal{O}_Y}(E_R) \ar[r]\ar[d]\ar[l,dotted,"r"', bend right]&E_R\ar[r]\ar[d]\ar[l,dotted,"s"', bend right]&0\\
		0\ar[r]& \Omega_{R/f^{-1}\mathcal{O}_Y}\otimes_R G_R \ar[r]\ar[d]&P^1_{R/f^{-1}\mathcal{O}_Y}(G_R)\ar[r]\ar[d] & G_R\ar[r]\ar[d]& 0\\
		&0&0&0&  
	\end{tikzcd}
\end{equation}

Here the arrows denoted by $s$ and $r$ come from a splitting of the middle exact sequence, which is defined  as follows: 
Since $E_R = R\otimes_{f^{-1}\mathcal{O}_Y}f^{-1}E$, we have by definition 
\begin{gather*}
	P^1_{R/f^{-1}\mathcal{O}_Y}(E_R) = (R\otimes_{f^{-1}\mathcal{O}_Y} R/I_{\Delta}^2)\otimes_R E_R \\
	= (R\otimes_{f^{-1}\mathcal{O}_Y} R \otimes_{f^{-1}\mathcal{O}_Y} f^{-1}E)/(I_{\Delta}^{2}\cdot  R\otimes R\otimes f^{-1}E). 
\end{gather*}
Then $s$ is given on local sections by $a\otimes m\mapsto a\otimes 1 \otimes m$, which one checks to be a morphism of left $R$-modules.
We see that the lower right square of \eqref{diag:exseqofexseq} is of the form \eqref{diag:cornersquare}. In particular, we get the induced morphism
\begin{equation}\label{eq:morredat}
	\delta: F_R\to \Omega_{R/f^{-1}\mathcal{O}_Y}\otimes_R G_R.
\end{equation}
By passing to derived categories, and taking extensions of scalars along $R\to \mathcal{O}_X$, this corresponds to a morphism 
\[\overline{\at}_{E,X/Y,G}:F\to L_{X/Y}\otimes G\] 
in $D^{\Delta}(\mathcal{O}_X)$, or equivalently $D^{\leq 0}(\mathcal{O}_X)$, which we call the \emph{reduced Atiyah class} of $E$ over $Y$ with respect to the sequence \eqref{eq:redatexseq} on $X$. We also write $\overline{\at}_E$ if the remaining data is understood.

\begin{remark}
	The morphism $\delta:F_R\to \Omega_{R/f^{-1}\mathcal{O}_Y}\otimes_R G_R$ obtained by using the lower right corner of \eqref{diag:exseqofexseq} is explicitly given as follows: For a local section $f=\sum r_i\otimes f^{-1}e_i$, where $r_i$ are sections of $R$ and $e_i$ are sections of $E$ (over $Y$) we have 
	\[\delta:f\mapsto \sum dr_i\otimes \overline{e_i}.\]
\end{remark}

\begin{remark}\label{rem:redatcomptopos}
	The following triangles commute
	\begin{enumerate}[label = \arabic*)]
		\item 
		\begin{equation*}
			\begin{tikzcd}
				F\ar[r,"-\overline{\at}_E"]\ar[dr,"\at_F"']& L_{X/Y}\otimes G\ar[d] \\
				& L_{X/Y}[1]\otimes F,
			\end{tikzcd}
		\end{equation*}
		
		\item 
		\begin{equation*}
			\begin{tikzcd}
				G\ar[d]\ar[dr,"{\at_G}"]&  \\
				F[1]\ar[r,"{\overline{\at}_E[1]}"']& L_{X/Y}[1]\otimes G.
			\end{tikzcd}
		\end{equation*}
	\end{enumerate}
	In both cases, the vertical morphisms are induced from the connecting map $G\to F[1]$ of the given exact sequence. This follows from the morphisms of triangles \eqref{diag:conediags}.
\end{remark}
\begin{remark}\label{rem:redatgen}
	In general, the reduced Atiyah class depends on slightly more than a map $\varphi:f^*E\to G$ in the derived category $D(\mathcal{O}_X)$, analogously to how cones in the derived category are unique only up to non-unique isomorphism. However, for any given $\varphi$, by a variation of the above using cocones/cones instead of kernels/cokernels, one can define a reduced Atiyah class $F:=\operatorname{cocone}(\varphi)\to L_{X/Y}\otimes G$, which is well-defined \emph{up to an element of} $\Ext^{-1}(F,G)$.  
\end{remark}

\subsection{The reduced Atiyah class via the graded cotangent complex}
Let $f:X\to Y$ be a morphism of ringed topoi and let $E$ be an $\mathcal{O}_Y$-module, Tor-independent to $f$. Let $E_X:=f^*E$ and let 
\[0\to F \to E_X\to G\to 0\] 
be an exact sequence of $\mathcal{O}_X$-modules. 
We will use the $\Z$-graded cotangent complex as defined in \cite[IV 2]{Ill}.
For a $\Z$-graded ring $A=\oplus A_i$, we let $k^i$ denote the functor that sends a graded $A$-module $M=\oplus M_i$ to the $A_0$-module given by its $i$-th graded piece $M_i$. 

Let $X[E_X]$ denote the graded ringed topos whose underlying site is the \'etale site of $X$ and whose sheaf of rings is given by $\mathcal{O}_X\oplus E$, where $E$ is placed in degree $1$, and similarly for $Y[E],X[G]$. Let $q:X[G]\to X$ be the morphism induced by the inclusion $\mathcal{O}_X\to \mathcal{O}_X\oplus G$ and similarly $r:X[E_X]\to X$.

\begin{proposition}\label{prop:redatisoms}
	\begin{enumerate}[label=(\roman*)]
		\item We have a canonical natural isomorphism \[k^1q_*L^{gr}_{X[E_X]/X}\simeq E_X\] of objects of $D(\mathcal{O}_X)$.\label{redatisoms1}
		\item We have a canonical natural isomorphism \[k^1q_*L^{gr}_{X[G]/X[E_X]}\simeq F[1]\] of objects of $D(\mathcal{O}_X)$.\label{redatisoms2}
		\item Up to the isomorphisms in \ref{redatisoms1} and \ref{redatisoms2}, applying $k^1\circ q_*$ to the connecting homomorphism of graded cotangent complexes associated to the composition $X[G]\to X[E_X]\to X[F]$ recovers the inclusion $F\to E$ up to a shift.  
	\end{enumerate}
\end{proposition}
\begin{proof}
	Statement \ref{redatisoms1} is \cite[IV (2.2.5)]{Ill}, the other parts can be deduced from that as in \cite[\S1.8]{Gill}:  Consider the distinguished triangle in $D^{gr}(X[G])$
	\[L^{gr}_{X[E_X]/X}\mid_{X[G]}\to L^{gr}_{X[G]/X}\to L^{gr}_{X[E_X]/X[G]}\xrightarrow{+1} L^{gr}_{X[E_X]/X}[1]. \]
	Applying $k^1\circ q_*$ and the isomorphisms of \ref{redatisoms1}, we obtain a distinguished triangle in $D(X)$:
	\[E_X\to G\to k^1\circ q_* L^{gr}_{X[E_X]/X[G]}\to E_X[1].\] 
	It follows that there is a unique isomorphism $k^1 \circ q_* L^{gr}_{X[E_X]/X[G]}\simeq F$ that identifies the connecting map on cohomology sheaves with the inclusion $F\to E$. 
\end{proof}

Consider the transitivity triangle of graded cotangent complexes associated to the morphism of graded ringed topoi $X[G]\to X[E_X]\to Y[E]$, and let 
\[L^{gr}_{X[G]/X[E_X]}\to L^{gr}_{X[E_X]/Y[E]}\mid_{X[G]}[1] \] 
be the connecting map. Applying $k^1\circ q_*$ and the isomorphisms of Proposition \ref{prop:redatisoms}, we obtain a map
\begin{equation}\label{eq:atredsecdef}
	F[1]\to (L_{X/Y}\otimes G)[1].
\end{equation}

\begin{proposition}\label{prop:redatsecdef}
	The map \eqref{eq:atredsecdef} agrees with the shift $\overline{\at}_{E,X/Y,G}[1]$ of the reduced Atiyah class. 	
\end{proposition}
\begin{proof} 
	The definition in \cite[2.3]{Gill} goes through in our setting and gives the same resulting notion of reduced Atiyah class.
	By Theorem 2.6 there, the map defined in this way agrees with \eqref{eq:atredsecdef} up to a shift. 
\end{proof}

\begin{corollary}\label{cor:redatpullback}
	The reduced Atiyah class is preserved under Tor-independent pullback: Consider a commutative diagram of ringed topoi
	\begin{equation*}
		\begin{tikzcd}
			X'\ar[r,"\alpha"]\ar[d]& X\ar[d] \\
			Y'\ar[r,"\beta"] & Y
		\end{tikzcd}
	\end{equation*}
	such that $\beta$ is Tor-independent to $E$ and such that $\alpha$ is Tor-independent to $E_X$ and $G$. Then the diagram
	\begin{equation*}
		\begin{tikzcd}
			\alpha^*F\ar[r,"\alpha^*\overline{\at}_{E,X/Y,G}"]\ar[dr,"\overline{\at}_{\beta^*E, X'/Y', \alpha^*E}"']&[27 pt] \alpha^*L_{X/Y}\otimes \alpha^*G\ar[d] \\
			& L_{X'/Y'}\otimes \alpha^* G
		\end{tikzcd}
	\end{equation*}
	commutes.	
\end{corollary}
\begin{proof}
	This follows directly from Proposition \ref{prop:redatsecdef} and the functoriality of the transitivity triangle for graded cotangent complexes.  
\end{proof}

\begin{lemma}\label{lem:commdiagconn}
	Consider a commutative diagram of ringed topoi 
	\begin{equation*}
		\begin{tikzcd}
			W\ar[r]&X'\ar[r]\ar[d]& X\ar[d] &\\
			&Y'\ar[r] & Y\ar[r]&Z.
		\end{tikzcd}
	\end{equation*}
	We assume further that $X$ and $Y'$ are Tor-independent over $Y$, and that the induced square of rings on $X'$ obtained by pulling back the structure sheaves of $Y,Y'$ and $X$ is cocartesian.
	Then the diagram of (shifted) connecting maps 
	\begin{equation*}
		\begin{tikzcd}
			L_{W/X'}\ar[r]\ar[d]& L_{X'/Y'}\mid_W[1]\simeq L_{X/Y}\mid_W[1]\ar[d] \\
			L_{X'/X}\mid_W[1]\simeq L_{Y'/Y}\mid_W[1]\ar[r] & L_{Y/Z}\mid_W[2]
		\end{tikzcd}
	\end{equation*}
	on $W$ anticommutes. Here $\cdot\mid_W$ denotes pullback to $W$. 
\end{lemma}
\begin{proof}
	By taking suitable simplicial resolutions, we are reduced to the setting that we have a diagram 
	\begin{equation*}
		\begin{tikzcd}
			&C\ar[r]&C' \ar[r] &D\\
			A\ar[r] &B\ar[r]\ar[u] 	 & B'\ar[u]  &
		\end{tikzcd}
	\end{equation*}
	of simplicial rings in a topos $T$, in which the square is cocartesian, and in which all maps are free in each simplicial degree. 
	
	We then get a diagram of $D$-modules
	\[
	\begin{tikzcd}
		&0\ar[d]&0\ar[d]&0\ar[d]& \\
		0\ar[r]& \Omega_{B/A}\otimes_B D\ar[r]\ar[d]&\Omega_{C/A}\otimes_C D\ar[r]\ar[d]& \substack{\Omega_{C/B}\otimes_C D\\=\Omega_{C'/B'}\otimes_{C'} D}\ar[d]\ar[r] &0 \\
		0\ar[r]& \Omega_{B'/A}\otimes_B' D\ar[r]\ar[d]& \Omega_{D/A} \ar[r]\ar[d]&\Omega_{D/B'}\ar[r]\ar[d]&0\\
		0\ar[r]& \substack{\Omega_{B'/B}\otimes_{B'}D\\=\Omega_{C'/C}\otimes_{C'} D} \ar[r]\ar[d]&\Omega_{D/C}\ar[r]\ar[d] & \Omega_{D/C'}\ar[r]\ar[d]& 0\\
		&0&0&0&  
	\end{tikzcd}
	\]
	
	It is now a basic exercise in homological algebra to show that the following induced diagram of connecting maps in $D^{\Delta}(D)$ anticommutes:
	\begin{equation*}
		\begin{tikzcd}
			\Omega_{D/C'}\ar[r,"\delta"]\ar[d,"\delta"]& \sigma\Omega_{C'/C}\ar[d,"\sigma \delta"] \\
			\sigma\Omega_{C'/B'}\otimes_{C'}D\ar[r,"{\sigma \delta}"] &\sigma^2 \Omega_{B/A}\otimes_B D
		\end{tikzcd}
	\end{equation*}
\end{proof}

\begin{proposition}\label{prop:redatcompbasic}
	Let $E,F,G$ and $f:X\to Y$ be as before, and let $Y\to Z$ be a morphism of ringed topoi. 
	Then the following diagram anti-commutes 
	\begin{equation*}
		\begin{tikzcd}
			F\ar[rr,"{\overline{\at}_{E,X/Y,G}}"]\ar[d]& & L_{X/Y}\otimes G\ar[d] \\
			E\ar[r,"{\at_{E,Y/Z}}"] & L_{Y/Z}[1]\otimes E\ar[r]&L_{Y/Z}[1]\otimes G,
		\end{tikzcd}
	\end{equation*}
	where the right vertical map is obtained from the connecting homomorphism of cotangent complexes by tensoring with $G$. 
\end{proposition}
\begin{proof}
	We apply Lemma \ref{lem:commdiagconn} to the diagram of ringed topoi
	\begin{equation*}
		\begin{tikzcd}
			X[G]\ar[r]&X[E]\ar[r]\ar[d]& X\ar[d] &\\
			&Y[E]\ar[r] & Y\ar[r]&Z,
		\end{tikzcd}
	\end{equation*}
	which gives the anti-commutative diagram
	\begin{equation*}
		\begin{tikzcd}
			L_{X[G]/X[E]}\ar[r]\ar[d]& L_{X/Y}\mid_{X[G]}[1]\ar[d] \\
			L_{Y[E]/Y}\mid_{X[G]}[1]\ar[r] & L_{Y/Z}\mid_{X[G]}[2].
		\end{tikzcd}
	\end{equation*}
	By applying $k^1\circ q_*$ and using the identifications of Proposition \ref{prop:redatisoms}, the result follows.
\end{proof}
\subsection{The Atiyah class for an exact sequence}\label{subsec:atexseqtopoi}
Let $f:X\to Y$ be a morphism of ringed topoi and let $0\to F\to E\to G\to 0$ be an exact sequence of bounded above complexes of $\mathcal{O}_X$-modules, such that $F,E$ and $G$ are dualizable and such that their duals lie again in $D^-(X)$. After taking appropriate resolutions and up to shifting, we may assume that $E,F$ and $G$ are concentrated in degrees $\leq 0$ and that they have flat components.  
Let $\mathcal{O}_X\to J$ be an injective resolution, and let $F^{\vee}, E^{\vee}$ and $G^{\vee}$ be the complexes obtained by applying $\Sheafhom_{X}(-,J)$ to $F,E$ and $G$ respectively. Let $N$ be large enough so that $E^{\vee},F^{\vee} $ and $G^{\vee}$ have no nonzero cohomology in degrees $\geq N$, and set $\overline{E}:=(\tau^{\leq N} E)[N]$, and similarly for $\overline{F}$, $\overline{G}$. We let $J':=\tau^{\leq N}J[N]$. Then the sequence $0\to \overline{G}\to\overline{E}\to \overline{F}\to 0$ is exact, and we have the natural commutative diagram of complexes  
\begin{equation*}
	\begin{tikzcd}
		\overline{G}\otimes E\ar[r]\ar[d]& \overline{G}\otimes G\ar[d] \\
		\overline{E}\otimes E\ar[r] & \overline{J}.
	\end{tikzcd}
\end{equation*} 
Using the Alexander--Whitney map of the Dold--Kan correspondence, we get such a diagram in the category of simplicial $\mathcal{O}_X$-modules. 

Now let $R=P^1_{f^{-1}\mathcal{O}_Y}(\mathcal{O}_X)$. Then we have the following commutative diagram of $R$-modules in which the rows are exact sequences
\begin{equation*}
	\begin{tikzcd}
		0\ar[r]&\overline{G}\otimes \Omega^1_{R/f^{-1}\mathcal{O}_Y}\otimes F\ar[r]\ar[d,"\beta"]&\overline{G}\otimes P^1_{R/f^{-1}\mathcal{O}_Y}(F)\ar[r]\ar[d]&\overline{G}\otimes F \ar[d]\ar[r]& 0 \\
		0\ar[r]&\overline{E}\otimes \Omega^1_{R/f^{-1}\mathcal{O}_Y}\otimes E\ar[r]\ar[d,"\alpha"]&\overline{E}\otimes P^1_{R/f^{-1}\mathcal{O}_Y}(E)\ar[r]\ar[d]&\overline{E}\otimes E \ar[d]\ar[r]& 0 \\
		0\ar[r]& \Omega^1_{R/f^{-1}\mathcal{O}_Y}\otimes \overline{J}\ar[r]&\alpha_*(\overline{E}\otimes P^1_{R/f^{-1}\mathcal{O}_Y}(E))\ar[r]&\overline{E}\otimes E \ar[r]& 0 
	\end{tikzcd}
\end{equation*}
Here the last row is obtained from the second by pushout along the map $\alpha$, and $\alpha$ is induced from the evaluation map $\overline{E}\otimes E\to J$. Abbreviate $P(E):=P^1_{R/f^{-1}\mathcal{O}_Y}$. Since the composition $\alpha\circ\beta$ is zero, we can quotient out $\overline{G}\otimes \Omega^1_{R/f^{-1}\mathcal{O}_Y}\otimes F$ to obtain a diagram.
\begin{equation*}
	\begin{tikzcd}
		0\ar[r]&0\ar[r]\ar[d]&\overline{G}\otimes F\ar[r]\ar[d]&\overline{G}\otimes F \ar[d,"\iota"]\ar[r]& 0 \\
		0\ar[r]& \Omega^1_{R/f^{-1}\mathcal{O}_Y}\otimes \overline{J}\ar[r]&\alpha_*(\overline{E}\otimes P(E))\ar[r]&\overline{E}\otimes E \ar[r]& 0.
	\end{tikzcd}
\end{equation*} 
Here, the vertical maps are injections of $R$-modules, so we can take the quotient exact sequence
\[0\to \Omega^1_{R/f^{-1}\mathcal{O}_Y}\otimes \overline{J} \to \frac{\alpha_*(\overline{E}\otimes P(E))}{\iota(\overline{G}\otimes F)}\to \frac{\overline{E}\otimes E}{\overline{G}\otimes F} \to 0. \]

The induced connecting map defines (after extending scalars to $\mathcal{O}_X$, applying the Dold--Kan correspondence, and shifting) a morphism in $D^-(\mathcal{O}_X)$
\[\frac{E^{\vee}\otimes E}{G^{\vee}\otimes F}[-1]\to L_{X/Y},\]
which we call the Atiyah class $\at_{\underline{E}}=\at_{\underline{E},X/Y}$ of the exact sequence $0\to F\to E\to G \to 0$.

\begin{remark}
	It is easy to see by standard arguments that the result of the construction is independent of the choice of $J$. The dependence on $N$ is not addressed here. 
\end{remark}

We observe directly from the construction that

\begin{corollary}\label{cor:atexsequsat}
	The morphism $\at_{\underline{E}}$ is compatible with the usual Atiyah class, i.e. the diagram 
	\begin{equation*}
		\begin{tikzcd}
			E^{\vee}\otimes E[-1]\ar[d]\ar[dr,"{\at'_E}"]&  \\
			\frac{E^{\vee}\otimes E}{G^{\vee}\otimes F}[-1]\ar[r,"\at_{\underline{E}}"] & L_{X/Y}. 
		\end{tikzcd}
	\end{equation*}
	commutes.
\end{corollary}

\begin{corollary}\label{cor:atexseqfuncy}
	The map $\at_{\underline{E},X/Y}$ is functorial in $Y$, i.e. given a map $Y\to Y'$, the composition 
	\[\frac{E^{\vee}\otimes E}{G^{\vee}\otimes F}[-1]\xrightarrow{\at_{\underline{E},X/Y'}} L_{X/Y'}\to L_{X/Y}  \] 
	equals $\at_{\underline{E},X/Y}$ (assuming we make the same choices of $N$ in each construction).
\end{corollary}

\begin{lemma}\label{lem:atexseqfuncx}
	The map $\at_{\underline{E},X/Y}$ is functorial for morphisms $a:X'\to X$. More precisely, given such a morphism, the diagram
	\begin{equation*}
		\begin{tikzcd}
			a^*\left(\frac{E^{\vee}\otimes E}{G^{\vee}\otimes F}[-1]\right) \ar[r,"\at_{\underline{E}}"]\ar[d]& a^*L_{X/Y}\ar[d] \\
			\frac{(a^*E)^{\vee}\otimes a^*E}{(a^*G)^{\vee}\otimes a^*F}[-1]\ar[r,"\at_{a^*\underline{E}}"] & L_{X'/Y}
		\end{tikzcd}
	\end{equation*}
	commutes (assuming we make the same choices of $N$ in the construction). Here, we assume that $E,F,G$ are already given by bounded above complexes with flat components. 
\end{lemma}
\begin{proof}
	The main point is that to compute the derived pullback of $\overline{E}$ (and similarly $\overline{G}, \overline{F}$), we may use either a flat resolution of $\overline{E}$-- denote this $La^*\overline{E}$ by abuse of notation, or repeat the construction with $a^*E$ in place of $E$--which we denote $\overline{a^*E}$. The two are related by a natural map $La^*E\to \overline{a^*E}$, which is a quasi-isomorphism, as can be checked locally where it follows from the assumption that $E$ is a perfect complex. The details are left to the reader. 
\end{proof}

Now suppose that we have an exact sequence $0\to F\to E_X\to G\to 0$ as in \eqref{eq:redatexseq}. Then the Atiyah class for the exact sequence is related to the reduced Atiyah class
\begin{proposition}\label{prop:redatexseqat}
	We have a commutative diagram 
	\begin{equation*}
		\begin{tikzcd}
			\frac{E^{\vee}\otimes E}{G^{\vee}\otimes F}[-1]\ar[r,"-"]\ar[dr, "\at_{\underline{E}}"']& G^{\vee}\otimes F\ar[d, "\overline{\at}_E'"] \\
			& L_{X/Y},
		\end{tikzcd}
	\end{equation*}
	where the horizontal map is \emph{minus} the natural connecting homomorphism.
\end{proposition} 
\begin{proof}
	We may work with the sequence 
	\[0\to F_R\to E_R\to G_R\to 0 \]
	instead of $F,E,G$ as in the construction of the reduced Atiyah class. 
	We need to show that we have a commutative diagram 
	\begin{equation*}
		\begin{tikzcd}
			0\ar[r]& \overline{G}\otimes F_R\ar[r]\ar[d]&\overline{E}\otimes E_R \ar[r]\ar[d]& \frac{\overline{E}\otimes E_R}{\overline{G}\otimes F_R} \ar[r]\ar[d,equals]& 0\\
			0\ar[r] & L_{X/Y}\otimes \overline{J}\ar[r]& \frac{\alpha_*(\overline{E}\otimes P(E_R))}{\iota(\overline{G}\otimes F_R)}\ar[r] & \frac{\overline{E}\otimes E_R}{\overline{G}\otimes F_R} \ar[r]& 0,
		\end{tikzcd}
	\end{equation*}
	where the left vertical arrow is obtained from the composition 
	\[\overline{G}\otimes F_R\to \overline{G}\otimes \Omega^1_{R/f^{-1}\mathcal{O}_Y}\otimes G_R\to \overline{G}\otimes \Omega^1_{R/f^{-1}\mathcal{O}_Y}\otimes G\to \overline{J},\]
	and therefore induces the reduced Atiyah class (up to a shift) when passing to $D(X)$.
	The proposition then clearly follows. 
	To construct the diagram, let $s:E_R\to P(E_R)$ denote the section in the construction of the reduced Atiyah class. Then the middle vertical map is the composition 
	\[\overline{E}\otimes E_R\xrightarrow{s\otimes \operatorname{id}} \overline{E}\otimes P(E)\to \frac{\alpha_*(\overline{E}\otimes P(E_R))}{\iota(\overline{G}\otimes F_R)}.\]
	Since the composition $\overline{E}\otimes E_R\to \overline{E}\otimes P(E_R)\to \overline{E}\otimes E_R$ is the identity, we naturally obtain a commutative diagram of the desired form and it remains only to show that the left vertical map is as desired. But in fact, we have a subdiagram 
	
	\begin{equation*}
		\begin{tikzcd}
			0\ar[r]& \overline{G}\otimes F_R\ar[r]\ar[d]&\overline{G}\otimes E_R \ar[r]\ar[d]& \frac{\overline{G}\otimes E_R}{\overline{G}\otimes F_R} \ar[r]\ar[d,equals]& 0\\
			0\ar[r] & L_{X/Y}\otimes \overline{J}\ar[r]& \frac{\alpha_*(\overline{G}\otimes P(E_R))}{\iota(\overline{G}\otimes F_R)} \ar[r] & \frac{\overline{G}\otimes E_R}{\overline{G}\otimes F_R} \ar[r]& 0,
		\end{tikzcd}
	\end{equation*}
	the lower row of which is identified with the row
	\[0\to L_{X/Y}\otimes \overline{J} \to \alpha_*(\overline{G}\otimes G_R)\to \overline{G}\otimes G_R\to 0.\]
	Under this identification, the middle map factors as 
	\[\overline{G}\otimes E_R\xrightarrow{id_{\overline{G}}\otimes s} \overline{G}\otimes P(G)\to \alpha_*(\overline{G}\otimes P(G)),\] and it follows that the left vertical map indeed factors through the reduced Atiyah class as desired. 
\end{proof}

\section{Definitions}
We construct the Atiyah class for an algebraic stack, and show that it is independent of various choices made in the construction. We then address the case of the reduced Atiyah class and of the Atiyah class of an exact sequence. 
\subsection{Construction of the Atiyah class}\label{sec:constr-at}

\begin{construction}\label{constr:atiyah}
	Let $f:\mathcal{X}\to \mathcal{Y}$ be a morphism of algebraic stacks and let $E\in D^{\leq 0}_{qcoh}(\mathcal{X})$. By truncating, we may assume that $E$ is represented by a complex with nonzero terms only in negative degrees. We let $E_{\eqtopos{W}}$ be the induced $\mathcal{O}_{\eqtopos{W}}$-module, which we also regard as a simplicial module. Choose a diagram as in \eqref{diag:2comm} and consider the setup of Situation \ref{sit:wxtopoi}. We let $R:=P_{h^{-1}\mathcal{O}_{\eqtopos{Y}}}(\eqtopos{W})$ be the free simplicial resolution. This is a simplicial ring on $\eqtopos{W}$ with components $R_X=P_{g^{-1}\mathcal{O}_{Y_{\bullet}}}(\mathcal{O}_{X_{\bullet}})$ and $R_W=P_{h^{-1}\mathcal{O}_{Y_{\bullet}}}(\mathcal{O}_{W_{\bullet}})$. We regard $E_{\eqtopos{W}}$ as an $R$-module via restriction of scalars.
	Consider the exact sequence of principal parts $ \underline{P}_{R/\mathcal{O}_{\eqtopos{Y}}}^1(E_{\eqtopos{W}})$ associated to $E_{\eqtopos{W}}$ and the ring map $\eqtopos{h}^{-1}\mathcal{O}_{\eqtopos{Y}}\to R$.
	\[0\to L_{\eqtopos{W}/\eqtopos{Y}}\otimes_{\mathcal{O}_{\eqtopos{W}}} E_{\eqtopos{W}} \to P_{R/\eqtopos{h}^{-1}\mathcal{O}_{\eqtopos{Y}}}^1(E_{\eqtopos{W}})  \to E_{\eqtopos{W}} \to 0.\]
	It induces a map 
	\begin{equation}
		\delta_{R/\eqtopos{h}^{-1} \mathcal{O}_{\eqtopos{Y}}}(E_{\eqtopos{W}}): E_{\eqtopos{W}} \to \sigma L_{\eqtopos{W}/\eqtopos{Y}}\otimes_{\mathcal{O}_{\eqtopos{W}}} E_{\eqtopos{W}} \label{eq:atfirstmap}
	\end{equation} in $D^{\Delta}(R)$. By Lemma \ref{lem:dertens}, this is just the restriction of a morphism in $D^{\Delta}(\eqtopos{W})$ and thus corresponds to a unique morphism $E_{\eqtopos{W}} \to  L_{\eqtopos{W}/\eqtopos{Y}}\otimes_{\mathcal{O}_{\eqtopos{W}}} E_{\eqtopos{W}}[1]$ in $D(\mathcal{O}_{\eqtopos{W}})$ by the Dold--Kan correspondence.  We have canonical natural isomorphisms $\coneop^{\Delta}(E)\xrightarrow{\sim} E_{W_{\bullet}}[1]$ and $\coneop(L_{\eqtopos{W}/\eqtopos{Y}})\xrightarrow{\sim} \eta_W^*L_{\mathcal{X}/\mathcal{Y}}[1]$ by Lemmas \ref{lem:conepresqcoh} and \ref{lem:coneprescot} respectively. Since $E$ has quasi-coherent cohomology sheaves, the assumptions of Lemma \ref{lem:tensorcone} \ref{lem:tensorcone1} are satisfied, so that we obtain a morphism 
	\begin{equation}\label{eq:atsecmap}
		E_{W_{\bullet}}[1]\to  \eta_W^*L_{\mathcal{X}/\mathcal{Y}}\otimes E_{W_{\bullet}}[2].
	\end{equation}
	Applying $\eta_{W*}$ and the shift $[-1]$ and multiplying by $-1$ (the sign is obtained from commuting two shift functors and is needed for compatibility with the usual Atiyah class), we obtain a morphism 
	\begin{equation}\label{eq:atdef}
		\at_{E,\mathcal{X}/\mathcal{Y}}:=\at_{E,\mathcal{X}/\mathcal{Y},X/Y}: E\to L_{\mathcal{X}/\mathcal{Y}}\otimes E[1].
	\end{equation}
	This is the \emph{Atiyah class of $E$ over $\mathcal{Y}$}. We will also write $\at_E$, if the morphism $f:\mathcal{X}\to \mathcal{Y}$ is understood. In Corollary \ref{cor:atindep}, we show that the Atiyah class is independent of choice of diagram \eqref{diag:2comm}.
\end{construction}

As a consequence of the construction, we have 
\begin{lemma}\label{lem:functoriallem}
	The morphism \eqref{eq:atdef} is functorial in $E\in D^{\leq 0}_{qcoh}(\mathcal{X})$ for a fixed choice of diagram \eqref{diag:2comm}. 
\end{lemma}

\begin{remark}\label{rem:atreple}
	\begin{enumerate}[label = \roman*)]
		\item	In Construction \ref{constr:atiyah}, let $F$ be an $R$-module together with an isomorphism $F\to E_{\eqtopos{W}}$ in $D^{\Delta}(R)$. Then we may instead work with the exact sequence $\underline{P}_{R/\eqtopos{h}^{-1}\mathcal{O}_{\eqtopos{Y}}}(F)$, and define \eqref{eq:atfirstmap} equivalently as the map obtained as the composition
		\[E_{\eqtopos{W}}\xrightarrow{\sim}F \to \sigma \underline{P}_{R/\mathcal{O}_{\eqtopos{Y}}}^1(F_{\eqtopos{W}}) \xrightarrow{\sim} \sigma L_{\eqtopos{W}/\eqtopos{Y}}\otimes_{\mathcal{O}_{\eqtopos{W}}} E_{\eqtopos{W}}.\]
		\item \label{rem:atreple2} We may also replace $\eqtopos{W}$ and $\eqtopos{Y}$ by their analogues $W_{\wedge}$ and $Y_{\wedge}$, associated to the respective diagrams 
		\[X_{\bullet}\xleftarrow{s_{\bullet}} W_{\bullet} \xrightarrow{t_{\bullet}} X_{\bullet}, \mbox{ and } Y_{\bullet} \xleftarrow{=}Y_{\bullet}\xrightarrow{=}Y_{\bullet}\] 
		throughout Construction \ref{constr:atiyah}.
	\end{enumerate}
\end{remark}

\paragraph{Well-definedness and compatibility with pullback.}
Suppose that we are given a map of diagrams \eqref{diag:2comm}, i.e. that we have a $2$-commuting cube

\begin{equation}\label{eq:functorialitycube}
	\begin{tikzcd}[cramped]
		X' \arrow[rr] \arrow[dr] \arrow[dd] &&
		\mathcal{X}' \arrow[dd] \arrow[dr,"A"] \\
		& X \arrow[rr,crossing over] &&
		\mathcal{X} \arrow[dd] \\
		Y' \arrow[rr] \arrow[dr] && \mathcal{Y}' \arrow[dr] \\
		& Y \arrow[rr] \arrow[uu,<-, crossing over]&& \mathcal{Y},
	\end{tikzcd}
\end{equation}
whose front and back faces are as in $\eqref{diag:2comm}$. (This means in particular, that for any two maps $X'\to \mathcal{Y}$ obtained by traveling along the edges of the cube, the two two-isomorphisms relating them by traversing the faces are identical.) 

\begin{lemma}\label{lem:functoriality}
	Let $E\in D^{\leq 0}_{qcoh}(\mathcal{X})$. Let 
	\[\at_E:=\at_{E,\mathcal{X}/\mathcal{Y},X/Y} \mbox{ and } \at_{LA^*E}:=\at_{LA^*E,\mathcal{X}'/\mathcal{Y}',X'/Y' }.\] 
	Then the diagram 
	\begin{equation*}
		\begin{tikzcd}[column sep= 3 em ]
			LA^*E\ar[r,"{LA^*\at_{E}}"]\ar[d]& LA^*L_{\mathcal{X}/\mathcal{Y}}\otimes LA^*E[1]\ar[d] \\
			LA^*E\ar[r,"{\at_{LA^*E}}"] & L_{\mathcal{X}'/\mathcal{Y}'} \otimes LA^*E[1]
		\end{tikzcd}
	\end{equation*}
	commutes.
\end{lemma}
\begin{proof}
	By using the setup of  Situation \ref{sit:wxtopoi} and Construction \ref{constr:atiyah} for the primed objects, we obtain a commutative cube of simplicial algebraic spaces
	\begin{equation*}
		\begin{tikzcd}[row sep=2.em]
			W'_{\bullet} \arrow[rr,"t_{\bullet}'"] \arrow[dr,swap,"s_{\bullet}'"] \arrow[dd,"a_{\bullet}"'] &&
			X_{\bullet}' \arrow[dd,swap, "b_{\bullet}" near start] \arrow[dr] \\
			& X_{\bullet}' \arrow[rr,crossing over] &&
			Y_{\bullet}' \arrow[dd] \\
			W_{\bullet} \arrow[rr,"t_{\bullet}" near end] \arrow[dr,swap,"s_{\bullet}"] && X_{\bullet} \arrow[dr] \\
			& X_{\bullet} \arrow[rr] \arrow[uu,<-, crossing over, "b_{\bullet}" near end]&& Y_{\bullet}
		\end{tikzcd}
	\end{equation*}
	as well as morphisms of topoi $\eqtopos{a}:\eqtopos{W}'\to \eqtopos{W}$ and $\eqtopos{Y}'\to \eqtopos{Y}$, which fit in a $2$-commutative square
	\begin{equation*}
		\begin{tikzcd}
			\eqtopos{W}'\ar[r]\ar[d]& \eqtopos{Y}'\ar[d] \\
			\eqtopos{W}\ar[r] & \eqtopos{Y}.
		\end{tikzcd}
	\end{equation*} 	
	By construction, the Atiyah class $\at_E$ corresponds via $\eta_W^*$ to a map $\alpha:E_{W_{\bullet}}\to \eta_WL_{\mathcal{X}/\mathcal{Y}}[1]\otimes E_{W_{\bullet}}$, obtained as the shift of the map \eqref{eq:atsecmap}. Since the pullback $LA^*$ can be computed as $\eta_{W'*}a_{\bullet}^*\eta_W^*$, we have a natural commutative diagram 
	\begin{equation*}
		\begin{tikzcd}[column sep = huge]
			\eta_{W'}^*LA^*E\ar[r,"\eta_{W'}^*LA^*\at_E"] \ar[d]&\eta_{W'}^*LA^*(L_{\mathcal{X}/\mathcal{Y}}[1]\otimes E) \ar[d] \\
			La_{\bullet}^*E_W\ar[r, "La_{\bullet}^*\alpha"] & La_{\bullet}^*(\eta_W^*L_{\mathcal{X}/\mathcal{Y}}[1]\otimes E_W).
		\end{tikzcd}
	\end{equation*}
	
	Let $E':=\eta_{W'*}(La_{\bullet}^*E_W)$. We apply Construction \ref{constr:atiyah} to the primed objects, i.e. with respect to the backside of \eqref{eq:functorialitycube} and $E'$. 
	We obtain the ring $R'=P_{\eqtopos{h}'^{-1}\mathcal{O}_{\eqtopos{Y}'}}(\mathcal{O}_{\eqtopos{W}'})$ and the map 
	\[\alpha':E'_{W'_{\bullet}}\to \eta_{W'}^*L_{\mathcal{X}'/\mathcal{Y}'}[1]\otimes E'_{W'_{\bullet}},\]
	as the shift of  the map \eqref{eq:atsecmap}. By functoriality of simplicial resolutions, we have a natural map  $\eqtopos{a}^{-1}R\to R'$ of rings on $\eqtopos{W}'$. 
	Now the statement of the lemma is equivalent to the following 
	\begin{claim}
		The diagram 
		\begin{equation*}
			\begin{tikzcd}
				La_{\bullet}^*E_{W_{\bullet}}\ar[r,"La_{\bullet}^*\alpha"] \ar[d]& La_{\bullet}^*(\eta_{W}^*L_{\mathcal{X}/\mathcal{Y}}[1]\otimes E_{W_{\bullet}}) \ar[d] \\		
				E'_{W'_{\bullet}}\ar[r,"\alpha'"]& \eta_{W'}^*L_{\mathcal{X}'/\mathcal{Y}'}[1]\otimes E'_{W'_{\bullet}} 
			\end{tikzcd}
		\end{equation*}
		in $D(W'_{\bullet})$ commutes, where the vertical maps are induced by the natural isomorphism $La_{\bullet}^*E_{W_{\bullet}}\xrightarrow{\sim} E'_{W'}$ and the pullback map on cotangent complexes.
	\end{claim}
	\begin{proof}
		We may assume that $E$ is represented by a complex of flat modules concentrated in degrees $\leq 0$, so that the same holds for $E_{\eqtopos{W}}$. Then we have $L\eqtopos{a}^*E_{\eqtopos{W}}=\eqtopos{a}^*E_{\eqtopos{W}}$ and its components are given by $b_{\bullet}^*E_X$ and $a_{\bullet}^*E_W$ respectively.

		By functoriality of the principal parts construction, we have a morphism of exact sequences of $\eqtopos{a}^{-1}R$-modules $\eqtopos{a}^{-1}\underline{P}^1_{\mathcal{O}_{\eqtopos{W}}/\eqtopos{h}^{-1}\mathcal{O}_{\eqtopos{Y}}}(E_{\eqtopos{W}})\to \underline{P}^1_{\mathcal{O}_{\eqtopos{W}'}/\eqtopos{h'}^{-1}\mathcal{O}_{\eqtopos{Y'}}}(\eqtopos{a}^*E_{\eqtopos{W}}).$	
		We get the induced diagram of connecting maps in $D^{\Delta}(\eqtopos{a}^{-1}R)$
		\begin{equation*}
			\begin{tikzcd}
				\eqtopos{a}^{-1}\delta_{R/\eqtopos{h}^{-1} \mathcal{O}_{\eqtopos{Y}}}(E_{\eqtopos{W}}):&[-1 cm]\eqtopos{a}^{-1}E_{\eqtopos{W}}\ar[r]\ar[d]&\sigma \eqtopos{a}^{-1}L_{\eqtopos{W}/\eqtopos{Y}} \otimes_{\eqtopos{a}^{-1}\mathcal{O}_{\eqtopos{W}}} \eqtopos{a}^{-1}E_{\eqtopos{W}} \ar[d] \\
				\delta_{R'/\eqtopos{h}'^{-1} \mathcal{O}_{\eqtopos{Y}'}}(\eqtopos{a}^*E_{\eqtopos{W}}):&\eqtopos{a}^*E_{\eqtopos{W}} \ar[r] & \sigma L_{\eqtopos{W}'/\eqtopos{Y'}} \otimes_{\mathcal{O}_{\eqtopos{W}'}} \eqtopos{a}^*E_{\eqtopos{W}}.
			\end{tikzcd}
		\end{equation*}	
		By adjunction, this corresponds to a diagram in $D^{\Delta}(R')$, which in turn corresponds to a diagram in $D(\mathcal{O}_{\eqtopos{W}})$:
		\begin{equation*}
			\begin{tikzcd}
				L\eqtopos{a}^*E_{\eqtopos{W}}\ar[r]\ar[d]& L\eqtopos{a}^*L_{\eqtopos{W}/\eqtopos{Y}}[1]\otimes L\eqtopos{a}^*E_{\eqtopos{W}}\ar[d] \\
				L\eqtopos{a}^*E_{\eqtopos{W}}\ar[r] & L_{\eqtopos{W}'/\eqtopos{Y}'}[1]\otimes L\eqtopos{a}^*E_{\eqtopos{W}}.
			\end{tikzcd}
		\end{equation*}
		After applying $\coneop$ and shifting, the upper line is canonically identified with $La_{\bullet}^*\alpha$ (this uses that we have a natural isomorphism $\coneop\circ L\eqtopos{a}^*\simeq La_{\bullet}^*\circ \coneop$, which one can check for $K$-flat $\mathcal{O}_{\eqtopos{W}}$ modules), while the lower line yields a map 
		\[\alpha'':La_{\bullet}^*E_{W_{\bullet}}\to \eta_{W'}^*L_{\mathcal{X}'/\mathcal{Y}'}[1]\otimes La_{\bullet}^*E_{W_{\bullet}}.\] 
		To finish the proof of the claim, we need to show that under the isomorphism $La^*E_{W}\to E'_{W'}$, the maps $\alpha''$ and $\alpha$ get identified. 
		This follows from Lemma \ref{lem:technicalfunctorialitylemma} and Remark \ref{rem:atreple} \ref{rem:atreple2}.
		
	\end{proof}
\end{proof}
\begin{corollary}\label{cor:atindep}
	The map $\at_{E}:E\to L_{\mathcal{X}/\mathcal{Y}}[1]\otimes E$ obtained from Construction \ref{constr:atiyah} is independent of choice of diagram \eqref{diag:2comm}.
\end{corollary}
\begin{proof}
	Suppose we are given two choices 
	
	\begin{minipage}{.4\textwidth}
		\begin{equation*}
			\begin{tikzcd}
				X_1\ar[r]\ar[d]& \mathcal{X}\ar[d] \\
				Y_1\ar[r] & \mathcal{Y}
			\end{tikzcd}
		\end{equation*}
	\end{minipage}

	\begin{minipage}{.4\textwidth}
		\begin{equation*}
			\begin{tikzcd}
				X_2\ar[r]\ar[d]& \mathcal{X}\ar[d] \\
				Y_2\ar[r] & \mathcal{Y}.
			\end{tikzcd}
		\end{equation*}
	\end{minipage}
	
	We need to show that $\at_{E,\mathcal{X}/\mathcal{Y},X_1/Y_1}$ and $\at_{E,\mathcal{X}/{\mathcal{Y}},X_2/Y_2}$ agree. By replacing $X_2\to Y_2$ with the fiber product $X_{1}\times_{\mathcal{X}}X_2\to Y_1\times_{\mathcal{Y}}Y_2$, we may without loss of generality assume that we have a $2$-commuting diagram 
	\begin{equation*}
		\begin{tikzcd}[row sep = small]
			&X_2\ar[dl]\ar[dd]\ar[dr] & \\
			X_1 \ar[rr, crossing over] \ar[dd]&& \mathcal{X}\ar[dd]\\
			&Y_2\ar[dl]\ar[dr]& \\
			Y_1\ar[rr] && \mathcal{Y}. 
		\end{tikzcd}
	\end{equation*}
	By adding in the identity morphisms on $\mathcal{X}$ and $\mathcal{Y}$, this gives a $2$-commuting cube \eqref{eq:functorialitycube}, and we obtain the desired equality from Lemma \ref{lem:functoriality}.
\end{proof}
\begin{corollary}
	The Atiyah class is compatible with pullback. More precisely, given a $2$-commutative square 
	\begin{equation*}
		\begin{tikzcd}
			\mathcal{X}'\ar[r,"A"]\ar[d]& \mathcal{X}\ar[d] \\
			\mathcal{Y}'\ar[r] & \mathcal{Y},
		\end{tikzcd}
	\end{equation*}
	the natural diagram 
	\begin{equation*}
		\begin{tikzcd}[column sep= .6 in ]
			LA^*E\ar[r,"LA^*\at_E"]\ar[d]& LA^*L_{\mathcal{X}/\mathcal{Y}}\otimes LA^*E[1]\ar[d] \\
			LA^*E\ar[r,"\at_{LA^*E}"] & L_{\mathcal{X}'/\mathcal{Y}'} \otimes LA^*E[1]
		\end{tikzcd}
	\end{equation*}
	commutes.
\end{corollary}
\begin{proof}
	This follows immediately from Lemma \ref{lem:functoriality}, after completing the square to a cube as in \eqref{eq:functorialitycube}.
\end{proof}

\paragraph{Compatibility with Illusie's definition.}
\begin{proposition}\label{prop:atcomp}
	If $\mathcal{X}$ and $\mathcal{Y}$ are in fact algebraic spaces, then the morphism $\at_E$ agrees with Illusie's Atiyah class for $E$ with respect to the map of ringed \'etale topoi $\mathcal{X}_{\operatorname{et}}\to \mathcal{Y}_{\operatorname{et}}$.
\end{proposition}
\begin{proof}
	In this case, we may choose $X=\mathcal{X}$ and $Y=\mathcal{Y}$ in diagram \eqref{diag:2comm}, so that we get $W=X$. Then $X_{\bullet}$ and $W_{\bullet}$ are constant strictly simplicial algebraic spaces with value $X$, and $Y_{\bullet}$ is the constant strictly simplicial algebraic space with value $Y$. 
	In this case, we have the morphism of topoi $\Delta:\eqtopos{W}\to X_{\bullet}$, where $\Delta^*M_X = (M_X,M_X,\operatorname{id},\operatorname{id})$. The morphism $\gamma$ of Lemma \ref{lem:coneforqiso} gives a natural transformation of functors $\coneop_{\mathcal{O}_W}\circ \Delta^*\Rightarrow [-1]$, which induces to a natural isomorphism of functors from $D(\mathcal{O}_{X_{\bullet}})$ to itself. 
	It follows from Construction \ref{constr:atiyah} that the morphism $E_{\eqtopos{W}} \to  L_{\eqtopos{W}/\eqtopos{Y}}[1]\otimes_{\mathcal{O}_{\eqtopos{W}}} E_{\eqtopos{W}}$ obtained from \eqref{eq:atfirstmap} by applying the functors $D^{\Delta}(R)\to D^{\Delta}(\mathcal{O}_{\eqtopos{W}})\to D(\eqtopos{W})$ is naturally identified with the pullback $\Delta^*\at_{E_{X_{\bullet}}, X_{\bullet}/Y_{\bullet}}$ of the Atiyah class for the map of topoi $X_{\bullet}\to Y_{\bullet}$. Thus taking the cone, we see that \eqref{eq:atsecmap} is identified with $\at_{E_{X_{\bullet}}, X_{\bullet}/Y_{\bullet}}[1]$. From the commutative diagram of topoi 
	
	\begin{equation*}
		\begin{tikzcd}
			X_{\bullet, \operatorname{lis-et}}\ar[r,"\pi_X"]\ar[d,"\epsilon"]& X_{\operatorname{lis-et}}\ar[d,"\epsilon"] \\
			X_{\bullet,\operatorname{et}}\ar[r] & X_{\operatorname{et}}
		\end{tikzcd}
	\end{equation*}
	we see that $\eta_{X*}\at_{E_{X_{\bullet}}, X_{\bullet}/Y_{\bullet}}$ is naturally identified with the pullback to the lisse-\'etale site of the usual Atiyah class for the map $X_{\operatorname{et}}\to Y_{\operatorname{et}}$, from which the result follows. 
\end{proof}

As a first application of our definition, we compute the Atiyah class of the universal vector bundle on $BGL_n$ over a chosen base field $k$.
\begin{example}
	Let $\mathcal{V}$ denote the universal rank $n$ locally free sheaf on $BGL_n$, which is the sheaf of sections of the vector bundle associated to the universal principal $GL_n$-bundle. 
	We have the $2$-cartesian diagram, where the map $x$ from $\Spec k$ is the one corresponding to the trivial $GL_n$-torsor 
	\begin{equation*}
		\begin{tikzcd}
			GL_n\ar[r,"t"]\ar[d,"s"]& \Spec k\ar[d,"x"] \\
			\Spec k\ar[r,"x"]\ar[ur, Rightarrow,"\rho"] & BGL_n.
		\end{tikzcd}
	\end{equation*}

	There is a natural trivialization of the pullback $\mathcal{V}_{\Spec k}$. Then for every $T$-valued point $g=(g_{i,j})$ on $GL_n$, the $2$-morphism $\rho$ indicated in the diagram induces a natural pullback map $\rho^*: t^*\mathcal{V}_{\Spec k}|_T\xrightarrow{\sim}s^*\mathcal{V}_{\Spec k}|_T$. With respect to the given trivialization of $\mathcal{V}$ over $\Spec k$, this morphism is just given by $g^{-1}:\mathcal{O}_T^{\oplus n} \to \mathcal{O}_T^{\oplus n}$.
	
	We now compute the Atiyah class of $\mathcal{V}$ using Construction \ref{constr:atiyah}, where we chose $X=Y=\Spec k$, so that we get $W=GL_n$. More precisely, we will calculate the restriction of $\at_{\mathcal{V}}$ to the etale site of $W$, so that we don't have to consider the associated strictly simplicial algebraic spaces. Thus, here we let $\eqtopos{W}$ and $\eqtopos{Y}$ be the topos associated to the diagrams 
	\[W\underset{t}{\overset{s}{\rightrightarrows}}X \mbox{ and } {Y}\rightrightarrows Y\]
	respectively and $\eqtopos{h}:\eqtopos{W}\to \eqtopos{Y}$ the natural map. and let $\mathcal{V}_{\eqtopos{W}}$ denote the locally free sheaf on $\eqtopos{W}$ obtained by pullback. First, to calculate $\at_{\mathcal{V}_{\eqtopos{W}},\eqtopos{W}/\eqtopos{Y}}$, we do not need to take a simplicial resolution of $\mathcal{O}_{\eqtopos{W}}$ over $\eqtopos{h}^{-1}\mathcal{O}_{\eqtopos{Y}}$ since the map $\eqtopos{W}\to \eqtopos{Y}$ is composed of smooth maps of algebraic spaces. Moreover, since $\coneop$ preserves mapping cones of complexes, we may calculate $\coneop(\at_{\mathcal{V}_{\eqtopos{W}},\eqtopos{W}/\eqtopos{Y}})$ by first applying $\coneop$ to the sequence of principal parts  $P_{\eqtopos{W}/\eqtopos{Y}}(\mathcal{V}_{\eqtopos{W}})$ and only then taking connecting homomorphisms.
	Writing out the pullback maps in the sequence of principal parts $P_{\eqtopos{W}/\eqtopos{Y}}(\mathcal{V}_{\eqtopos{W}})$ gives the following diagram on $W$, where we use that $\Omega_{X/Y}=0$:
	\begin{equation*}
		\begin{tikzcd}
			0\ar[r]& 0 \ar[r]\ar[d]& s^*\mathcal{V}_{X}\ar[r]\ar[d]& s^*\mathcal{V}_X\ar[r]\ar[d]& 0 \\
			0\ar[r]& \Omega_{W/Y}\otimes \mathcal{V}_W \ar[r]&P_{W/Y}(\mathcal{V}_{W}) \ar[r]& \mathcal{V}_W\ar[r]& 0 \\
			0\ar[r]& 0 \ar[r]\ar[u]& t^*\mathcal{V}_{X}\ar[r]\ar[u]& t^*\mathcal{V}_X\ar[r]\ar[u]& 0
		\end{tikzcd}
	\end{equation*}
	Since $\mathcal{V}_W$ is pulled back from $X=Y=\Spec k$ via $s$, we have a natural splitting $P_{W/Y}(\mathcal{V}_W)=\Omega_{W/Y}\otimes s^*\mathcal{V}_{X}\oplus s^*\mathcal{V}_X$. With respect to this splitting, one calculates that the induced pullback map along $t$ is given by:
	\[t^*\mathcal{V}_{X} \xrightarrow{(d\rho, \rho)} \Omega_{W/Y}\otimes s^*\mathcal{V}_{X}\oplus s^*\mathcal{V}_X. \]
	It follows, that after applying $\coneop$, the resulting exact sequence has the following exact subsequence, and the termwise inclusions give quasi-isomorphisms:
	\[0\to \Omega_{W/Y}\otimes s^*\mathcal{V}_X\to \operatorname{Cone}(t^*\mathcal{V}_X\xrightarrow{d\rho}\Omega_{W/Y}\otimes s^*\mathcal{V}_X) \to t^*\mathcal{V}_X[1]\to 0. \]
	It is a general fact, that for an exact sequence of this form, the connecting map is canonically quasi-isomorphic to minus the shift of the map the cone is taken over for the middle term, i.e. in our case, this gives 
	\[-d\rho[1]: t^*\mathcal{V}_X[1]\to \Omega_{W/Y}\otimes s^*\mathcal{V}_X[1].\] 
	In particular, we find that there is a natural isomorphism $s^*x^*L_{BGL_n}[1]\simeq \Omega_{W/Y}$, and that using this identification and the trivialization of $\mathcal{V}_X$, we have that $s^*x^*\at_{\mathcal{V}}$ is given by
	\begin{align}\label{eq:atgln}
		\mathcal{O}_{GL_n}\xrightarrow{ d(T^{-1})\circ T} \Omega_{GL_n}\otimes \mathcal{O}_{GL_n}^{\oplus n},
	\end{align}
	where $T=(T_{ij})$ is the universal matrix on $GL_n$. 
	In particular, for $n=1$, we have 
\end{example}

\begin{example}\label{ex:trdetgln}
	By taking the trace in \eqref{eq:atgln}, we obtain the section 
	\[\mathcal{O}_{GL_n}\to \Omega_{GL_n}\]
	given by $\operatorname{tr}(d(T^{-1})T)$. This is equal to 
	\[d(\det T^{-1}) \det T.\]
	In particular, for $n=1$, we have $T=(t)$ as a $1\times 1$ matrix, and the pullback of $\at_{\mathcal{L}}\otimes \mathcal{L}^{-1}$ -- with $\mathcal{L}$ the universal rank $1$ sheaf -- to $\mathbb{G}_m=GL_1$ is given by
	\[\mathcal{O}_{\mathbb{G}_m}\to \Omega_{\mathbb{G}_m},\]
	sending $1$ to $d(1/t) t=-dt/t$. 
	We have the natural map $DET:BGL_n\to B\mathbb{G}_m$ given by associating to a locally free sheaf its determinant line bundle. By taking fiber products with a point, this induces a map $GL_n\to \mathbb{G}_m$, (which is again just the determinant) whose pullback map on differentials is given by $dt\mapsto d(\det T)$.
	In particular, we have $td(1/t)\mapsto \det T d(\det T^{-1})$. 
	This shows that we have a natural identification $\operatorname{tr}(s^*x^*\at_{\mathcal{V}})= s^*x^*(\at_{\det \mathcal{V}}\otimes \det \mathcal{V}^{-1})$. One can show that this isomorphism is part of a descent datum, and thus we have 
	\[\operatorname{tr}(\at_{\mathcal{V}})= \at_{\det \mathcal{V}}\otimes \det(\mathcal{V})^{-1}.\]
\end{example}

\subsection{Construction of the reduced Atiyah class}
\begin{construction}\label{constr:redat}
	
	Let $f:\mathcal{X}\to \mathcal{Y}$ be a morphism of algebraic stacks. Let $E\in D^-_{qcoh}(\mathcal{Y})$ and let $E_{\mathcal{X}}:= f^*E$. Let $F \to E_{\mathcal{X}}\to G \xrightarrow{+1}$ be a distinguished triangle in $D^-_{qcoh}(\mathcal{X})$ such that $R\Hom^{-1}(F, G) = 0$. 
	We let $X_{\bullet}, Y_{\bullet}$ and $W_{\bullet}$ be as in Situation \ref{sit:wxtopoi} and let $R=P_{h^{-1}\mathcal{O}_{\eqtopos{Y}}}(\mathcal{O}_{\eqtopos{W}})$. Let $F_{\eqtopos{W}}, E_{\eqtopos{W}}$ and $G_{\eqtopos{W}}$ denote the $\mathcal{O}_{\eqtopos{W}}$-modules induced by $E_{\mathcal{X}}$ and $G$ respectively, and let $E_{\eqtopos{Y}}$ denote the $\mathcal{O}_{\eqtopos{Y}}$-module induced by $E$. 
	Then, by the construction in \S \ref{subsec:redatiy}, and using Remark \ref{rem:redatgen}, we obtain a morphism 
	\[\overline{\at}_{E_{\eqtopos{Y}} }:F_{\eqtopos{W}}\to L_{\eqtopos{W}/\eqtopos{Y}}\otimes G_{\eqtopos{W}}.\]  
	Applying the sequence of maps of \eqref{diag:hugecommdiag}, we get 
	\[\overline{\at}_{E,\mathcal{X}/\mathcal{Y},G}:F\to L_{\mathcal{X}/\mathcal{Y}}\otimes G.\]  
	We also write $\overline{\at}_E$ if the rest of the data is clear.	
\end{construction}

\begin{proposition}
	The reduced Atiyah class is independent of the choice of cover $X/Y$. 
\end{proposition}
\begin{proof}
	This is similar to Lemma \ref{lem:functoriality},  and left to the reader (and slightly simpler since $E$ and $G$ are assumed to be quasicoherent sheaves, rather than more general complexes). 
\end{proof}
\begin{proposition}
	If $\mathcal{X}$ and $\mathcal{Y}$ are algebraic spaces, the reduced Atiyah class of Construction \ref{constr:redat} agrees with the one for ringed topoi defined in \S \ref{subsec:redatiy}. 
\end{proposition}
\begin{proof}
	The proof is analogous to the proof of Proposition \ref{prop:atcomp} and left to the reader.
\end{proof}

\begin{corollary}
	The following triangles commute
	\begin{enumerate}[label = \arabic*)]
		\item 
		\begin{equation*}
			\begin{tikzcd}
				F\ar[r,"-\overline{\at}_E"]\ar[dr,"\at_F"']& L_{\mathcal{X}/\mathcal{Y}}\otimes G\ar[d] \\
				& L_{\mathcal{X}/\mathcal{Y}}[1]\otimes F,
			\end{tikzcd}
		\end{equation*}
		
		\item 
		\begin{equation*}
			\begin{tikzcd}
				G\ar[d]\ar[dr,"{\at_G}"]&  \\
				F[1]\ar[r,"{\overline{\at}_E[1]}"']& L_{\mathcal{X}/\mathcal{Y}}[1]\otimes G
			\end{tikzcd}
		\end{equation*}
	\end{enumerate} 
	In both cases, the vertical morphisms are induced from the connecting map $G\to F[1]$ of the given exact sequence.	
\end{corollary}
\begin{proof}
	From Remark \ref{rem:redatcomptopos}, we get commutative triangles for the classical Atiyah classes $\at_{E_{\eqtopos{W}},\eqtopos{W}/\eqtopos{Y},G_{\eqtopos{W}} }$ and $\overline{\at}_{E_{\eqtopos{Y}}, \eqtopos{W}/\eqtopos{Y},G_{\eqtopos{W}}}$. The result follows by passing to $D(\mathcal{X})$.
\end{proof}
\subsection{Construction of the Atiyah class for an exact sequence}
Let $f:\mathcal{X}\to \mathcal{Y}$ be a morphism of algebraic stacks and let $0\to F\to E\to G\to 0$ be an exact sequence of quasicoherent sheaves on $\mathcal{X}$, which we denote by $\underline{E}$. Assume that $F,E$ and $G$ are dualizable as objects of the derived category of $\mathcal{X}$. We let $\eqtopos{X}$, $\eqtopos{Y}$ and $\eqtopos{W}$ be as in Situation \ref{sit:wxtopoi}. We also let $R=P_{h^{-1}\mathcal{O}_{\eqtopos{Y}}}\mathcal{O}_{\eqtopos{W}}$, and we let $F_{\eqtopos{W}}, E_{\eqtopos{W}}$ and $G_{\eqtopos{W}}$ denote the sheaves on $\eqtopos{W}$ induced by $F,E,G$ respectively. Let 
\[\alpha:\frac{E_{\eqtopos{W}}\otimes E_{\eqtopos{W}}^{\vee}}{F_{\eqtopos{W}}\otimes G^{\vee}_{\eqtopos{W}}}[-1]\to L_{\eqtopos{W}/\eqtopos{Y  }} \] be the Atiyah class of the exact sequence $\underline{E}_{\eqtopos{W}}$ with respect to the morphism of topoi $\eqtopos{W}\to\eqtopos{Y}$. This is a morphism in $D(\eqtopos{W})$. By taking cones, shifting, and applying $\eta_{W*}$, we obtain a morphism 
\[\at_{\underline{E},\mathcal{X}/\mathcal{Y}}:\frac{E\otimes E^{\vee}}{F\otimes G^{\vee}}[-1]\to L_{\mathcal{X}/\mathcal{Y}},\]
which we call the Atiyah class of the exact sequence $\underline{E}$. We also write $\at_{\underline{E}}$ if the morphism $f:\mathcal{X}\to \mathcal{Y}$ is understood. 

From the construction and Corollary \ref{cor:atexsequsat}, we conclude
\begin{corollary}\label{cor:atexseqatcomp}
	The morphism $\at_{\underline{E}}$ is compatible with the usual Atiyah class, i.e. the diagram 
	\begin{equation*}
		\begin{tikzcd}
			E^{\vee}\otimes E[-1]\ar[d]\ar[dr,"{\at'_E}"]&  \\
			\frac{E^{\vee}\otimes E}{G^{\vee}\otimes F}[-1]\ar[r,"\at_{\underline{E}}"] & L_{\mathcal{X}/\mathcal{Y}}. 
		\end{tikzcd}
	\end{equation*}
	commutes.
\end{corollary}
Similarly, we conclude immediately from Proposition \ref{prop:redatexseqat}: 
\begin{corollary}\label{cor:redatexseqatstack}
	We have a commutative diagram 
	\begin{equation*}
		\begin{tikzcd}
			\frac{E^{\vee}\otimes E}{G^{\vee}\otimes F}[-1]\ar[r,"-"]\ar[dr, "\at_{\underline{E}}"']& G^{\vee}\otimes F\ar[d, "\overline{\at}_E'"] \\
			& L_{\mathcal{X}/\mathcal{Y}},
		\end{tikzcd}
	\end{equation*}
	where the horizontal map is the natural connecting homomorphism.
\end{corollary}

\section{Properties}\label{sec:properties}
\subsection{Tensor compatibility of the Atiyah class.}

Let $A$ be a ring, and $B$ an $A$-algebra. For a $B$-module $M$ let $\underline{P}^1_{B/A}(M)$ denote the exact sequence 
\[0\to \Omega_{B/A}\otimes_B M\to P_{B/A}^1(M)\to M\to 0.\]
We will also write $\underline{P}^1_{B/A}(M)$ to denote the corresponding element of $\Ext^1_B(M,\Omega_{B/A}\otimes_B M)$.

\begin{lemma}\label{lem:tensormods}
	Let $M$ and $N$ be $B$-modules and suppose that $M$ is flat. Then in $\Ext^1_B(M\otimes_B N, \Omega_{B/A}\otimes_B M\otimes_B N)$ we have the following equality:
	\[\underline{P}_{B/A}(M)\otimes_B N+M\otimes_B \underline{P}_{B/A}(N)=\underline{P}_{B/A}(M\otimes N).\]
	Here we regard $M\otimes_B \underline{P}_{B/A}(N)$ as an extension of $M\otimes N$ by $ \Omega_{B/A}\otimes_B M\otimes_B N$ via the symmetry isomorphism of the tensor product. 
	
\end{lemma}
\begin{proof}
	Recall that $P^1_{B/A}(M) = B\otimes_A M/(I_{\Delta}^{2}B\otimes_A M)$ , where $I_{\Delta}$ is the kernel of the map of $A$-algebras $B\otimes_A B\to B$; the $B$-module structure on $P^1_{B/A}(M)$ is given by action on the left side of the tensor product. Therefore we have natural isomorphisms $P_{B/A}(M)\otimes_B N \simeq N\otimes_B P_{B/A}(M)\simeq N\otimes_A M/I_{N,M}$, where $I_{N,M}:=I_{\Delta}^2 (N\otimes_A M)$, and similarly with $N$ and $M$ reversed. 
	
	We then have that the left hand side of the equality in the lemma is represented by the Baer sum of the following two exact sequences:
	\begin{equation}\label{eq:tensorseq1}
		0\to \Omega_{B/A}\otimes_B M\otimes_B  N\xrightarrow{j_1} N\otimes_A M/ I_{N,M}^2 \xrightarrow{p_1} M\otimes_B N \to 0, 
	\end{equation}
	\begin{equation}\label{eq:tensorseq2}
		0\to \Omega_{B/A}\otimes_B M\otimes_B  N\xrightarrow{j_2} M\otimes_A N/ I_{M,N}^2 \xrightarrow{p_2} M\otimes_B N \to 0, 
	\end{equation}
	
	where 
	\begin{align*}
		j_1(db\otimes m\otimes n) &= [bn\otimes m - n\otimes bm],\\ 
		p_1([n\otimes m])&=m\otimes n, \mbox{ and }\\ 
		j_2(db\otimes m\otimes n) &= [bm\otimes n- m\otimes bn],\\ 
		p_2([m\otimes n])&= m\otimes n.
	\end{align*}
	
	We make the definition of the Baer sum explicit: We have submodules 
	\[Q\subset R \subset N\otimes_A M/I_{N,M}^2\oplus M\otimes N/I_{M,N}^2,\]
	given by 
	\[R=\{(a, b) \,\mid\,  p_1(a)=p_2(b) \};\]
	\[Q=\{(j_1(x),-j_2(x))\, \mid \, x\in \Omega_{B/A}\otimes_B M\otimes_B N\}.\]

	Then the Baer sum of \eqref{eq:tensorseq1} and \eqref{eq:tensorseq2} is given by 
	\[0\to \Omega_{B/A} \otimes_B M\otimes_B N \xrightarrow{j_3} R/Q\xrightarrow{p_3} M\otimes_B N\to 0, \]
	where the maps are defined by $p_3((a,b)+ Q)=p_1(a)$, and $j_3(x)= (j_1(x),0)+Q$.  
	
	Note that we also have (by exactness of the sequences \eqref{eq:tensorseq1} and \eqref{eq:tensorseq2}) 
	\begin{equation}\label{eq:Qdefalt}
		Q = \{(\sum [n_i\otimes m_i], \sum [m_i\otimes n_i])\,\mid\, \sum m_i\otimes n_i =0 \mbox{ in } M\otimes_B N\}.
	\end{equation}
	
	We claim that this extension is isomorphic to $\underline{P}^1_{B/A}(M\otimes_B N)$. We first construct a map of $B$-modules $P_{B/A}^1(M\otimes_B N)\to R/Q$. By the universal property of the module of principal parts, this is equivalent to giving an $A$-linear degree one differential operator $M\otimes_B N\to R/Q$. We define
	\begin{align*}
		D: M\otimes_B N &\to R/Q,\\
		\sum m_i\otimes n_i &\mapsto ([\sum n_i\otimes m_i], [\sum m_i\otimes n_i])+Q
	\end{align*}
	This is well-defined, since if $\sum m_i\otimes n_i=\sum m_j'\otimes n_j'$ in $M\otimes_B N$, then the difference 
	\[([\sum n_i\otimes m_i -\sum n_j'\otimes m_j'], [\sum m_i\otimes n_i-\sum m_j'\otimes n_j'])\]
	lies in $Q$ due to \eqref{eq:Qdefalt}. 
	The map $D$ is clearly $A$-linear, and a first order differential operator, since for any $b\in B$, we have 
	\[D(b (m\otimes n))-bD(m\otimes n) = ([n\otimes bm - bn\otimes m],0)+Q,\]
	and thus the map $x\mapsto D(bx)-bD(x)$ is $B$-linear. Thus, we have the corresponding map of $B$-modules
	\begin{align*}
		P^1_{B/A}(N\otimes_B M)&\to R/Q\\
		[b\otimes x]&\mapsto b D(x). 
	\end{align*} 
	A straightforward explicit computation shows that the diagram
	\begin{equation*}
		\begin{tikzcd}[cramped]\label{eq:exseqdiag}
			0\ar[r]&\Omega_{B/A}\otimes_B N\otimes_B M\ar[r]\ar[d,equal]&P_{B/A}^1(N\otimes_B M) \ar[r]\ar[d]& N\otimes_B M\ar[r]\ar[d,equals]&0 \\
			0\ar[r] & \Omega_{B/A}\otimes_B N\otimes_B M\ar[r] & R/Q \ar[r]& N\otimes_B M\ar[r] &0
		\end{tikzcd}
	\end{equation*}
	commutes, which finishes the proof.  
\end{proof}
Since taking tensor products, K\"ahler differentials, modules of principal parts and Baer sums commute with sheafification and are suitably functorial, we have as a consequence
\begin{corollary}\label{cor:tensormods}
	Lemma \ref{lem:tensormods} holds when $A$ and $B$ are rings in a topos.
\end{corollary}
Using Lemma \ref{lem:tensormods}, we can prove the tensor compatibility of the Atiyah class.
Regard $E\otimes {\at}_F$ as a map $E\otimes F\to L_{\mathcal{X}/\mathcal{Y}}\otimes E\otimes F [1]$ via the composition 
\[E\otimes F\xrightarrow{E\otimes \at_F} E\otimes (L_{\mathcal{X}/\mathcal{Y}}\otimes F [1])\simeq E\otimes L_{\mathcal{X}/\mathcal{Y}} \otimes F [1]\simeq L_{\mathcal{X}/\mathcal{Y}}\otimes E\otimes F[1].\]
Here, the second map is the map defining the triangulated structure on $E\otimes -$, and the third map is the symmetry isomorphism exchanging $E$ and $L_{\mathcal{X}/\mathcal{Y}}$.
\begin{proposition}\label{prop:tensorcomp}
	Let $\mathcal{X}\to \mathcal{Y}$ be a morphism of algebraic stacks, and $E,F\in D^{\leq 0}_{qcoh}(\mathcal{O}_{\mathcal{X}})$. Then we have the equality
	\[\at_{E\otimes F}= \at_{E}\otimes  F+E\otimes \at_{F}.\] 
\end{proposition}
\begin{proof}
	We use the setup of Construction \ref{constr:atiyah}.
	Without loss of generality, we may assume that $E$ is represented by a complex of flat $\mathcal{O}_{\mathcal{X}}$-modules, so in particular we may choose $E\otimes F\in D^{\leq 0}_{qcoh}(\mathcal{O}_{\mathcal{X}})$ to be represented by the usual tensor product of the complexes $E$ and $F$. 
	Let $E_R\to E_{\eqtopos{W}}$ be a flat resolution of the $R$-module $E_{\eqtopos{W}}$.  By the construction of the Atiyah class in \ref{sec:constr-at}, we may use the exact sequences 
	\[\underline{P}^1_{R/\eqtopos{h}^{-1}\mathcal{O}_{\eqtopos{Y}}}(E_R) \mbox{ and } \underline{P}^1_{R/\eqtopos{h}^{-1}\mathcal{O}_{\eqtopos{Y}}}(E_R\otimes_R F_{\eqtopos{W}})\] to compute $\at_{E}$ and $\at_{E\otimes F}$ respectively. By Corollary \ref{cor:tensormods}, we have an equality of extensions of $R$-modules
	\[\underline{P}^1_{R/\eqtopos{h}^{-1}\mathcal{O}_{\eqtopos{Y}}}(E_R)\otimes_R F_{\eqtopos{W}}+ E_R\otimes_R\underline{P}^1_{R/\eqtopos{h}^{-1}\mathcal{O}_{\eqtopos{Y}}}(F_{\eqtopos{W}}) = \underline{P}^1_{R/\eqtopos{h}^{-1}\mathcal{O}_{\eqtopos{Y}}}(E_R\otimes_R F_{\eqtopos{W}}).\]
	After taking connecting maps in $D^{\Delta}(R)$, and using that extensions of scalars and the Dold-Kan correspondence are compatible with the symmetry isomorphisms of the derived tensor product, we get an equality of maps
	\[\at_{E_{\eqtopos{W}}\otimes F_{\eqtopos{W}}}=\at_{E_{\eqtopos{W}}}\otimes F_{\eqtopos{W}} + E_{\eqtopos{W}}\otimes\at_{F_{\eqtopos{W}}}.\]
	
	By Lemma \ref{lem:tensorcone}, this implies the result after passing to $D^{\leq 0}(\mathcal{O}_{W_{\bullet}})$ and then to $D^{\leq 0}(\mathcal{X})$.
\end{proof}
If we assume that one of the Atiyah classes of $E$ or $F$ vanish, this simplifies to
\begin{corollary}
	Suppose that $\at_F=0$ (respectively that $\at_E=0$), then we have $\at_{E\otimes F}=\at_E\otimes F$ (respectively $\at_{E\otimes F}=E\otimes{\at_F}$) after identifying the targets using symmetry of the tensor product. 
\end{corollary}

\begin{corollary}[Shift invariance]\label{cor:shiftinv}
	For $E\in D^{\leq 0}_{qcoh}(\mathcal{X})$, the following diagram commutes
	\begin{equation*}
		\begin{tikzcd}
			E[1]\ar[r,"{\at_{E[1]}}"]\ar[dr,swap,"{\at_E[1]}"]&L_{\mathcal{X}/\mathcal{Y}}[1]\otimes (E[1]) \ar[d,"{\tau[1]}"] \\
			\,& L_{\mathcal{X}/\mathcal{Y}}[1]\otimes E [1],
		\end{tikzcd}
	\end{equation*}
	where $\tau $ is the map defining the triangulated structure on the tensor product functor $L_{\mathcal{X}/\mathcal{Y}}[1]\otimes -$ (see \cite[\href{https://stacks.math.columbia.edu/tag/0G6A}{Tag 0G6A} and \href{https://stacks.math.columbia.edu/tag/0G6E}{Tag 0G6E}]{St20}). 
\end{corollary}
\begin{proof}
	This follows from Proposition \ref{prop:tensorcomp} applied to $F\otimes E$ with $F=\mathcal{O}_{\mathcal{X}}[1]$ and noting that $\at_{F}=0$, since $F$ is pulled back from $\mathcal{Y}$.
\end{proof}

\begin{remark}
	Using Corollary \ref{cor:shiftinv}, it is straightforward to extend the definition of $\at_E$ to any $E\in D^-_{qcoh}(\mathcal{X})$.
\end{remark}
Another consequence is the following:

\begin{corollary}\label{cor:atmapoftri}
	The following diagram is a morphism of exact triangles
	\begin{equation*}
		\begin{tikzcd}
			F\ar[r]\ar[d,"{\at_F}"]& E\ar[d,"{\at_E}"]\ar[r]&G\ar[r]\ar[d,"{\at_G}"] &F[1]\ar[d,"{\at_F[1]}"] \\
			L_{\mathcal{X}/\mathcal{Y}}[1]\otimes F\ar[r] &L_{\mathcal{X}/\mathcal{Y}}[1]\otimes E \ar[r]&L_{\mathcal{X}/\mathcal{Y}}[1]\otimes G\ar[r]&L_{\mathcal{X}/\mathcal{Y}}[1]\otimes F[1].
		\end{tikzcd}
	\end{equation*}
	Here the lower row is obtained by applying the triangulated functor $L_{\mathcal{X}/\mathcal{Y}}[1]\otimes -$ to the upper row.
\end{corollary}
\begin{proof}
	The commutativity of the first two squares follows from the functoriality of the Atiyah class (Lemma \ref{lem:functoriallem}), while the commutativity of the last square follows from functoriality combined with Corollary \ref{cor:shiftinv}.
\end{proof}

\subsection{Compatibility with traces}
We let $f:\mathcal{X}\to \mathcal{Y}$ be a morphism of algebraic stacks, and consider Atiyah classes with respect to this morphism. In Proposition \ref{prop:attrace} and Corollary \ref{cor:atdettr}, we make global boundedness assumptions, but the results likely hold for arbitrary perfect complexes, (cf. \cite{LiOl}).
\begin{proposition}\label{prop:attrace}
	Let $F\to E\to G\to F[1]$ be a distinguished triangle of perfect complexes in $D^b_{qcoh}(\mathcal{X})$  and assume each of $E,F,G$ has finite Tor-amplitude. Then 
	\[\operatorname{tr}(\at_E)= \operatorname{tr}(\at_F)+\operatorname{tr}(\at_G).\]
\end{proposition}
\begin{proof}
	By the shift invariance of the Atiyah class and Lemma \ref{lem:trtensor}, we may assume that the given distinguished triangle is represented by a short exact sequence of complexes
	\[0\to F\to E\to G\to 0,\]
	where $F,E,G$ have flat components and lie in $C^{\leq 0}(\mathcal{X})$.
	
	We then use the setup of Situation \ref{sit:exseqdualized} with $L=L_{\mathcal{X}/\mathcal{Y}}[1]$ and $R=\mathcal{O}_{\mathcal{X}}$. (this is justified by Corollary \ref{cor:atmapoftri}). Then by Proposition \ref{prop:maytrace}, we need to find a morphism 
	\[\frac{E\otimes E^{\vee}}{F\otimes G^{\vee}}\to L_{\mathcal{X}/\mathcal{Y}}\otimes \frac{E\otimes E^{\vee}}{F\otimes G^{\vee}}\]
	making the diagram there commute. By the assumption on Tor-dimension, each term of the sequence 
	\[0\to G^{\vee}\to E^{\vee}\to F^{\vee}\to 0\]
	lies in $D^-_{qcoh}(\mathcal{X})$. We can therefore find an integer $N$, such that the natural maps from the good truncation $\tau_{\leq N}$ are quasi-isomorphisms for each term.  Set $\overline{E}:=(\tau_{\leq N}E)[N]$, and define $\overline{F},\overline{G}$ in the same way, so that we obtain an exact sequence in $C^{\leq 0}(\mathcal{O}_{\mathcal{X}})$. 
	We also let $\overline{E}_{\eqtopos{W}}, \overline{F}_{\eqtopos{W}}$ and $\overline{{G}}_{\eqtopos{W}}$ denote the induced complexes of $\mathcal{O}_{\eqtopos{W}}$-modules, which we regard as simplicial modules via the Dold--Kan correspondence and as $R$-modules by restriction of scalars. 
	Choose resolutions of $F_{\eqtopos{W}},E_{\eqtopos{W}}$ and $G_{\eqtopos{W}}$ by termwise flat $R$-modules, so that we have a morphism of exact sequences
	\begin{equation*}
		\begin{tikzcd}
			0\ar[r]& F_R\ar[r]\ar[d]& E_R\ar[r]\ar[d]& G_R\ar[r]\ar[d]&0  \\
			0\ar[r] & F_{\eqtopos{W}}\ar[r]& E_{\eqtopos{W}}\ar[r]& G_{\eqtopos{W}}\ar[r]& 0.
		\end{tikzcd}
	\end{equation*}
	We then obtain a natural morphism of exact sequences of $R$-modules
	\[\underline{P}^1_{R/h^{-1}\mathcal{O}_{\eqtopos{Y}}}(F_{R})\otimes_R \overline{G}_{\eqtopos{W}}\to \underline{P}^1_{R/h^{-1}\mathcal{O}_{\eqtopos{Y}}}(E_{R})\otimes_R \overline{E}.\]
	This is termwise an injection of complexes, so we get a quotient exact sequence $\underline{S}$ of the form 
	\[0\to \Omega^1_{R/h^{-1}\mathcal{O}_{\eqtopos{Y}}}\otimes_R \frac{E_{R}\otimes_R \overline{E}_{\eqtopos{W}}}{F_R\otimes_R\overline{G}_{\eqtopos{W}}}\to ***\to \frac{E_R\otimes \overline{E}_{\eqtopos{W}}}{F_R\otimes \overline{G}_{\eqtopos{W}}}\to 0.\] 
	Taking the connecting map of this sequence and passing back to $D(\eqtopos{W})$, we obtain a morphism $\alpha$ making the following square commute
	\begin{equation*}
		\begin{tikzcd}
			E_{\eqtopos{W}}\otimes \overline{E}_{\eqtopos{W}}\ar[r]\ar[d,"\at_{E_{\eqtopos{W}}}\otimes \overline{E}_{\eqtopos{W}}"]&\frac{	E_{\eqtopos{W}}\otimes \overline{E}_{\eqtopos{W}}}{	F_{\eqtopos{W}}\otimes \overline{G}_{\eqtopos{W}}} \ar[d,"\alpha"] \\
			L_{\eqtopos{W}/\eqtopos{Y}}[1]\otimes E_{\eqtopos{W}}\otimes \overline{E}_{\eqtopos{W}}\ar[r] & L_{\eqtopos{W}/\eqtopos{Y}}[1]\otimes \frac{	E_{\eqtopos{W}}\otimes \overline{E}_{\eqtopos{W}}}{	F_{\eqtopos{W}}\otimes \overline{G}_{\eqtopos{W}}}.
		\end{tikzcd}
	\end{equation*}
	We claim that the diagram
	\begin{equation}\label{diag:Fcommutes}
		\begin{tikzcd}
			F_{\eqtopos{W}}\otimes \overline{F}_{\eqtopos{W}}\ar[r]\ar[d,"\at_{F_{\eqtopos{W}}}\otimes \overline{F}_{\eqtopos{W}}"]&\frac{	E_{\eqtopos{W}}\otimes \overline{E}_{\eqtopos{W}}}{	F_{\eqtopos{W}}\otimes \overline{G}_{\eqtopos{W}}} \ar[d,"\alpha"] \\
			L_{\eqtopos{W}/\eqtopos{Y}}[1]\otimes F_{\eqtopos{W}}\otimes \overline{F}_{\eqtopos{W}}\ar[r] & L_{\eqtopos{W}/\eqtopos{Y}}[1]\otimes \frac{	E_{\eqtopos{W}}\otimes \overline{E}_{\eqtopos{W}}}{	F_{\eqtopos{W}}\otimes \overline{G}_{\eqtopos{W}}}.
		\end{tikzcd}
	\end{equation}
	also commutes (and similarly with $G$ in place of $F$ in the left column). This follows from the observation that the exact sequence $\underline{P}^1_{R/\eqtopos{h}^{-1}\mathcal{O}(\eqtopos{Y})}(F_R)\otimes \overline{F}_{\eqtopos{W}}$ is isomorphic to the termwise cokernel of the map of exact sequences 
	\[\underline{P}^1_{R/h^{-1}\mathcal{O}_{\eqtopos{Y}}}(F_{R})\otimes_R \overline{G}_{\eqtopos{W}}\to \underline{P}^1_{R/h^{-1}\mathcal{O}_{\eqtopos{Y}}}(F_{R})\otimes_R \overline{E}_{\eqtopos{W}},\]
	and thus includes into $\underline{S}$. Then we get the commutativity \eqref{diag:Fcommutes} by taking connecting morphisms and passing to $D(\eqtopos{W})$. The argument for $G$ in place of $F$ goes similarly. 
	Applying the functor $\eta_{W*}\circ\coneop(\alpha) \circ [-1]$ then gives a morphism
	\[\frac{E\otimes E^{\vee}}{F\otimes G^{\vee}}[N]\to L_{\mathcal{X}/\mathcal{Y}}\otimes \frac{E\otimes E^{\vee}}{F\otimes G^{\vee}}[N].\]
	Shifting by $-N$ gives the desired map in Proposition \ref{prop:maytrace}. 
\end{proof}

\begin{corollary}\label{cor:atdettr}
	Let $E\in D^-_{qcoh}(\mathcal{O}_{\mathcal{X}})$ be an object that can be represented by a finite length complex of locally free sheaves. Then we have an equality of maps $\mathcal{O}_{\mathcal{X}}\to L_{\mathcal{X}/\mathcal{Y}}[1]$:
	\[\operatorname{tr}(\at_{E})=\frac{\at_{\det E}}{\det E}.\]
\end{corollary}
\begin{proof}
	We argue by induction on the number $k$ of nonzero components in a resolution by locally free sheaves. Without loss of generality, we may assume that $E$ is given by such a resolution. If $k=1$, the result follows from Example \ref{ex:trdetgln}. If $k>1$, we may write $k=k_1+k_2$ for positive integers $k_1,k_2$ and can take bad truncations of $E$ to get an exact sequence of complexes
	\[0\to F\to E\to G\to 0,\]
	where $F$ has $k_1$ nonzero components and $G$ has $k_2$, and they are all locally free. Then by Proposition \ref{prop:attrace} and the induction hypothesis, on one hand, we have
	\[\operatorname{tr}(\at_E) = \operatorname{tr}(\at_F)+\operatorname{tr}(\at_G)=\frac{\at_{\det F}}{\det F}+\frac{\at_{\det G}}{\det G}.\]
	On the other hand, the determinant is multiplicative in exact sequences of perfect complexes, which gives $\det E=\det F\otimes \det G$. Using the tensor compatibility Proposition \ref{prop:tensorcomp}, we find 
	\[\frac{\at_{\det E}}{\det E} = \frac{\at_{\det F}\otimes \det G + \det F\otimes \at_{\det G}}{\det F\otimes \det G}= \frac{\at_{\det F}}{\det F}+\frac{\at_{\det G}}{\det G}.\] 
\end{proof}		
\subsection{Compatibility of Atiyah class and reduced Atiyah class}\label{subsec:compatredat}
We prove Proposition \ref{prop:redatcomphard}.
Consider the setup of Construction \ref{constr:redat}, and let $\mathcal{Y}\to \mathcal{Z}$ be a further morphism of algebraic stacks. Choose a diagram 
\begin{equation*}
	\begin{tikzcd}
		X\ar[r]&\mathcal{X}_Y\ar[d]\ar[r]&\mathcal{X}_Z\ar[r]\ar[d]&\mathcal{X}\ar[d]\\
		&Y				   \ar[r]&\mathcal{Y}_Z\ar[r]\ar[d]& \mathcal{Y}\ar[d] \\
		& 						 &Z\ar[r] 				   & \mathcal{Z}
	\end{tikzcd}
\end{equation*}
in which all squares are cartesian, the horizontal morphisms are smooth and surjective, and in which $X,Y$ and $Z$ are algebraic spaces. We let $X_{\bullet}$, $Y_{\bullet}$ and $Z_{\bullet}$ be the strictly simplicial algebraic spaces associated to compositions along the horizontal rows respectively. We also let $V:=Y\times_{\mathcal{Y}_Z}Y$ with strictly simplicial algebraic space $V_{\bullet}$ associated to the morphism $V\to \mathcal{Y}$, and further $W:=X\times_{\mathcal{X}_Y} X$ and $\widetilde{W}:=X\times_{\mathcal{X}_Z} X$ with strictly simplicial algebraic spaces $W_{\bullet}$ and $\widetilde{W}_{\bullet}$ associated to the morphisms to $\mathcal{X}$. We have the following natural commutative diagram
\begin{equation*}
	\begin{tikzcd}			
		X_{\bullet}\ar[d]&\ar[l,"s_{\bullet}"']W_{\bullet}\ar[r,"t_{\bullet}"]\ar[d]& X_{\bullet}\ar[d] \\
		X_{\bullet}\ar[d] &\widetilde{W}_{\bullet} \ar[d]\ar[r,"{t_{\bullet}}"]\ar[l,"{s_{\bullet}}"'] & X_{\bullet}\ar[d]\\
		Y_{\bullet} & V_{\bullet}\ar[l,"{s_{\bullet}}"']\ar[r,"{t_{\bullet}}"]& Y_{\bullet}	
	\end{tikzcd}
\end{equation*}
Here, by abuse of notation, we use $s_{\bullet}$ and $t_{\bullet}$ to denote the morphisms given by (degreewise) projection to the first and second factor respectively.
We define diagrammatic topoi $\eqtopos{W}$, $\eqtopos{\widetilde{W}}$ and $\eqtopos{V}$  by the rows of this diagram and $\eqtopos{Y}$ and $\eqtopos{Z}$ associated to the constant diagram with values $Y_{\bullet}$ and $Z_{\bullet}$ respectively.  Then we have morphisms 
\[\eqtopos{W}\xrightarrow{\eqtopos{j}} \eqtopos{\widetilde{W}}\to \eqtopos{V}\to \eqtopos{Y}\to \eqtopos{Z}.\]
Denote the sheaves on either of these obtained by pulling back $E, F$ or $G$ respectively by a corresponding subscript. In particular, we have $E_{\eqtopos{Y}}$ on $Y$, which is Tor-independent with the morphism $\eqtopos{W}\to \eqtopos{Y}$, and we have an exact sequence 
\[0\to F_{\eqtopos{W}}\to E_{\eqtopos{W}}\to G_{\eqtopos{W}}\to 0.\]
By construction, the reduced Atiyah class $\overline{\at}_{E,\mathcal{X}/\mathcal{Y},G}$ is obtained from $\overline{\at}_{E_{\eqtopos{Y}},\eqtopos{W}/\eqtopos{Y},\eqtopos{G}}$ by applying $\coneop$ and descent to $D(\mathcal{X})$. Similarly, the Atiyah class $\at_{E,\mathcal{Y}/\mathcal{Z}}$ is obtained from $\at_{E_{\eqtopos{V}},\eqtopos{V}/\eqtopos{Z }}$ by applying $\coneop$ and passing to $D(\mathcal{Y})$.

Let $\eqtopos{\widetilde{f}}$ denote the morphism $\eqtopos{\widetilde{W}}\to \eqtopos{V}$. By Proposition \ref{prop:redatcompbasic}, we have the anti-commutative diagram
\begin{equation*}
	\begin{tikzcd}[column sep = large]
		F_{\eqtopos{\widetilde{W}}}\ar[rr,"{\overline{\at}_{E_{\eqtopos{V}},\eqtopos{\widetilde{W}}/\eqtopos{V},G_{\eqtopos{W}}}}"]\ar[d]&[5pt] & L_{\eqtopos{\widetilde{W}}/\eqtopos{V}}\otimes G_{\eqtopos{\widetilde{W}}}\ar[d] \\	 		
		E_{\eqtopos{\widetilde{W}}}\ar[r,"{\eqtopos{\widetilde{f}}^*\at_{E_{\eqtopos{V}},\eqtopos{V}/\eqtopos{Z}}}"] & \eqtopos{\widetilde{f}}^*L_{\eqtopos{V}/\eqtopos{Z}}[1]\otimes E_{\eqtopos{\widetilde{W}}}\ar[r]&\eqtopos{\widetilde{f}}^*L_{\eqtopos{V}/\eqtopos{Z}}[1]\otimes G_{\eqtopos{\widetilde{W}}}.
	\end{tikzcd}
\end{equation*}
We also have a morphism $\eqtopos{Y}\to \eqtopos{V}$ induced by the morphism $Y_{\bullet}\to V_{\bullet}$, given by the diagonal of $Y_n\times_{\mathcal{Y}_{Z_n}}Y_n$ in degree $n$. This fits into the commutative diagram
\begin{equation*}
	\begin{tikzcd}
		\eqtopos{W}\ar[r,"\eqtopos{j}"]\ar[d]& \eqtopos{\widetilde{W}}\ar[d] \\
		\eqtopos{Y}\ar[r] & \eqtopos{V}.
	\end{tikzcd}
\end{equation*}
Here, the horizontal maps are Tor-independent to $E_{\eqtopos{V}}$ and to $E_{\eqtopos{\widetilde{W}}}$ and $G_{\eqtopos{\widetilde{W}}}$ respectively, thus we can apply Corollary \ref{cor:redatpullback}. Moreover, the pullback $\eqtopos{j}^*L_{\eqtopos{\widetilde{W}}/\eqtopos{V}}\to L_{\eqtopos{W}/\eqtopos{Y}}$ is a quasi-isomorphism (this follows, since the morphisms $Y_{\bullet}\to V_{\bullet}\leftarrow \widetilde{W}_{\bullet}$ are Tor-independent, which can be checked degreewise).
In conclusion, by pulling back along $\eqtopos{j}$, we obtain an anti-commutative diagram in $D(\eqtopos{W})$:
\begin{equation*}
	\begin{tikzcd}[column sep = large]
		F_{\eqtopos{W}}\ar[rr,"{\overline{\at}_{E_{\eqtopos{Y}},\eqtopos{W}/\eqtopos{Y},G_{\eqtopos{W}}}}"]\ar[dd]&[5pt] & L_{\eqtopos{W}/\eqtopos{Y}}\otimes G_{\eqtopos{W}}\\
		& & \eqtopos{j}^*L_{\eqtopos{\widetilde{W}}/\eqtopos{V}}\otimes G_{\eqtopos{W}}\ar[u,"\sim"']\ar[d] \\	 		
		E_{\eqtopos{W}}\ar[r,"{\eqtopos{f}^*\at_{E_{\eqtopos{V}},\eqtopos{V}/\eqtopos{Z}}}"] & \eqtopos{f}^*L_{\eqtopos{V}/\eqtopos{Z}}[1]\otimes E_{\eqtopos{W}}\ar[r]&\eqtopos{f}^*L_{\eqtopos{V}/\eqtopos{Z}}[1]\otimes G_{\eqtopos{W}},
	\end{tikzcd}
\end{equation*}
where $\eqtopos{f}$ is the morphism $\eqtopos{W}\to \eqtopos{V}$.
By applying $\coneop$ and passing to $D(\mathcal{X})$, we conclude 
\begin{proposition}\label{prop:atclassredat}
	The diagram 
	\begin{equation*}
		\begin{tikzcd}
			F\ar[rr,"{\overline{\at}_{E,\mathcal{X}/ \mathcal{Y} ,G}}"]\ar[d]&[20 pt] & L_{\mathcal{X}/\mathcal{Y}}\otimes G\ar[d] \\
			E_{\mathcal{X}}\ar[r,"{Lf^*\at_{E, \mathcal{Y}  /\mathcal{Z}}}"] & Lf^*L_{\mathcal{Y}/\mathcal{Z}}[1]\otimes E\ar[r]&Lf^*L_{\mathcal{Y}/\mathcal{Z}}[1]\otimes G
		\end{tikzcd}
	\end{equation*}
	anti-commutes. 
\end{proposition} 
If $E,F,G$ are perfect complexes, this directly implies Proposition \ref{prop:redatcomphard}.   

\subsection{Compatibilities of the Atiyah class for an exact sequence}

\begin{proposition}\label{prop:atexseqfuncystack}
	The map $\at_{\underline{E},\mathcal{X}/\mathcal{Y}}$ is functorial in $\mathcal{Y}$, i.e. given a map $\mathcal{Y}\to \mathcal{Z}$, the composition 
	\[\frac{E^{\vee}\otimes E}{G^{\vee}\otimes F}[-1]\xrightarrow{\at_{\underline{E},\mathcal{X}/\mathcal{Z}}} L_{\mathcal{X}/\mathcal{Z}}\to L_{\mathcal{X}/\mathcal{Y}}  \] 
	equals $\at_{\underline{E},\mathcal{X}/\mathcal{Y}}$ (at least assuming we make the same choices of $N$ in the construction).
\end{proposition} 
\begin{proof}
	We use the setup of \S \ref{subsec:compatredat}. There we had the topoi 
	\[\eqtopos{W}\to \widetilde{\eqtopos{W}}\to \eqtopos{V} \to \eqtopos{Y}\to \eqtopos{Z}\]
	and the following morphisms of maps of topoi
	\[\eqtopos{W}/\eqtopos{Y}\to \widetilde{\eqtopos{W}}/\eqtopos{V}\to \widetilde{\eqtopos{W}}/\eqtopos{Z}\]
	whose cotangent complexes represent (after taking cones and shifting) $L_{\mathcal{X}/\mathcal{Y}}$, $L_{\mathcal{X}/\mathcal{Y}}$ and $L_{\mathcal{X}/\mathcal{Z}}$ respectively. The result now follows from Corollary \ref{cor:atexseqfuncy} and Lemma \ref{lem:atexseqfuncx}.
\end{proof}
\begin{proof}[Proof of Proposition \ref{prop:atexseqcomp}]
	The commutativity of the first square follows from the functoriality of the usual Atiyah class and Corollary \ref{cor:atexseqatcomp}. Commutativity of the second square follows from Proposition \ref{prop:atexseqfuncystack}, and Corollary \ref{cor:redatexseqatstack}. Finally, the commutativity of the square involving connecting morphisms is Proposition \ref{prop:atclassredat}. 
\end{proof}
\section{Deformation theoretic properties}
In this section, we prove Proposition \ref{prop:atisob}. The proof of Proposition \ref{prop:redatisob} is similar, but not addressed here. For details see \cite[\S 4]{Gill}.  In \S \S \ref{subsec:obstrcomp} - \ref{subsec:autcomp}, we use the following notation: $X$ is a smooth projective variety over a base field $k$, and $\mathcal{M}$ is an open substack of the moduli stack of coherent sheaves on $X$. 
\subsection{Deformations of morphisms to algebraic stacks}
Let $\mathcal{Y}$ be an algebraic stack over a base scheme $S$ and let $T$ be a scheme over $S$. Here we consider the problem of deforming maps from the scheme $T$ to $\mathcal{Y}$. As a special case of \cite[Theorem 1.5]{Olso3}, we have
\begin{theorem}\label{th:thdefalgstackmap}
	Let $g:T\to \mathcal{Y}$ be a morphism and let $j:T\hookrightarrow \overline{T}$ be a square zero extension of $T$ by a quasicoherent sheaf $I$. Then 
	\begin{enumerate}[label=\arabic*)]
		\item There is a natural obstruction class $\omega(g,\overline{T})\in \Ext^1(g^*L_{\mathcal{Y}},I)$ which vanishes if and only if there is an extension of $g$ to a morphism $\overline{g}:\overline{T}\to \mathcal{Y}$. 
		\item If an extension of $g$ to $\overline{T}$ exists, then the set of isomorphism classes of extensions naturally forms a torsor under $\Ext^0(g^*L_{\mathcal{Y}},I)$.
		\item For a fixed extension $\overline{g}$ of $g$, the set of automorphisms of $\overline{g}$ as an extension of $g$ is canonically isomorphic to $\Ext^{-1}(g^*L_{\mathcal{Y}}, I)$.
	\end{enumerate}
\end{theorem}
One can describe the characterizations in Theorem \ref{th:thdefalgstackmap} explicitly. 

\begin{remark}[Obstructions]\label{rem:obstrsmor}
	The morphism $g$ induces a natural map $g^*L_{\mathcal{Y}}\to L_T$. Similarly, $j$ induces a natural map $L_T\to L_{T/\overline{T}}$, and $L_{T/\overline{T}}$ is concentrated in degrees $\leq -1$ with $h^{-1}(L_{T/\overline{T}})$ naturally isomorphic to $I$. The obstruction class $\omega(g,\overline{T})$ is then given by the composition 
	\[g^*L_{\mathcal{Y}}\to L_{T/\overline{T}}\to I[1].\]
	This follows from the construction in \cite[4.8]{Olso3} and the construction of the obstruction class for topoi in \cite[III 2.2]{Ill}.	
\end{remark}

\begin{remark}[Deformations]\label{rem:corrdefs}
	For a given $g:T\to \mathcal{Y}$, let $\overline{T}$ be the trivial extension of $T$ by $I$, given by taking the structure sheaf $\mathcal{O}_T\oplus I$ on $T$. Then there is a natural morphism $\overline{T}\to T$ corresponding to the inclusion $\mathcal{O}_T\oplus \{0\}\subset \mathcal{O}_T\oplus I$, giving rise to a canonical extension of $g$ as the composition $\overline{T}\to T\to \mathcal{Y}$.  Taking this as a base point, the torsor structure in Theorem \ref{th:thdefalgstackmap} (2) induces a bijection between the set of isomorphism classes of extensions of $g$ and the group $\Ext^0(g^*L_{\mathcal{Y}},I)$. To describe this bijection explicitly, note that we have a natural isomorphism $h^0(i^*L_{\overline{T}/T})\simeq I$. Now, for a given extension $\overline{g}:\overline{T}\to \mathcal{Y}$ consider the composition $\overline{g}^*L_{\mathcal{Y}}\to L_{\overline{T}}\to L_{\overline{T}/T}$. Up to given isomorphisms, this restricts to a map $\alpha_{\overline{g}}:g^*L_{\mathcal{Y}}\to I$ on $T$. The association $\overline{g}\mapsto \alpha_{\overline{g}}$ is the bijection in question. This follows from  \cite[4.8]{Olso3} and the construction in \cite[III 2.2]{Ill}.
\end{remark}

\begin{remark}[Automorphisms]\label{rem:corrauts}
	Consider a fixed square zero extension $j:T\hookrightarrow \overline{T}$ with sheaf of ideals $I$, and an extension $\overline{g}:\overline{T}\to \mathcal{Y}$ of $g$. Let $\operatorname{Aut}(\overline{g})$ denote the automorphism group of $\overline{g}$ as an extension of $g$, i.e. the group of $2$-isomorphisms of $\overline{g}:\overline{T}\to \mathcal{Y}$ that restrict to the identity $2$-isomorphism when restricted to $T$. Let $\operatorname{Aut}_{\mathcal{Y}}(\overline{T})$ denote the group of automorphisms of $\overline{T}$ as an extension of $T$ over $\mathcal{Y}$ whose elements are pairs $(a,\phi)$, where $a:\overline{T}\to \overline{T}$ is an automorphism satisfying $a\circ j = j$, and where $\phi$ is a $2$-isomorphism $\phi:\overline{g}\circ a \Rightarrow \overline{g}$. Similarly, we let  $\operatorname{Aut}_{S}(\overline{T})$ denote the group of automorphisms of $\overline{T}$ as an extension of $T$ over $S$. 
	Then we have a natural forgetful map $\operatorname{Aut}_{\mathcal{Y}}(\overline{T})\to \operatorname{Aut}_S(\overline{T})$ whose Kernel is $\operatorname{Aut}(\overline{g})$. We have identifications $\Ext^0(L_{X/\mathcal{Y}},I)\simeq \operatorname{Aut}_{\mathcal{Y}}(\overline{T})$, which are natural in $\mathcal{Y}$ (and in particular hold for $S$ in place of $\mathcal{Y}$), and via the exact triangle $g^*L_{\mathcal{Y}/S}\to L_{T/S}\to L_{T/\mathcal{Y}}\xrightarrow{+1}$, we obtain the exact sequence 
	\[0\to \Ext^{-1}(g^*L_{\mathcal{Y}/S},I)\to \Ext^0(L_{T/\mathcal{Y}}, I)\to \Ext^0(L_{T/S},I)\]
	By what is said above, this gives an identification 
	\[\Ext^{-1}(g^*L_{\mathcal{Y}/S},I) \simeq \operatorname{Aut}(\overline{g}).\]
\end{remark}
For our purposes a different characterization of the bijection $\Ext^{-1}(g^*L_{\mathcal{Y}/S},I)\simeq \operatorname{Aut}(\overline{g})$ than the one given in Remark \ref{rem:corrauts} will be needed. For the rest of this subsection, we consider the case where $\overline{T}$ is the trivial square zero extension of $T$ by $I$ and where $\overline{g}$ is the trivial extension of $g$. 
Let $y:Y\to\mathcal{Y}$ be a smooth cover by an algebraic space and assume that $g:T\to \mathcal{Y}$ factors through $Y$ (this can always be arranged by passing to an \'etale cover of $T$, which is enough for our later application). We fix such a factorization $g_Y:T\to Y$ (with an implicit choice of $2$-isomorphism $y\circ g_Y\Rightarrow g)$. Form the Cartesian diagram 
\begin{equation*}
	\begin{tikzcd}
		Z\ar[r,"t"]\ar[d,"s"]& Y\ar[d] \\
		Y\ar[r] & \mathcal{Y}.
	\end{tikzcd}
\end{equation*}
We observe that $Y$ naturally has the structure of a groupoid algebraic space with the space of morphisms given by $Z$, and that we have a natural equivalence $[Y/Z]\xrightarrow{\sim}\mathcal{Y}$. Moreover, by definition of the $2$-cartesian product, the set of automorphisms of the morphism $g:T\to \mathcal{Y}$ is in natural bijection to the set of maps $f:T\to Z$ satisfying $s\circ f=t\circ f = g_Y$. In particular, there is a morphism $e:T\to Z$ corresponding to the identity automorphism of $g_Y$. 

Now let $\overline{g}_Y$ denote the composition $\overline{T}\to T\xrightarrow{g_Y} Y$, which is a lift of $\overline{g}$. Then we observe
\begin{lemma}\label{lem:autcorrspecial}
	We have a natural bijection betwen $\operatorname{Aut}(\overline{g})$ and the set of morphisms $\overline{f}:\overline{T}\to Z$ satisfying $s\circ \overline{f}= t\circ \overline{f} =\overline{g}_Y$ and $\overline{f}\circ j = e$.  In other words, the group of infinitesimal automorphisms of $\overline{g}$ is in bijection with the group of deformations of $e$ to $\overline{T}$ that induce the trivial deformation of $g_Y$ upon composition with either $s$ or $t$.
\end{lemma}   
As a consequence of this Lemma, we have a canonical isomorphism
\[\operatorname{Aut}(g)\simeq \Ker\left(\Hom(e^*L_Z,I)\xrightarrow{(-s^*,t^*)} \Hom(e^*s^*L_Y,I)\oplus \Hom(e^*t^*L_Y,I)\right).\] 
Let $z:=y\circ s:Z\to \mathcal{Y}$. By Lemma \ref{lem:coneprescot}, we have the natural isomorphism \[z^*L_{\mathcal{Y}}\simeq \coneop(s^*L_Y\oplus t^*L_Y\oplus L_Z)[-1].\]
Using this, we get an identification
\begin{equation}\label{eq:autcorruse}
\begin{aligned}
	\operatorname{Aut}(\overline{g})&\simeq \Ext^{-1}(e^*z^*L_{\mathcal{Y}},I)=\Ext^{-1}(g^*L_{\mathcal{Y}},I)\\
	\varphi&\mapsto \tau_{\varphi}.
\end{aligned}
\end{equation}
\subsection{Deformations of sheaves}
Let $X, T$ be schemes over a common base field $k$. Let $T\hookrightarrow \overline{T}$ be a square zero extension defined by an ideal sheaf $I$. Let also $E$ be a $T$-flat quasicoherent sheaf on $X\times T$. We consider the problem of extending $E$ to a $\overline{T}$-flat coherent sheaf on $X\times \overline{T}$. Let $\pi:X\times T\to T$ denote the projection.  By \cite[IV Proposition 3.1.8]{Ill}, we have 
\begin{theorem}
	\begin{enumerate}[label=\arabic*)]
		\item There is a natural obstruction class $\omega^{sh}(\linebreak[1]E, \overline{T})\in\linebreak[1] \Ext^2_{X\times T}(\linebreak[1]E,\pi^*I\otimes E )$ which vanishes if and only if there is an extension of $E$ to a $\overline{T}$-flat sheaf on $X\times \overline{T}$.
		\item If a $\overline{T}$-flat extension of $E$ to $X\times \overline{T}$ exists, then the set of isomorphism classes of such extensions naturally forms a torsor under $\Ext^1_{X\times T}(E,\pi^*I\otimes E)$.
		\item For a fixed $\overline{T}$-flat extension $\overline{E}$, the set of automorphisms of $\overline{E}$ which restrict to the identity on $E$ is canonically isomorphic to $\Hom_{X\times T}(E,\pi^*I\otimes E)$. 
	\end{enumerate}
\end{theorem}
We make some of the natural maps implied in this theorem explicit.
\begin{remark}[Obstructions]\label{rem:obstrshvs}
	For a given $\overline{T}$, the obstruction class $\omega^{sh}(E,\overline{T})$ is given by the composition 
	\[E\xrightarrow{\at_{E, X\times T/S}} L_{X\times T}[1]\otimes E\to L_{X\times T/X\times \overline{T}}[1]\otimes E\to \pi^*I[2]\otimes E.\]
	Here, the first map is the Atiyah class, the second map is induced from the naturality of cotangent complexes, and the last map is induced from the natural identification $\tau_{\geq 1}L_{T/\overline{T}}\simeq I[1]$.
	This is proven in \cite[IV Proposition 3.1.8]{Ill}. 
\end{remark}
\begin{remark}[Deformations]\label{rem:corrdefshvs}
	Let $\overline{T}$ be the trivial extension of $T$ by $I$. Then there is a canonical flat extension of $E$ to $X\times \overline{T}$ given by $E\oplus \pi^*I\otimes E$, with multiplication by $I$ given by $j(e,0)=(0,je)$ for local sections. The torsor structure on the space of extensions therefore gives rise to a bijection between $\Ext^1_{X\times T}(E,I\otimes E)$ and the set of extensions of $E$ to $\overline{T}$. To describe this bijection, let $\nu\in \Ext^1_{X\times T}(E,\pi^*I\otimes E)$, corresponding to an extension
	\[0\to E\otimes I \xrightarrow{\mu} \overline{E}\xrightarrow{\rho} E\to 0.\]
	We make this into an $\mathcal{O}_{X\times T}\oplus\pi^*I$-module, by defining the action of $I$ on $\overline{E}$ on local sections as $jx:=\mu(j\otimes \rho(x))$. One checks that this defines a $\overline{T}$-flat coherent sheaf on $X\times\overline{T}$ extending $E$. It is straightforward to see that this construction is invertible.  
\end{remark}

\begin{remark}[Automorphisms]\label{rem:autshvs}
	Let $\overline{T}$ be the trivial extension of $T$ by $I$, and let $\overline{E}=E\oplus \pi^*I\otimes E$ be the canonical flat extension of $E$ to $\overline{T}$. For an element $a\in \Hom_{X\times T}(E,\pi^*I\otimes E)$, the map $\psi_a:\overline{E}\to \overline{E}$ given locally by $(x_1,j\otimes x_2)\mapsto (x_1,j\otimes x_2 +\varphi(x_1))$ is an automorphism of $\overline{E}$ which restricts to the identity on $E$. This gives the claimed bijection of automorphism groups for this choice of $\overline{E}$. 
\end{remark}

\begin{remark}
	Suppose that $X$ is a smooth projective variety, and that $T$ is of finite type over $k$. Then in particular, $E$ has a finite length resolution by locally free sheaves. Then we have natural isomorphisms
	\begin{align*}
	\Ext^i_{X\times T}(E,\pi^*I\otimes E)&\simeq \Ext^i_{X\times T}(\mathcal{O}_{X\times T}, \pi^*I \otimes E\otimes E^{\vee})\\
	&\simeq \Ext^i_T(\mathcal{O}_T, I\otimes R\pi_*(E\otimes E^{\vee}))\\
	&\simeq \Ext^i_T(R\pi_*(E\otimes E^{\vee}), I)
	\end{align*}
	Here, all tensor products and duals are taken in the derived sense. The first isomorphism is due to the fact that $E$ is dualizable, the second one is push-pull adjunction and the projection formula, and the third one uses that $R\pi_*(E\otimes E^{\vee})$ is dualizable. 
\end{remark}

\subsection{Comparison of obstruction classes}\label{subsec:obstrcomp}
Now let $X,\mathcal{M}$ be as specified in the beginning of the section with universal sheaf $\mathcal{E}$ on $\mathcal{M}\times X$. Our goal is to show that the map $\At_{\mathcal{E}}:R\pi_*(\mathcal{E}\otimes \mathcal{E}^{\vee})^{\vee}[-1]\to L_{\mathcal{M}}$ is surjective on $h^{-1}$. By the arguments of \cite[\S 4]{BeFa}, it is enough to show that for every map from an affine scheme $g:T\to \mathcal{M}$ and any quasicoherent sheaf $I$ on $T$ the map 
\[(g^*\At_{\mathcal{E}})^*:\Ext^1_T(g^*L_\mathcal{M},I)\to \Ext^2_T(g^*R\pi_*(\mathcal{E}\otimes \mathcal{E}^{\vee})^{\vee}, I)\]
given by composition with the Atiyah class is injective. In fact, it is enough to show that for any such $g$ and $I$ and any square zero extension $T\hookrightarrow \overline{T}$ there exists an extension of $g$ to $\overline{T}$ if and only if the image of the obstruction class $\omega(g,\overline{T})$ under this map vanishes. By Remark \ref{rem:obstrsmor}, the obstruction class $\omega(g,\overline{T})$ is obtained as the composition of the natural maps 
\[g^*L_{\mathcal{M}}\to L_T\to L_{T/\overline{T}}\to I[1].\]
On the other hand, let $E=g^*\mathcal{E}$, so that we have a commutative diagram 
\begin{equation*}
	\begin{tikzcd}
		g^*R\pi_*(\mathcal{E}\otimes \mathcal{E}^{\vee})^{\vee}[-1]\ar[r,"g^*\At_{\mathcal{E}}"]\ar[d,"\sim"]& g^*L_{\mathcal{M}}\ar[d] \\
		R\pi_*(E\otimes E^{\vee})^{\vee}[-1] \ar[r,"\At_E"] & L_T,
	\end{tikzcd}
\end{equation*}
where the left vertical map is the canonical base change isomorphism. 
We find that the image of $\omega(g,\overline{T})$ under $(g^*\At_{\mathcal{E}})^*$ is up to the given isomorphism equal to the composition
\[R\pi_*(E\otimes E^{\vee})^{\vee}[-1]\xrightarrow{\At_E}L_T\to L_{T/\overline{T}}\to I[1].\]
By the definition of $\At_E$, this corresponds to the morphism
\[E \xrightarrow{\at_E} L_{T\times X}[1]\otimes E\to L_{T\times X/\overline{T}\times X}[1]\otimes E\to \pi^*I[2]\otimes E.\]
By Remark \ref{rem:obstrshvs}, this is exactly the obstruction to the existence of a $\overline{T}$-flat extension of $E$ to $\overline{T}\times X$, and therefore an obstruction to the existence of an extension of $\overline{g}$ by the universal property of $\mathcal{M}$. This shows what we needed. 
\subsection{Comparison of deformation spaces} \label{subsec:defcomp}
We show that the map $\At_{\mathcal{E}}:R\pi_*(\mathcal{E}\otimes \mathcal{E}^{\vee})^{\vee}[-1]\to L_{\mathcal{M}}$ is an isomorphism on $h^0$. 
By the arguments of \cite[\S 4]{BeFa}, it is enough to show that for every map from an affine scheme $g:T\to \mathcal{M}$ and any quasicoherent sheaf $I$ on $T$ the map 
\[(g^*\At_{\mathcal{E}})^*:\Ext^0_T(g^*L_\mathcal{M},I)\to \Ext^1_T(g^*R\pi_*(\mathcal{E}\otimes \mathcal{E}^{\vee})^{\vee}, I)\]
given by composition with the Atiyah class is an isomorphism. This follows from the following stronger statement:
\begin{lemma}
	Let $\overline{T}$ be the trivial extension of $T$ by $I$ and let $\overline{g}:\overline{T}\to \mathcal{M}$ be any extension of $g$. Let $\overline{E}$ be the corresponding extension of $E$. Then the class $\alpha_{\overline{g}}$ of Remark \ref{rem:corrdefs} is mapped by $(g^*\At_{\mathcal{E}})^*$ to the class of 
	\[\Ext^1_T(g^*R\pi_*(\mathcal{E}\otimes \mathcal{E}^{\vee})^{\vee}, I)\simeq \Ext^1_T(E,I\otimes E)\]
	corresponding to the extension $\overline{E}$ via Remark \ref{rem:corrdefshvs}. 
\end{lemma}
\begin{proof}
	By  Remark \ref{rem:corrdefs}, we have that $(g^*\At_{\mathcal{E}})^*\alpha_{\overline{g}}$ is equal to the composition
	\[g^*R\pi_*(\mathcal{E}\otimes \mathcal{E}^{\vee})^{\vee}[-1]\xrightarrow{g^*\At_{\mathcal{E}}}g^*L_\mathcal{M}\xrightarrow{\overline{g}^*\mid_T} j^*L_{\overline{T}}\to j^*L_{\overline{T}/T}\to I.\]
	The composition of the first two maps is just $j^*\At_{\overline{E}}$. Therefore, by definition of $\At$, this corresponds under adjunction to the morphism 
	\[E\xrightarrow{j^*\at_{\overline{E}}} L_{\overline{T}\times X/T\times X}[1]\otimes E\to \pi^*I[1]\otimes E.\]
	Thus, we are reduced to showing that this morphism agrees with the class $\beta_{\overline{E}}$ in $\Ext^1(E,\pi^*I\otimes E)$. This is shown in Lemma \ref{lem:explicitdefs} below. 
\end{proof}
\begin{lemma}\label{lem:explicitdefs}
	Let $T$ be an affine scheme and $I$ a quasicoherent sheaf on $T$ and let $j:T\to \overline{T}$ be the trivial square zero thickening of $T$ with ideal sheaf $I$. Let $E$ be a coherent sheaf on $T\times X$ and $\overline{E}$ an extension of $E$ to $\overline{T}$ such that the induced map $I\otimes_{T} E\to I\overline{E}\subset \overline{E}$ is an isomorphism (see \cite[IV 3.1]{Ill}). Let $\beta\in \Ext^1_{T\times X}(E,\pi^*I\otimes E)$ be the corresponding extension class.  Then the composition 
	\[E\xrightarrow{j^*\at_{\overline{E}}} j^*L_{\overline{T}\times X/T\times X}[1]\otimes E\to \pi^*I[1]\otimes E\]
	equals $\beta$. 
\end{lemma}
\begin{proof}
	Let $r:\overline{T}\to T$ be the projection coming from the inclusion $\mathcal{O}_T=\mathcal{O}_T\oplus \{0\}\subset\mathcal{O}_T\oplus I$. Let $R=P_{r^{-1}\mathcal{O}_T}(\mathcal{O}_{\overline{T}})$ be the standard simplicial resolution. Then we have the following commutative diagram of $R$-modules with exact rows 
	\begin{equation*}
		\begin{tikzcd}
			0\ar[r]&\Omega^1_{R/r^{-1}\mathcal{O}_T}\otimes_R \overline{E}\ar[r]\ar[d]& P^1_{R/r^{-1}\mathcal{O}_T}(\overline{E})\ar[r]\ar[d]&E \ar[r]\ar[d]&0 \\
			0\ar[r] &\Omega_{\overline{T}/T}\otimes_{\overline{T}} \overline{E} \ar[r]&P^1_{\overline{T}/T}(\overline{E})\ar[r]&E\ar[r]& 0
		\end{tikzcd}
	\end{equation*}
	The upper row is used to define the Atiyah class, so by taking connecting map in $D^{\Delta}(R)$ and passing to $D(\overline{T})$, we get the commutative diagram 
	\begin{equation*}
		\begin{tikzcd}
			\overline{E}\ar[r,"\at_{\overline{E},\overline{T}/T}"]\ar[dr]& L_{\overline{T}/T}\otimes_{\overline{T}}\overline{E}\ar[d] \\
			& \Omega_{\overline{T}/T}\otimes_{\overline{T}}\overline{E},
		\end{tikzcd}
	\end{equation*}
	where the diagonal map is the connecting map associated to the sequence of principal parts $\underline{P}_{\overline{T}/T}(\overline{E})$. This reduces us to showing that the restriction of $\underline{P}_{\overline{T}/T}(\overline{E})$ along $j$ is equal to the extension $\beta$ corresponding to $E$ via the identification $j^*\Omega^1_{\overline{T}/T}=I$.
	This follows from a straightforward calculation using that $P^1_{\overline{T}/T}(E)=(\mathcal{O}_T\oplus I)\otimes_{\mathcal{O}_T}\overline{E}$. 
\end{proof}

\subsection{Comparison of automorphism groups}\label{subsec:autcomp}
We show that the map $\At_{\mathcal{E}}:R\pi_*(\mathcal{E}\otimes \mathcal{E}^{\vee})^{\vee}[-1]\to L_{\mathcal{M}}$ induces an isomorpism on $h^{-1}$. It is enough to show that for every map from an affine scheme $g:T\to \mathcal{M}$ and any quasicoherent sheaf $I$ on $T$ the map 
\[(g^*\At_{\mathcal{E}})^*:\Ext^{-1}_T(g^*L_{\mathcal{M}},I)\to \Ext^0(g^*R\pi_*(\mathcal{E}\otimes \mathcal{E}^\vee)^{\vee},I)\]
given by composition with the Atiyah class is an isomorphism. This follows from the following more precise statement 
\begin{lemma}
	Let $\overline{T}$ be the trivial extension of $T$ by $I$ and let $\overline{g}:\overline{T}\to \mathcal{M}$ be the trivial extension of $g$ to $\overline{T}$. Let $\overline{E}$ be the corresponding extension of $E$. Let $\varphi$ be an automorphism of $\overline{g}$ extending $\id_g$, which we view as an automorphism $\gamma_{\varphi}$ of $\overline{E}$ via the universal property of $\mathcal{M}$. Then the element $\tau_{\varphi}$ corresponding to $\varphi$  via \eqref{eq:autcorruse} is mapped to the element of $\Hom_{X\times T}(E,\pi^*I\otimes E)$ corresponding to $\varphi$ via Remark \ref{rem:autshvs}. 
\end{lemma}
\begin{proof}
	We will first make two reductions:
	First, the statement can be checked \'etale locally on $T$, therefore we can choose a smooth cover $Y$ and assume that $g$ factors through $Y$, so that we are in the situation of Lemma \ref{lem:autcorrspecial}.
	Then in particular, $\varphi$ corresponds to a morphism $f:\overline{T}\to Z$ which is the image of $\tau_{\varphi}$ in $\Hom(e^*L_{Z},I) $ 
	Second, the image of $\tau_{\varphi}$ is via adjunction identified with the composition
	\[E\xrightarrow{g^*\at_{\mathcal{E}}}(g\times \id_X)^* L_{\mathcal{M}\times X/X}[1]\otimes E\xrightarrow{\pi^*\tau_\varphi\otimes E} \pi^*I\otimes E,\]
	where the first morphism is just $\At_E$, and the second morphism is the one induced by the infinitesimal automorphism $\varphi\times \id$ of the extension $\overline{g}\times \id_X$ of $g\times \id_X$. Since $\tau_{\varphi}$ factors through a map $\tau_{\varphi}':h^1(e^*L_{\mathcal{M}})\to I$, we can rewrite this composition as
	\[E\xrightarrow{(\tau_{\geq 0}\otimes \id_E)\circ g^*\at_{\mathcal{E}}} h^0((g\times \id_X)^*L_{\mathcal{M}\times X/X}[1])\otimes E\xrightarrow{g^*\tau_{\varphi}'\otimes \id_E} \pi^*I\otimes E.\]
	This reduces the problem to understanding the maps $(\tau_{\geq 0}\otimes id_E)\circ g^*\at_{\mathcal{E}}$ and $g^*\tau_{\varphi}'\otimes \id_E$. 
	By Lemma \ref{lem:atclsnake} below, the former is given as follows: For a section $m$ of $E_Z$, write $m=\sum x_i s^*n_i=\sum x_j' t^*n_j'$ using the isomorphisms $s^*E_Y\simeq E_Z\simeq t^*E_Y$. Then $e\mapsto \sum dx_i\otimes s^*e_i-\sum dx_j'\otimes \varphi(t^*e_j')$ in $\Omega_Z\otimes E_Z/(s^*(\Omega_Y\otimes_Y E_Y)+t^*(\Omega_Y\otimes_Y E_Y))$. In particular, if $m=s^*n = \sum x_j't^*n_j'$, then  $s^*y^*h^0(\at_{\mathcal{E}})(m) = -\sum dx_j'\otimes \varphi(t^*n_j')$. Pulling back along $g$ gives 
	\[ g^*(\tau_{\geq 0}\otimes \id_E \circ \at_{\mathcal{E}})(g^*m) = -\sum dx_j\otimes \varphi(g^*m).\] 

	The morphism $\tau'_{\varphi}:\operatorname{Coker}(s^*\Omega_Y\oplus t^*\Omega_Y \to \Omega_Z)\mid_T\to \Omega_{\overline{T}/T}\mid_T=I$ sends $dx$ to $df(x)$, where $df:\Omega_Z\to I$ is the derivation describing $f$ as a deformation of $e$ (s.t. $f=e+I\otimes e + df:e^{-1}\mathcal{O}_Z\to \mathcal{O}_T\oplus I$). In conclusion, we get that the composition is given by the map $g^*n\mapsto \sum -df(x_i)\otimes f^*\varphi(n_i)$ if $s^*n=\sum x_i t^*n_i$.
	On the other hand, we have the automorphism $\gamma_{\varphi}$ of $\overline{E}$ given by the composition $\overline{E}\simeq f^*t^*E_Y\xrightarrow{f^*\varphi}f^*s^*E_Y\simeq \overline{E}$. We compute 
	\[\gamma_{\varphi}(\overline{g}^*n) =f^*\varphi(\sum x_i t^*n_i)=\sum f^*x_i f^*\varphi(t^*n_i) \]
	whenever $s^*n=\sum x_i t^*n_i$. 
	Since $\varphi$ restricts to the identity automorphism on $T$, we know that $\gamma_{\varphi}$ is of the form $\id_{\overline{E}}+\rho$, where $\rho:E\to I\otimes E$ is a morphism. We claim that $\rho(g^*n) = dx_i\otimes f^*\varphi(t^*n_i)$ whenever $s^*n=\sum x_i t^*n_i$. But we have: 
	\[\rho(g^*n) = (f-e)^*\varphi(\sum x_i\otimes t^*n_i) =(0, \sum dfx_i \otimes \varphi(t^*n_i)).\]
\end{proof}
\begin{lemma}\label{lem:atclsnake}
	In the situation of Lemma \ref{lem:autcorrspecial}, suppose that $E\in D^{\leq 0}_{qcoh}(\mathcal{Y})$. Then the composition $s^*y^*E\xrightarrow{s^*y^*\at_E} s^*y^*L_{\mathcal{Y}}[1]\otimes s^*y^*E\to s^*y^*h^{1}(L_{\mathcal{Y}})\otimes s^*y^*E$ agrees up to natural isomorphisms with the connecting homomorphism $\delta $, obtained by applying the Snake Lemma to the following diagram
	\begin{equation*}
		\begin{tikzcd}[column sep = small]
			0\ar[r]& \substack{s^*(\Omega^1_Y\otimes E_Y)\oplus\\(t^*\Omega^1_Y\otimes E_Y)} \ar[r]\ar[d]& s^*P_{Y}^1(E_Y)\oplus t^*P_Y^1(E_Y)\ar[r]\ar[d] & s^*E_Y\oplus t^*E_Y \ar[d] \ar[r] &0\\
			0\ar[r]& \Omega_Z\otimes E_Z\ar[r]& P_Z^1(E_Z)\ar[r]& E_Z \ar[r] & 0
		\end{tikzcd}
	\end{equation*}
\end{lemma}
\begin{proof}
	Let $R:=P_{k}(\mathcal{O}_Z)$ and $Q=P_{k}(\mathcal{O}_Y)$ be the standard simplicial resolutions and $s_R^*$ and $t_R^*$ the pullback functors from $Q$-modules to $R$-modules induced by $s$ and $t$ respectively. Then, we have a commutative diagram of $R$-modules
	\begin{equation*}
		\begin{tikzcd}[column sep = small]
			0\ar[r]& \substack{s_R^*(L_Y\otimes E_Y)\oplus\\t_R^*(L_Y\otimes E_Y) }\ar[r]\ar[d,"\alpha"]& s_R^*P_{Q}^1(E_Y)\oplus t_R^*P_Q^1(E)\ar[r]\ar[d,"\beta"] & s^*E_Y\oplus t^*E_Y \ar[d,"\gamma"] \ar[r] &0\\
			0\ar[r]& L_Z\otimes E_Z\ar[r]\ar[d]& P_R^1(E_Z)\ar[r]\ar[d]& E_Z \ar[r]\ar[d] & 0 \\
			0\ar[r]& C(\alpha)\ar[r] & C(\beta)\ar[r] &C(\gamma) \ar[r]& 0.
		\end{tikzcd}
	\end{equation*}
	 By the construction of the Atiyah class, we have natural isomorphisms $C(\alpha)\simeq s^*y^*L_{\mathcal{Y}}[1]\otimes E$ and $C(\gamma)\simeq s^*y^*E[1]$ with respect to which the connecting map associated to the lower row is naturally identified with $s^*y^*\at_E[1]$ in $D^{\Delta}(R)$. The diagram maps to the similar diagram of $\mathcal{O}_Z$-modules
	\begin{equation*}
		\begin{tikzcd}[column sep = small]
			0\ar[r]& \substack{s^*(\Omega_Y\otimes E_Y)\oplus \\ t^*(\Omega_Y\otimes E_Y)} \ar[r]\ar[d,"\alpha'"]& s^*P_{Y}^1(E_Y)\oplus t^*P_Y^1(E_Y)\ar[r]\ar[d,"\beta'"] & s^*E_Y\oplus t^*E_Y \ar[d,"\gamma'"] \ar[r] &0\\
			0\ar[r]& \Omega_Z\otimes E_Z\ar[r]\ar[d]& P_Z^1(E_Z)\ar[r]\ar[d]& E_Z \ar[r]\ar[d] & 0 \\
			0\ar[r]& \dfrac{\Omega_Z}{s^*\Omega_Y\oplus t^*\Omega_Y}\otimes E_Z\ar[r] & *\ar[r] &C(\gamma') \ar[r]& 0,
		\end{tikzcd}
	\end{equation*}
	where the last row is obtained by taking cones of the vertical morphisms and then pushing out along $C(\alpha')\to \Omega_Z/(s^*\Omega_Y\oplus t^*\Omega_Y) \otimes E_Z$. The morphism $C(\alpha)\to \Omega_Z/(s^*\Omega_Y\oplus t^*\Omega_Y)\otimes E_Z$ here is identified with the truncation morphism $L_{\mathcal{Y}}[1]\to h^1(L_{\mathcal{Y}})$ tensored with $E_Z$.
	By pulling back the lower row along the natural map $s^*y^*E[1]\to C(\gamma')$ obtained from the diagonal map $s^*y^*E\to s^*E_Y\oplus t^*E_Y$, we get an extension
	\[0\to \frac{\Omega_Z}{s^*\Omega_Y+ t^*\Omega_Y }\otimes E_Z \to * \to s^*y^*E[1]\to 0.\]
	That the induced connecting homomorphism agrees with the one coming from the Snake Lemma follows from Lemma \ref{lem:snakelemmacomp}.
\end{proof}
\begin{lemma}\label{lem:snakelemmacomp}
	Suppose that we have a commutative diagram in an abelian category with exact rows 
	\begin{equation*}
		\begin{tikzcd}
			0\ar[r]&A\ar[r]\ar[d,"\alpha"]& B\ar[r]\ar[d,"\beta"]&C\ar[r]\ar[d,"\gamma"]& 0 \\
			0 \ar[r] & D\ar[r]& E\ar[r]&F\ar[r]& 0. 
		\end{tikzcd}
	\end{equation*}
	Form the exact sequence 
	\[\underline{C}: 0\to C(\alpha)\to C(\beta)\to C(\gamma)\to 0\]
	and let $j:C(\alpha)\to \operatorname{Coker}(\alpha)$ and $i:\Ker(\gamma)[1]\to C(\gamma)$ denote the natural maps. Then the connecting homomorphism associated to the sequence of complexes of $A$-modules
	\[j_*i^*\underline{C}=i^*j_*\underline{C}:0\to \operatorname{Coker}(\alpha) \to M \to \Ker(\gamma)[1]\to 0\]
	is exactly \emph{minus} the shift by $1$ of the connecting morphism $\delta:\Ker(\gamma)\to \operatorname{Coker}(\alpha)$ in the Snake Lemma.  
\end{lemma}
\begin{proof}
	One checks by a direct computation that $M$ is just the cone over $\delta$ with the two maps being the canonical inclusion and projection. It follows by the usual considerations of mapping cones that the connecting morphism of the sequence $j_*i^*\underline{C}$ is $-\delta$. 
\end{proof}

\begin{remark}\label{rem:complexes}
	To conclude, we point out that the conclusions of \S \ref{subsec:obstrcomp} -- \S \ref{subsec:autcomp} still hold when $\mathcal{M}$ is open in a moduli stack of perfect complexes, and $\mathcal{E}$ is the universal complex, using the results of \cite{HuTho}. 
	We assume here that the complexes being parametrized have vanishing negative $\Ext$-groups universally on $\mathcal{M}$.
	
	The arguments of \ref{subsec:obstrcomp} and \ref{subsec:defcomp} go through with Theorem 3.3 and Corollary 3.4 of \cite{HuTho} in place of Remark \ref{rem:obstrshvs} and Remark \ref{rem:corrdefshvs} respectively. The arguments of \ref{subsec:autcomp} go through unchanged for complexes.
\end{remark}